\documentclass[11pt]{elsarticle}

\usepackage{amsmath,amssymb,amsthm,mathtools,longtable,booktabs,array}
\usepackage{lmodern}
\usepackage[a4paper,left=3cm,right=3cm,top=3.5cm,bottom=3.5cm]{geometry}
\usepackage{hyperref}
\newtheorem{thm}{Theorem}[section]
\newtheorem{lem}[thm]{Lemma}
\newtheorem{pro}[thm]{Proposition}
\newtheorem{cor}[thm]{Corollary}

\numberwithin{equation}{section}

\begin{document}

\begin{frontmatter}
\title{A sharp lower bound for $h$-fold restricted sumsets of integers\tnoteref{mytitlenote}}
\tnotetext[mytitlenote]{This work was supported by the National Natural Science Foundation of China (Grant No. 12101007) and the Natural Science Foundation of Anhui Province (Grant No. 2008085QA06).}
\author[mymainaddress]{Yujie Wang\corref{mycorrespondingauthor}}
\ead{wangyujie9291@126.com}
\cortext[mycorrespondingauthor]{Corresponding author}
\address[mymainaddress]{School of Mathematics and Statistics, Anhui Normal University, Wuhu 241002, P. R. China}
\begin{abstract}
Let $A\subseteq[0, L]$ be a set of $k\geqslant 9$ integers such that $0, L\in A$ and $\gcd(A)=1$.
Define the $h$-fold restricted sumset
$$h^{\wedge}A=\{a_1+\cdots+a_h: a_i\in A,\ 1\leqslant i\leqslant h, \ \ a_i\neq a_j\ {\rm for}\ i\neq j\}. $$
In this paper, we prove that, for every $1\leqslant h\leqslant k-1$, we have
$$|h^{\wedge}A|\geqslant h(k-h)-k+2+\min\{L, h(k-h)-\mathbf{1}_h\}, $$
where $\mathbf{1}_h=1$ for $h\in\{2, k-2\}$ and $\mathbf{1}_h=0$ otherwise.
The bound is sharp for every $1\leqslant h\leqslant k-1$ and the cardinality threshold $k\geqslant 9$ cannot be decreased.

We also give a lower bound for the union of two consecutive restricted sumsets in a cyclic group and characterize equality.
\end{abstract}
\begin{keyword}
restricted sumsets; Freiman--Lev conjecture; cyclic groups
\MSC[2020] 11B13
\end{keyword}
\end{frontmatter}

\section{Introduction}
Let $A$ be a finite nonempty set of integers and let $l(A)$ denote the difference of the largest and the smallest elements of $A$.
For any positive integer $h$, define
the $h$-fold sumset
$$hA=\{a_1+\cdots+a_h:a_i\in A,\ 1\leqslant i\leqslant h\}$$
and the $h$-fold restricted sumset
$$h^{\wedge}A=\{a_1+\cdots+a_h: a_i\in A,\ 1\leqslant i\leqslant h, \ \ a_i\neq a_j\ {\rm for}\ i\neq j\}. $$
Here, $h^{\wedge}A=\emptyset$ if $|A|<h$ and $0^{\wedge}A=\{0\}$.

For an integer set $X$ with at least two elements, write
$$d(X)=\gcd(X-X). $$
We call $X$ {\it primitive} if $d(X)=1$ and {\it normalized} if in addition $\min X=0$.

For integers $a, b$, write
$$[a, b]=\{x\in\mathbb Z: a\leqslant x\leqslant b\}, $$
thus, $[a, b]=\emptyset$ if $a>b$. For a real number $x$, write $x_+=\max\{0, x\}$.
For a dilation we use
$$d\cdot X=\{dx: x\in X\}. $$

Sumsets and restricted sumsets are central objects of study in additive number theory.
The following standard result is given in Nathanson \cite[Theorems 1.3 and 1.6]{Nathanson1996}.
\begin{thm}
Let $A$ be a set of $k\geqslant 3$ integers and $h\geqslant 2$. Then
$$|hA|\geqslant hk-h+1. $$
Moreover, equality holds if and only if $A$ is a $k$-term arithmetic progression.
\end{thm}

For restricted sumsets, Nathanson \cite[Theorems 1 and 2]{Nathanson1} proved the following result.

\begin{thm}\label{thm:nathanson-restricted-minimum}
Let $A$ be a set of $k$ integers and $0\leqslant h\leqslant k$. Then
$$|h^{\wedge}A|\geqslant h(k-h)+1.$$
Moreover, if $k\geqslant 5$ and $2\leqslant h\leqslant k-2$, then equality holds if and only if $A$ is a $k$-term arithmetic progression.
\end{thm}

Besides the cardinality $|A|$, the diameter $l(A)$ provides further information about the cardinalities of sumsets and restricted sumsets.
In 1959, Freiman \cite{Freiman} proved that

\begin{thm}
Let $A\subseteq [0, l]$ be a set of $k\geqslant 3$ integers such that $0, l\in A$, $\gcd(A)=1$. Then
$$|2A|\geqslant \left\{\begin{array}{ll}
l+k &\text{ if } l\leqslant 2k-3, \\
3k-3 &\text{ if } l\geqslant 2k-2. \end{array}\right.$$
\end{thm}

In 1996, Lev \cite[Theorem 1]{Lev} gave the following result.

\begin{thm}
Let $h\geqslant 2$ and $k\geqslant 2$ be integers.
Let $A\subseteq [0, l]$ be a set of $k$ integers such that $0, l\in A$, $\gcd(A)=1$.
Then
$$|hA|\geqslant |(h-1)A|+\min\{l, h(k-2)+1\}. $$
\end{thm}

The first case of the following theorem is due to Lev \cite[Theorem 1]{Lev2}.
For the second case, the range $2k-4\leqslant l\leqslant 2k-3$ follows from Wang and Tang \cite[Theorem 1.5]{Wang}.
When $l\geqslant2k-2$ and the second largest element of $A$ is less than $2k-4$, the result follows from Wang and Tang \cite[Theorem 1.1]{WangRestricted}.
The remaining range, where $l\geqslant2k-2$ and the second largest element is at least $2k-4$, follows from \cite[Theorem 1.1]{Wang2026}.

\begin{thm}\label{thm:freiman-lev-twofold}
Let $A\subseteq [0, l]$ be a set of $k>7$ integers such that $0, l\in A$, $\gcd(A)=1$. Then
$$|2^{\wedge}A|\geqslant \left\{\begin{array}{ll}
l+k-2 &\text{ if } l\leqslant 2k-5, \\
3k-7 &\text{ if } l\geqslant 2k-4. \end{array}\right.$$
\end{thm}

For three-fold restricted addition in finite abelian groups, Lev \cite{LevThree} gave conditions on the density of a set and the number of elements of order two in the group under which its restricted sumset is the whole group, apart from specified coset configurations.
For integer sets and bounds involving the diameter, Tang and Wang \cite[Theorem 3.1]{TW} proved the following result.

\begin{thm}\label{T6}
Let $A\subseteq [0, l]$ be a set of $k\geqslant 5$ integers such that $0, l\in A$.
If $l\leqslant 2k-5$, then
$$|3^{\wedge}A|\geqslant 2k+l-7. $$
\end{thm}

Tang and Wang \cite[Remark 4.1]{TW} also proved the following result
for higher orders.

\begin{thm}\label{thm:tang-wang-higher}
Let $A$ be a set of $k\geqslant5$ integers and let $h\geqslant3$ be an integer.
If $l(A)\leqslant2k-2h+1$, then
$$|h^{\wedge}A|\geqslant l(A)+h(k-h)-k+2.$$
\end{thm}

In this paper, we obtain the following result.

\begin{thm}\label{thm:main}
Let $A\subseteq[0, L]$ be a set of $k\geqslant9$ integers with $0, L\in A$ and $\gcd(A)=1$.
For every $1\leqslant h\leqslant k-1$, we have
$$|h^{\wedge}A|\geqslant h(k-h)-k+2+\min\{L, h(k-h)-\mathbf{1}_h\}, $$
where $\mathbf{1}_h=1$ for $h\in\{2, k-2\}$ and $\mathbf{1}_h=0$ otherwise.
\end{thm}

The bound is sharp for every $1\leqslant h\leqslant k-1$ and every diameter $L\geqslant k-1$.
For $3\leqslant h\leqslant k-3$, let
$$A=[0, k-2]\cup\{L\}, \ \ L\geqslant k-1, \ \ 3\leqslant h\leqslant k-3. $$
Then
$$|h^{\wedge}A|=h(k-h)-k+2+\min\{L, h(k-h)\}. $$
For $h=2$ and $k-1\leqslant L\leqslant2k-5$, the same set satisfies
$$2^{\wedge}A=[1,2k-5]\cup[L,L+k-2]=[1,L+k-2].$$
For $h=2$ and $L\geqslant2k-5$, take $A=[0,k-3]\cup\{L-1,L\}$. Then
$$2^{\wedge}A=[1,2k-7]\cup[L-1,L+k-3]\cup\{2L-1\},$$
where the three sets are disjoint, so $|2^{\wedge}A|=3k-7$.
Complementation gives equality for $h=k-2$, while equality for $h=1,k-1$ holds for every admissible $A$.

The lower threshold in Theorem \ref{thm:main} cannot be replaced by $k\geqslant 8$.
Indeed, put $A=[0, 3]\cup[13, 16]$. Then
$$4^{\wedge}A=\{6\}\cup[16, 22]\cup[28, 36]\cup[42, 48]\cup\{58\}, $$
thus, $|4^{\wedge}A|=25$. For $k=8$, $h=4$ and $L=16$, the right-hand side of the asserted bound is $26$.

For $B\subseteq\mathbb Z_N$ and $0\leqslant m\leqslant|B|-1$, put
$$U_m(B)=m^{\wedge}B\cup(m+1)^{\wedge}B.$$
For a finite set $X$ with $c=|X|$ and $0\leqslant j\leqslant c-1$, complementation gives
$$U_j(X)=\sum_{x\in X}x-U_{c-j-1}(X).$$

The key tools is the following direct and inverse theorems in cyclic groups.

\begin{thm}\label{thm:r4-coarse-cyclic-all-orders}
Let $0\in B\subseteq\mathbb Z_N$ be a set of $b\geqslant 6$ such that $\langle B\rangle=\mathbb Z_N$. Then
$$|U_m(B)|\geqslant\min\{N, 2b-2\}, \ \ 2\leqslant m\leqslant \left\lfloor\frac{b-1}{2}\right\rfloor. $$
\end{thm}

\begin{thm}\label{thm:r4-coarse-cyclic-inverse}
Let $0\in B\subseteq\mathbb Z_N$ be a set of $b\geqslant 7$ such that $\langle B\rangle=\mathbb Z_N$ and $N\geqslant 2b$.
For every $2\leqslant m\leqslant\left\lfloor\frac{b-1}{2}\right\rfloor$, we have
$$|U_m(B)|=2b-2$$
if and only if $B=H\cup\{x\}$ for a subgroup $H$ of order $b-1$.
\end{thm}

\section{Preliminaries}
Throughout this paper, for an integer set $X$ with $|X|=n$ and $1\leqslant r\leqslant n$, put
$$
\begin{gathered}
e_r(X)=l(r^{\wedge}X)+1-|r^{\wedge}X|,\\
D_r(X)=2|r^{\wedge}X|-l(r^{\wedge}X)-r(n-r)-2,\\
a_r(X)=2r(n-r)-n+1-|r^{\wedge}X|.
\end{gathered}
$$
Then
$$e_r(X)+D_r(X)=|r^{\wedge}X|-r(n-r)-1, $$
$$l(r^{\wedge}X)=r(n-r)+2e_r(X)+D_r(X), $$
$$a_r(X)=r(n-r)-n-e_r(X)-D_r(X).$$
We call $X$ {\it critical} at order $r$ if $a_r(X)\geqslant0$, equivalently
$$|r^{\wedge}X|<2r(n-r)-n+2 \text{ or } e_r(X)+D_r(X)\leqslant r(n-r)-n.$$
It is {\it noncritical} at order $r$ if $a_r(X)<0$.

For every nonempty subset $X$ of an abelian group $G$, let
$$\operatorname{Stab}(X)=\{g\in G: g+X=X\}$$
be the stabilizer of $X$. We call $X$ \emph{periodic} if $\operatorname{Stab}(X)\neq\{0\}$ and \emph{aperiodic} otherwise.

\begin{lem}[Kneser's theorem, \cite{Kneser}]\label{thm:Kneser}
Let $X, Y$ be finite nonempty subsets of an abelian group $G$ and $H=\operatorname{Stab}(X+Y)$. Then
$$|X+Y|\geqslant |X+H|+|Y+H|-|H|.$$
\end{lem}

\begin{lem}[Theorem 4.1, \cite{DSH1994}]\label{lem:group-restricted}
Let $p$ be prime and $B\subseteq\mathbb Z_p$ such that $|B|=n$. For every $0\leqslant j\leqslant n$, we have
$$|j^{\wedge}B|\geqslant\min\{p, j(n-j)+1\}. $$
\end{lem}

\begin{lem}[Cyclic specialization of Theorem 1.1, \cite{GGH}]
\label{cyc:lem:strict-cyclic-layer}
Let $A\subseteq \mathbb{Z}_{N}$ be a set of $n\geqslant 2$ elements. Then
$$|j^{\wedge}A|\geqslant n, \ \ 1\leqslant j\leqslant n-1. $$
If $n\geqslant 5$, $2\leqslant j\leqslant n-2$ and $A$ is not a subgroup coset or $n\geqslant 4$, $2\leqslant j\leqslant n-2$ and $A$ is aperiodic, then
$$|j^{\wedge}A|\geqslant n+1. $$
In particular, if $A$ is a subgroup of $\mathbb{Z}_{N}$ of order $h$, then $j^{\wedge}A=A$ for all $1\leqslant j\leqslant h-1$.
\end{lem}

\begin{proof}
For $j=1$ or $j=n-1$, we have
$$1^{\wedge}A=A, \ \ (n-1)^{\wedge}A=\sum_{x\in A}x-A, \ \ |j^{\wedge}A|=n.$$
Suppose that $2\leqslant j\leqslant n-2$. Then $n\geqslant4$.
For every elementary $2$-subgroup $H$ of $\mathbb Z_N$ and $x\in\mathbb Z_N$, we have
$$|H|\leqslant2<n, \ \ A\neq x+H, \ \ A\neq(x+H)\setminus\{y\}, \ \ y\in x+H.$$
By \cite[Theorem 1.1]{GGH}, we have $|j^{\wedge}A|\geqslant n$.
If equality holds, then either $A=x+H$ for a subgroup $H$, or
$$n=4, \ \ A=(x+H)\cup(y+H), \ \ |H|=2, \ \ x+H\neq y+H.$$
If $n\geqslant5$ and $A$ is not a subgroup coset, neither case holds.
In both cases, we have
$$|H|\geqslant2, \ \ A+H=A, \ \ H\subseteq\operatorname{Stab}(A),$$
so neither case holds when $n\geqslant4$ and $A$ is aperiodic.
Thus, under either hypothesis, $|j^{\wedge}A|\geqslant n+1$.

If $A$ is a subgroup of order $h$, then, for $1\leqslant j\leqslant h-1$, we have
$$j^{\wedge}A\subseteq A, \ \ |j^{\wedge}A|\geqslant|A|,$$
thus, $j^{\wedge}A=A$.

This completes the proof of Lemma \ref{cyc:lem:strict-cyclic-layer}.
\end{proof}

\begin{lem}[Olson \cite{Olson}; Lemma 2.3 and (2), \cite{GaoHamidoune}]\label{lem:olson-boundary}
Let $Y$ be a nonempty proper subset of a finite abelian group $G$ and let $D\subseteq G\setminus\{0\}$ generate $G$.
Then there exists $d\in D$ such that
$$|(Y+d)\setminus Y|\geqslant\min\left\{\frac{|Y|+1}{2}, \frac{|G\setminus Y|+1}{2}, \frac{|D\cup(-D)|+2}{4}\right\}.$$
\end{lem}

For every positive integer $N$, let $\pi_N: \mathbb Z\longrightarrow\mathbb Z_N$ be reduction modulo $N$.
Let $h$ be a positive integer and $X\subseteq\mathbb Z$. If $\pi_N$ is injective on $X$, then
$$h^{\wedge}\pi_N(X)=\pi_N(h^{\wedge}X). $$

\begin{lem}\label{lem:endpoint-lifting}
Let $L\geqslant2$. Let $J\subseteq[1,L-1]$ be a set of $N$ integers
and $1\leqslant p\leqslant N-1$. Put
$$E=\{0,L\}\cup J, \ \ C=\{0\}\cup\pi_L(J).$$
Then
$$|(p+1)^{\wedge}E|\geqslant|p^{\wedge}J|+|U_p(C)|, \ \ l((p+1)^{\wedge}E)=L+l(p^{\wedge}J). $$
Moreover, $|C|=N+1$, $C$ generates $\mathbb Z_L$ whenever $E$ is primitive and
$$\pi_L((p+1)^{\wedge}E)=U_p(C)=\pi_L((p-1)^{\wedge}J\cup p^{\wedge}J\cup(p+1)^{\wedge}J).$$
\end{lem}

\begin{proof}
Since $\pi_L$ is injective on $\{0\}\cup J$, we have
$$|C|=N+1.$$
If $E$ is primitive, then
$$\gcd(\{L\}\cup J)=1, \ \ \langle C\rangle=\mathbb Z_L.$$
Since $E=\{0,L\}\cup J$ and $J\subseteq[1,L-1]$, we have
$$
(p+1)^{\wedge}E
=(p+1)^{\wedge}J\cup p^{\wedge}J
\cup(L+p^{\wedge}J)\cup(L+(p-1)^{\wedge}J).
$$
Since $0\notin\pi_L(J)$, for $j\in\{p,p+1\}$, we have
$$j^{\wedge}C=j^{\wedge}\pi_L(J)\cup(0+(j-1)^{\wedge}\pi_L(J))=\pi_L((j-1)^{\wedge}J\cup j^{\wedge}J). $$
Thus,
$$\pi_L((p+1)^{\wedge}E)=\pi_L((p-1)^{\wedge}J\cup p^{\wedge}J\cup(p+1)^{\wedge}J)=U_p(C). $$

For every $c\in\pi_L((p+1)^{\wedge}E)$, put
$$s_c=\max\bigl((p+1)^{\wedge}E\cap\pi_L^{-1}(c)\bigr).$$
If $s_c\in p^{\wedge}J$, then
$$
s_c+L\in(L+p^{\wedge}J)\cap\pi_L^{-1}(c)
\subseteq(p+1)^{\wedge}E\cap\pi_L^{-1}(c),
\ \ s_c+L>s_c,
$$
a contradiction. Thus,
$$
\pi_L\bigl((p+1)^{\wedge}E\setminus p^{\wedge}J\bigr)
=\pi_L((p+1)^{\wedge}E)
=p^{\wedge}C\cup(p+1)^{\wedge}C.
$$
Since $p^{\wedge}J\subseteq(p+1)^{\wedge}E$, we have
$$
\begin{aligned}
|(p+1)^{\wedge}E|
&=|p^{\wedge}J|
+\left|(p+1)^{\wedge}E\setminus p^{\wedge}J\right|\\
&\geqslant|p^{\wedge}J|
+\left|\pi_L\bigl((p+1)^{\wedge}E\setminus p^{\wedge}J\bigr)\right|\\
&=|p^{\wedge}J|+|U_p(C)|.
\end{aligned}
$$

Since $0<\min J\leqslant\max J<L$, we have
$$
\min (p+1)^{\wedge}E=\min p^{\wedge}J,
\ \ \max (p+1)^{\wedge}E=L+\max p^{\wedge}J.
$$
Thus,
$$
l((p+1)^{\wedge}E)
=L+\max p^{\wedge}J-\min p^{\wedge}J
=L+l(p^{\wedge}J).
$$

This completes the proof of Lemma \ref{lem:endpoint-lifting}.
\end{proof}

\begin{cor}\label{lem:exact-endpoint-transfer}
Under the hypotheses and notation of Lemma \ref{lem:endpoint-lifting}, put
$$v=|U_p(C)|, \ \ \rho=|(p+1)^{\wedge}E|-|p^{\wedge}J|-v. $$
Then $\rho\geqslant0$ and
$$
\begin{aligned}
D_{p+1}(E)&=D_p(J)+2v+2\rho-L-N-1,\\
e_{p+1}(E)&=e_p(J)+L-v-\rho.
\end{aligned}
$$
\end{cor}

\begin{proof}
By Lemma \ref{lem:endpoint-lifting}, we have
$$\rho\geqslant0, \ \ l((p+1)^{\wedge}E)=L+l(p^{\wedge}J).$$
Since $|E|=N+2$ and
$$(p+1)(N-p+1)=p(N-p)+N+1,$$
by the definitions of $D_{p+1}(E)$ and $e_{p+1}(E)$, we have
$$
\begin{aligned}
D_{p+1}(E)
&=2(|p^{\wedge}J|+v+\rho)-L-l(p^{\wedge}J)
-(p+1)(N-p+1)-2\\
&=D_p(J)+2v+2\rho-L-N-1,\\
e_{p+1}(E)
&=L+l(p^{\wedge}J)+1-|p^{\wedge}J|-v-\rho\\
&=e_p(J)+L-v-\rho.
\end{aligned}
$$

This completes the proof of Corollary \ref{lem:exact-endpoint-transfer}.
\end{proof}

For every finite subset $X$ of an abelian group $G$ and $z\in G$, let
$$
\mu_X(z)=
\left|\left\{\{x,y\}\subseteq X:x\neq y,\ x+y=z\right\}\right|.
$$
Put
$$\mathcal P(X)=\{z\in G:\mu_X(z)\geqslant2\}.$$

\begin{lem}\label{lem:disjoint-pairs}
Let $B$ be a finite subset of an abelian group $G$ and $z\in G$. Then
$$
\mu_B(z)\leqslant\left\lfloor\frac{|B|}{2}\right\rfloor,
\ \ \sum_{z\in G}\mu_B(z)=\binom{|B|}{2}.
$$
Moreover,
$$B+\mathcal P(B)\subseteq3^{\wedge}B.$$
\end{lem}

\begin{proof}
Let $A_1,A_2\subseteq B$ satisfy $|A_1|=|A_2|=2$ and
$$\sum_{x\in A_1}x=\sum_{x\in A_2}x=z.$$
If $b\in A_1\cap A_2$, then
$$A_1=A_2=\{b,z-b\}.$$
Thus, distinct pairs with sum $z$ are disjoint, and
$$
2\mu_B(z)
=\left|\bigcup_{\substack{A\subseteq B,\ |A|=2\\
\sum_{x\in A}x=z}}A\right|
\leqslant|B|.
$$
Hence,
$$\mu_B(z)\leqslant\left\lfloor\frac{|B|}{2}\right\rfloor.$$
Since
$$
\{A\subseteq B:|A|=2\}
=\bigsqcup_{z\in2^{\wedge}B}
\left\{A\subseteq B:|A|=2,\ \sum_{x\in A}x=z\right\},
$$
we have
$$
\sum_{z\in G}\mu_B(z)
=\sum_{z\in2^{\wedge}B}\mu_B(z)
=\binom{|B|}{2}.
$$

For every $b\in B$ and $z\in\mathcal P(B)$, choose
$A_1,A_2\subseteq B$ such that
$$
|A_1|=|A_2|=2,\ \ A_1\cap A_2=\emptyset,\ \
\sum_{x\in A_1}x=\sum_{x\in A_2}x=z.
$$
Since
$$|\{i\in\{1,2\}:b\in A_i\}|\leqslant1,$$
there exists $i\in\{1,2\}$ such that
$$A_i\subseteq B\setminus\{b\}, \ \
z=\sum_{x\in A_i}x\in2^{\wedge}(B\setminus\{b\}).$$
Thus,
$$\mathcal P(B)\subseteq2^{\wedge}(B\setminus\{b\}), \ \ b\in B,$$
and
$$
\begin{aligned}
B+\mathcal P(B)
&=\bigcup_{b\in B}\bigl(b+\mathcal P(B)\bigr)\\
&\subseteq\bigcup_{b\in B}
\bigl(b+2^{\wedge}(B\setminus\{b\})\bigr)
=3^{\wedge}B.
\end{aligned}
$$

This completes the proof of Lemma \ref{lem:disjoint-pairs}.
\end{proof}

For integers $a$, $n\geqslant1$ and $0\leqslant r\leqslant n$, we have
\begin{equation}\label{eq:unit-interval-restricted-ranks}
r^{\wedge}[a,a+n-1]
=\left[ra+\frac{r(r-1)}2,\ ra+\frac{r(2n-r-1)}2\right].
\end{equation}

\begin{pro}\label{pro:one-residue-class}
Let $A\subseteq[0,L]$ be a set of $k\geqslant7$ integers such that
$0,L\in A$ and $\gcd(A)=1$.
If $I=A\setminus\{0,L\}$ satisfies $d(I)>1$, then for every
$3\leqslant h\leqslant k-3$,
$$|h^{\wedge}A|\geqslant2h(k-h)-k+1.$$
Equality can hold only if $k=2h$, $h\geqslant4$ and $L=4h-4$.
Moreover,
$$
|h^{\wedge}A|
\geqslant h(k-h)-k+2+\min\{L,h(k-h)\}.
$$
\end{pro}

\begin{proof}
Since
$$h^{\wedge}A=\sum_{a\in A}a-(k-h)^{\wedge}A,$$
we may assume that $3\leqslant h\leqslant\frac{k}{2}$.
Suppose that
$$|h^{\wedge}A|\leqslant2h(k-h)-k+1.$$
Put $d=d(I)$ and choose $a\in I$. Then
$$d\geqslant2, \ \ I\subseteq a+d\mathbb Z, \ \
d(A)=\gcd(A)=1.$$

If $a\equiv0\pmod d$ or $a\equiv L\pmod d$, put
$$C=A\cap(a+d\mathbb Z).$$
Since $\gcd(A)=1$ and $d\geqslant2$, we have
$$|\{0,L\}\cap(a+d\mathbb Z)|=1, \ \ |C|=k-1.$$
Write $A=C\mathbin{\dot\cup}\{x\}$. Then
$$
h^{\wedge}C\subseteq ha+d\mathbb Z,\ \
x+(h-1)^{\wedge}C\subseteq x+(h-1)a+d\mathbb Z,
\ \ x\not\equiv a\pmod d.
$$
Thus,
$$
h^{\wedge}A
=h^{\wedge}C\mathbin{\dot\cup}
\bigl(x+(h-1)^{\wedge}C\bigr).
$$
By Theorem \ref{thm:nathanson-restricted-minimum}, we have
$$
\begin{aligned}
|h^{\wedge}A|
&=|h^{\wedge}C|+|(h-1)^{\wedge}C|\\
&\geqslant h(k-h-1)+1+(h-1)(k-h)+1\\
&=2h(k-h)-k+2,
\end{aligned}
$$
a contradiction. Thus,
$$a\not\equiv0,L\pmod d.$$

If $L\not\equiv0\pmod d$, then $ha$, $(h-1)a$ and
$L+(h-1)a$ are pairwise distinct modulo $d$.
Thus,
$$
h^{\wedge}I\mathbin{\dot\cup}(h-1)^{\wedge}I
\mathbin{\dot\cup}\bigl(L+(h-1)^{\wedge}I\bigr)
\subseteq h^{\wedge}A.
$$
By Theorem \ref{thm:nathanson-restricted-minimum}, we have
$$
\begin{aligned}
|h^{\wedge}A|
&\geqslant|h^{\wedge}I|+2|(h-1)^{\wedge}I|\\
&\geqslant h(k-h-2)+1
+2\bigl((h-1)(k-h-1)+1\bigr)\\
&=2h(k-h)-k+2+(h-1)(k-h-3)\\
&\geqslant2h(k-h)-k+2,
\end{aligned}
$$
a contradiction. Thus, $d\mid L$.
Since $d=d(I)$ and $a\in I$, we have
$$\gcd(I)=\gcd(a,d),$$
and hence
$$\gcd(a,d)=\gcd(L,a,d)=\gcd(A)=1.$$

Put $C=\pi_L(I)$.
For every $c\in\pi_L((h-1)^{\wedge}I)$, put
$$t_c=\max\bigl((h-1)^{\wedge}I\cap\pi_L^{-1}(c)\bigr).$$
Since $t_c+L>t_c$ and $\pi_L(t_c+L)=c$, we have
$$
t_c+L\in
\bigl(L+(h-1)^{\wedge}I\bigr)\setminus(h-1)^{\wedge}I.
$$
Thus,
$$
\pi_L\bigl((L+(h-1)^{\wedge}I)\setminus(h-1)^{\wedge}I\bigr)
=\pi_L((h-1)^{\wedge}I).
$$
Since $\pi_L$ is injective on $I$, we have
$$|C|=k-2, \ \ 1\leqslant h-1\leqslant|C|-1.$$
By Lemma \ref{cyc:lem:strict-cyclic-layer}, we have
$$
\begin{aligned}
\left|(h-1)^{\wedge}I\cup(L+(h-1)^{\wedge}I)\right|
&=|(h-1)^{\wedge}I|
+\left|(L+(h-1)^{\wedge}I)\setminus(h-1)^{\wedge}I\right|\\
&\geqslant|(h-1)^{\wedge}I|
+|\pi_L((h-1)^{\wedge}I)|\\
&=|(h-1)^{\wedge}I|+|(h-1)^{\wedge}C|\\
&\geqslant|(h-1)^{\wedge}I|+k-2.
\end{aligned}
$$
Since
$$
h^{\wedge}I\subseteq ha+d\mathbb Z,\ \
(h-1)^{\wedge}I\cup(L+(h-1)^{\wedge}I)
\subseteq(h-1)a+d\mathbb Z,
\ \ a\not\equiv0\pmod d,
$$
we have
$$
h^{\wedge}I\mathbin{\dot\cup}
\bigl((h-1)^{\wedge}I\cup(L+(h-1)^{\wedge}I)\bigr)
\subseteq h^{\wedge}A.
$$
By Theorem \ref{thm:nathanson-restricted-minimum}, we have
$$
\begin{aligned}
2h(k-h)-k+1
&\geqslant|h^{\wedge}A|\\
&\geqslant|h^{\wedge}I|
+\left|(h-1)^{\wedge}I\cup(L+(h-1)^{\wedge}I)\right|\\
&\geqslant|h^{\wedge}I|+|(h-1)^{\wedge}I|+k-2\\
&\geqslant h(k-h-2)+1+(h-1)(k-h-1)+1+k-2\\
&=2h(k-h)-2h+1.
\end{aligned}
$$
Since $k\geqslant2h$ and $k\geqslant7$, we have
$$k=2h, \ \ h\geqslant4, \ \ |h^{\wedge}A|=2h^2-2h+1.$$

If $d\geqslant3$, then $ha$, $(h-1)a$ and $L+(h-2)a$
are pairwise distinct modulo $d$, and
$$\emptyset\neq L+(h-2)^{\wedge}I\subseteq h^{\wedge}A.$$
Thus,
$$
\begin{aligned}
|h^{\wedge}A|
&\geqslant|h^{\wedge}I|
+\left|(h-1)^{\wedge}I\cup(L+(h-1)^{\wedge}I)\right|
+|L+(h-2)^{\wedge}I|\\
&\geqslant|h^{\wedge}I|+|(h-1)^{\wedge}I|+2h-2+1\\
&\geqslant2h^2-2h+2,
\end{aligned}
$$
a contradiction. Hence, $d=2$.

Equality in the preceding bounds gives
$$
|(h-1)^{\wedge}I|=(h-1)^2+1,
$$
$$
\left|(h-1)^{\wedge}I\cup(L+(h-1)^{\wedge}I)\right|
=|(h-1)^{\wedge}I|+2h-2.
$$
Since
$$|I|=2h-2\geqslant6, \ \
2\leqslant h-1\leqslant|I|-2, \ \ d(I)=2,$$
by Theorem \ref{thm:nathanson-restricted-minimum}, we have
$$I=\min I+2[0,2h-3].$$
By \eqref{eq:unit-interval-restricted-ranks}, we have
$$(h-1)^{\wedge}I
=\min (h-1)^{\wedge}I+2[0,(h-1)^2].$$
Since $2\mid L$, we have
$$
\begin{aligned}
2h-2
&=\left|(h-1)^{\wedge}I\cup(L+(h-1)^{\wedge}I)\right|
-|(h-1)^{\wedge}I|\\
&=\min\left\{\frac L2,(h-1)^2+1\right\}.
\end{aligned}
$$
Together with
$$(h-1)^2+1-(2h-2)=(h-2)^2>0,$$
we have $L=4h-4$.
This proves the first assertion and its necessary equality conditions.

If equality holds in the first assertion, then
$$k=2h, \ \ h\geqslant4, \ \ L=4h-4\leqslant h^2.$$
Thus,
$$
|h^{\wedge}A|
-\bigl(h(k-h)-k+2+\min\{L,h(k-h)\}\bigr)
=(h-1)(h-3)\geqslant0.
$$
In every other case, we have
$$
|h^{\wedge}A|
\geqslant2h(k-h)-k+2
\geqslant h(k-h)-k+2+\min\{L,h(k-h)\}.
$$

This completes the proof of Proposition \ref{pro:one-residue-class}.
\end{proof}

\section{Consecutive restricted sumsets in cyclic groups}\label{sec:cyclic-sums}
For a disjoint partition $B=B_1\mathbin{\dot\cup}\cdots\mathbin{\dot\cup}B_t$, we have
\begin{equation}\label{cyc:eq:disjoint-layer}
j^{\wedge}B=\bigcup_{\substack{j_1,\ldots,j_t\geqslant0\\j_1+\cdots+j_t=j}}\left(j_1^{\wedge}B_1+\cdots+j_t^{\wedge}B_t\right).
\end{equation}

\begin{lem}\label{cyc:lem:full-coset-lifting}
Let $H$ be a subgroup of a finite cyclic group $G$ with $|H|\geqslant2$.
Let $\pi:G\to G/H$ be the quotient map, $\emptyset\neq T\subseteq G/H$ and $B\subseteq\pi^{-1}(T)$.
Put $\delta=|\pi^{-1}(T)\setminus B|$ and suppose that $\delta\leqslant|H|-2$. Then
$$\pi^{-1}(T+2^{\wedge}T)\subseteq3^{\wedge}B,$$
and equality holds for $|H|=2$.
\end{lem}

\begin{proof}
For $t\in T$, put $B_t=B\cap\pi^{-1}(\{t\})$. Then
$$|B_t|\geqslant|H|-\delta\geqslant 2.$$
For $|T|=1$, we have
$$\pi^{-1}(T+2^{\wedge}T)=\emptyset\subseteq 3^{\wedge}B.$$
If $|H|=2$, then $|B|=2$ and $3^{\wedge}B=\emptyset$.

Suppose that $|T|\geqslant 2$. By $\delta\leqslant|H|-2$ we have $|B|\geqslant 4$.
For every $s\in T+2^{\wedge}T$, write
$$s=t_1+t_2+t_3, \ \ t_1,t_2,t_3\in T, \ \ t_1\notin\{t_2,t_3\}.$$
Choose $c\in B_{t_3}$.
For every $z\in\pi^{-1}(\{s\})$, we have
$$B_{t_1}, \ \ z-c-(B_{t_2}\setminus\{c\})
\subseteq\pi^{-1}(\{t_1\}).$$
Since $t_1\neq t_2$ and
$$(|H|-|B_{t_1}|)+(|H|-|B_{t_2}|)\leqslant\delta,$$
we have
$$
\begin{aligned}
|B_{t_1}\cap(z-c-(B_{t_2}\setminus\{c\}))|
&\geqslant|B_{t_1}|+|B_{t_2}\setminus\{c\}|-|H|\\
&\geqslant|H|-\delta-1\geqslant1.
\end{aligned}
$$
Since $B_{t_1}\cap(B_{t_2}\cup\{c\})=\emptyset$, we have
$$\pi^{-1}(\{s\})
=B_{t_1}+(B_{t_2}\setminus\{c\})+c
\subseteq3^{\wedge}B.$$
Thus,
$$\pi^{-1}(T+2^{\wedge}T)\subseteq3^{\wedge}B.$$

If $|H|=2$, then $|B_t|=2$ for every $t\in T$.
For every $X\subseteq B$ with $|X|=3$, we have $|\pi(X)|\geqslant2$, thus,
$$\pi(3^{\wedge}B)\subseteq T+2^{\wedge}T, $$
it follows that
$$3^{\wedge}B\subseteq\pi^{-1}(T+2^{\wedge}T).$$
Hence,
$$3^{\wedge}B=\pi^{-1}(T+2^{\wedge}T).$$

This completes the proof of Lemma \ref{cyc:lem:full-coset-lifting}.
\end{proof}

\begin{lem}\label{cyc:lem:mixed-pair}
Let $B\subseteq\mathbb Z_n$ be a set of $b\geqslant 3$ elements such that $0\in B$ and $\langle B\rangle=\mathbb Z_n$. Then
$$
|B+2^{\wedge}B|\geqslant
\begin{cases}
\min\{n, 5\}, &b=3,\\
\min\{n, 2b\}, &b\geqslant4.
\end{cases}
$$
\end{lem}

\begin{proof}
If $B+2^{\wedge}B=\mathbb Z_n$, the assertion holds. Suppose that $B+2^{\wedge}B\neq\mathbb Z_n$ and put
$$H=\operatorname{Stab}(B+2^{\wedge}B). $$
If $H=\{0\}$, then
$$\operatorname{Stab}(B)\subseteq\operatorname{Stab}(B+2^{\wedge}B)=\{0\}.$$
By Lemmas \ref{thm:Kneser} and \ref{cyc:lem:strict-cyclic-layer}, we have
$$
|B+2^{\wedge}B|\geqslant |B|+|2^{\wedge}B|-1\geqslant
\begin{cases}
5, &b=3,\\
2b, &b\geqslant 4.
\end{cases}
$$

Suppose that $|H|\geqslant 2$.
Let $\pi:\mathbb Z_n\to\mathbb Z_n/H$ be the quotient map and put $Q=\pi(B)$. Then
\begin{equation}\label{e2}2^{\wedge}Q\subseteq\pi(2^{\wedge}B). \end{equation}
If $|Q|=1$, then by $0\in B$ we have
$$B\subseteq H, \ \ \mathbb Z_n=\langle B\rangle\subseteq H\subsetneq\mathbb Z_n,$$
a contradiction. Thus, $|Q|\geqslant 2$.

Now we shall show that
\begin{equation}\label{e3}|\pi(2^{\wedge}B)|\geqslant |Q|. \end{equation}
If $|Q|\geqslant 3$, then by Lemma \ref{cyc:lem:strict-cyclic-layer} and \eqref{e2}, we have
$$|\pi(2^{\wedge}B)|\geqslant |2^{\wedge}Q|\geqslant |Q|.$$
If $|Q|=2$, write $Q=\{u, v\}$.
Since $b\geqslant 3$, without loss of generality, assume that
$$|B\cap\pi^{-1}(\{u\})|\geqslant 2.$$
Then
$$\{2u, u+v\}\subseteq\pi(2^{\wedge}B), \ \ 2u\neq u+v,$$
thus,
$$|\pi(2^{\wedge}B)|\geqslant 2=|Q|.$$

If $|B\cap\pi^{-1}(\{u\})|=1$ for some $u\in Q$, then
\begin{equation}\label{e4}b\leqslant (|Q|-1)|H|+1. \end{equation}
By Lemma \ref{thm:Kneser}, we have
$$|B+2^{\wedge}B|\geqslant|B+H|+|2^{\wedge}B+H|-|H|=|H|(|Q|+|\pi(2^{\wedge}B)|-1). $$
By \eqref{e3} and \eqref{e4}, we have
$$|B+2^{\wedge}B|\geqslant (2|Q|-1)|H|\geqslant 2b+|H|-2\geqslant 2b. $$

Now we assume that $|B\cap\pi^{-1}(\{u\})|\geqslant 2$ for all $u\in Q$. Then
$$\pi(2^{\wedge}B)=2Q, \ \ \pi(B+2^{\wedge}B)=3Q.$$
Since $H=\operatorname{Stab}(B+2^{\wedge}B)$, we have
$$\operatorname{Stab}(2Q)\subseteq\operatorname{Stab}(3Q)=\{0\}.$$
By Lemma \ref{thm:Kneser} applied to $2Q$ and $3Q$, respectively, we have
$$
\begin{aligned}
|B+2^{\wedge}B|
&=|H||3Q|\geqslant |H|(|2Q|+|Q|-1)\\
&\geqslant |H|(3|Q|-2)\geqslant 2|H||Q|\geqslant 2b.
\end{aligned}
$$

This completes the proof of Lemma \ref{cyc:lem:mixed-pair}.
\end{proof}

\begin{pro}\label{cyc:pro:periodic-input-strict-pair-triple}
Let $0\in B\subseteq\mathbb Z_N$ be a set of $b\geqslant 7$ elements such that $\langle B\rangle=\mathbb Z_N$.
Suppose that $H\leqslant\mathbb Z_N$ satisfies
$$|H|\geqslant 2, \ \ \delta=|(B+H)\backslash B|\leqslant|H|-2.$$
If $N\geqslant2|B+H|$ or $a+H\subseteq B$ for some $a\in\mathbb Z_N$, then
$$|3^{\wedge}B|\geqslant\min\{N, 2b-1\}.$$
\end{pro}

\begin{proof}
Let $\pi:\mathbb Z_N\to\mathbb Z_N/H$ be the quotient map and put
$$Q=\pi(B), \ \ B_t=B\cap\pi^{-1}(\{t\}), \ \ t\in Q.$$
Then
$$0\in Q, \ \ \langle Q\rangle=\mathbb Z_N/H, \ \
b=|H||Q|-\delta, \ \ |B_t|\geqslant|H|-\delta\geqslant2.$$
By Lemma \ref{cyc:lem:full-coset-lifting}, we have
$$\pi^{-1}(Q+2^{\wedge}Q)\subseteq3^{\wedge}B.$$

If $|Q|=1$, then
$$\mathbb Z_N=\langle B\rangle\subseteq H, \ \ |B+H|=N, $$
thus, $N<2|B+H|$, it follows that $a+H\subseteq B$ for some $a\in\mathbb Z_N$. Hence, $B=H=\mathbb Z_N$, so,
$$3^{\wedge}B=\mathbb Z_N.$$

If $|Q|\geqslant 4$, then by Lemma \ref{cyc:lem:full-coset-lifting} and Lemma \ref{cyc:lem:mixed-pair}, we have
$$
\begin{aligned}
|3^{\wedge}B|
&\geqslant|H|\min\left\{\frac{N}{|H|}, 2|Q|\right\}\\
&=\min\{N, 2|H||Q|\}\geqslant\min\{N, 2b-1\}.
\end{aligned}
$$

Suppose that $|Q|=2$. Then $7\leqslant b\leqslant 2|H|$, thus, $|H|\geqslant 4$.
Write $Q=\{0,x\}$. Since $\langle Q\rangle=\mathbb Z_N/H$, we have
$$\langle x\rangle=\mathbb Z_N/H.$$
Moreover,
$$Q+2^{\wedge}Q=\{x, 2x\}. $$
By Lemma \ref{cyc:lem:full-coset-lifting} we have
$$\pi^{-1}(\{x, 2x\})\subseteq 3^{\wedge}B.$$

If $N=2|H|$, then
$$\mathbb Z_N=\pi^{-1}(\{x,2x\})\subseteq3^{\wedge}B\subseteq\mathbb Z_N.$$

If $N=3|H|$, then $N<2|B+H|$, thus, $|B_y|=|H|$ for some $y\in Q$.
Since $3y=0$ in $\mathbb Z_N/H$, by Lemma \ref{cyc:lem:strict-cyclic-layer}, we have
$$3^{\wedge}B_y=H, \ \ \mathbb Z_N=H\cup\pi^{-1}(\{x, 2x\})\subseteq3^{\wedge}B.$$

Suppose that $N\geqslant 4|H|$. Since $0, x, 2x, 3x$ are distinct, we have
$$\pi^{-1}(\{x, 2x\})\mathbin{\dot\cup}3^{\wedge}B_0\mathbin{\dot\cup}3^{\wedge}B_x\subseteq 3^{\wedge}B.$$
By Lemma \ref{cyc:lem:strict-cyclic-layer}, we have
$$|3^{\wedge}B_t|\geqslant|B_t| \ \ (|B_t|\geqslant 4), \ \ |3^{\wedge}B_t|=|B_t|-2 \ \ (|B_t|\in\{2, 3\}).$$
If $|B_0|, |B_x|\geqslant 4$, then
$$|3^{\wedge}B|
\geqslant2|H|+|B_0|+|B_x|
=2b+\delta\geqslant2b.$$
Otherwise, since $|B_0|+|B_x|=b\geqslant7$, we have
$$\max\{|B_0|,|B_x|\}\geqslant4, \ \
\delta\geqslant|H|-3\geqslant1.$$
Thus,
$$
\begin{aligned}
|3^{\wedge}B|
&\geqslant2|H|+|3^{\wedge}B_0|+|3^{\wedge}B_x|\\
&\geqslant2|H|+|B_0|+|B_x|-2\\
&=2b+\delta-2\geqslant2b-1.
\end{aligned}
$$

It remains to consider $|Q|=3$. By Lemma \ref{cyc:lem:mixed-pair}, we have
$$|Q+2^{\wedge}Q|\geqslant\min\left\{\frac{N}{|H|},5\right\}.$$
If $Q+2^{\wedge}Q=\mathbb Z_N/H$, then
$$\mathbb Z_N=\pi^{-1}(Q+2^{\wedge}Q)\subseteq 3^{\wedge}B.$$
If $|Q+2^{\wedge}Q|\geqslant 6$, then
$$|3^{\wedge}B|\geqslant|H||Q+2^{\wedge}Q|\geqslant 6|H|\geqslant 2b.$$

Suppose that $N\geqslant 6|H|$ and $|Q+2^{\wedge}Q|=5$.
Since $2^{\wedge}Q=\sum_{t\in Q}t-Q$, we have
$$|Q-Q|=5.$$
Write $Q=\{0, u, v\}$, where $u, v\neq 0$ and $u\neq v$.
If $v\neq -u$, $v\neq 2u$ and $u\neq 2v$, then
$$\{u, -u\}, \ \ \{v, -v\}, \ \ \{u-v, v-u\}$$
are pairwise disjoint.
Since a cyclic group has at most one nonzero element of order two, we have
$$|Q-Q|=1+|\{u, -u\}|+|\{v, -v\}|+|\{u-v, v-u\}|\geqslant 6,$$
a contradiction.
Thus, $Q$ is a three-term arithmetic progression.
Choose $a, d\in\mathbb Z_N$ such that
$$Q=\pi(a)+\{0, \pi(d), 2\pi(d)\}.$$
Since $0\in Q$ and $\langle Q\rangle=\mathbb Z_N/H$, we have
$$\pi(a)\in\{0, -\pi(d), -2\pi(d)\}, \ \ \langle\pi(d)\rangle=\mathbb Z_N/H.$$
Moreover,
$$Q+2^{\wedge}Q=3\pi(a)+\{\pi(d), 2\pi(d), 3\pi(d), 4\pi(d), 5\pi(d)\}.$$
Since $\pi(d)$ has order $N/|H|\geqslant 6$, we have
$$3\pi(a)\notin Q+2^{\wedge}Q, \ \ \pi^{-1}(Q+2^{\wedge}Q)\mathbin{\dot\cup}3^{\wedge}B_{\pi(a)}\subseteq 3^{\wedge}B.$$

If $\delta\geqslant1$, then
$$|3^{\wedge}B_{\pi(a)}|\geqslant|B_{\pi(a)}|-2\geqslant|H|-\delta-2,$$
thus,
$$|3^{\wedge}B|\geqslant 5|H|+|H|-\delta-2=2b+\delta-2\geqslant 2b-1.$$

If $\delta=0$ and $|H|\geqslant 4$, then
$$B_{\pi(a)}=a+H, \ \ 3^{\wedge}B_{\pi(a)}=3a+H,$$
thus,
$$|3^{\wedge}B|\geqslant 6|H|=2b.$$

Suppose that $\delta=0$ and $|H|\leqslant 3$.
Since $b=3|H|\geqslant7$, we have
$$|H|=3, \ \ b=9, \ \ H=\{z\in\mathbb Z_N:3z=0\}, \ \ \sum_{z\in H}z=0.$$
Thus,
$$
\begin{aligned}
3^{\wedge}B_{\pi(a)}
&=3^{\wedge}(a+H)=\{3a\},\\
3^{\wedge}B_{\pi(a+2d)}
&=3^{\wedge}(a+2d+H)=\{3a+6d\}.
\end{aligned}
$$
Since $\pi(d)$ has order $\frac{N}{|H|}\geqslant6$, we have
$$\pi(\{3a, 3a+6d\})\cap (Q+2^{\wedge}Q)=\emptyset.$$
If $6d=0$, then
$$2d\in\{z\in\mathbb Z_N: 3z=0\}=H, \ \ 2\pi(d)=0,$$
a contradiction.
Hence,
$$\pi^{-1}(Q+2^{\wedge}Q)\mathbin{\dot\cup}\{3a, 3a+6d\}\subseteq 3^{\wedge}B,$$
and
$$|3^{\wedge}B|\geqslant 5\cdot3+2=17=2b-1.$$

This completes the proof of Proposition \ref{cyc:pro:periodic-input-strict-pair-triple}.
\end{proof}

\begin{lem}\label{cyc:lem:aperiodic-popular-endpoint}
Let $B$ be an aperiodic subset of a finite cyclic group $G$ such that $b=|B|\geqslant6$.
Then $|\mathcal P(B)|\geqslant b-1$ in each of the following cases:

(1) $b=6$ and $|2^{\wedge}B|\leqslant9$;

(2) $b\geqslant7$ and $|2^{\wedge}B|\leqslant2b-2$.
\end{lem}

\begin{proof}
By Lemma \ref{lem:disjoint-pairs}, we have
$$\mu_B(s)\leqslant\left\lfloor\frac b2\right\rfloor, \ \ s\in G$$
and
$$\binom b2
=|2^{\wedge}B|+\sum_{s\in\mathcal P(B)}(\mu_B(s)-1).$$

Suppose that $b=2t$, where $t\geqslant3$.
For every $f\in G$, we have
$$\mu_B(f)=t\Longrightarrow f-B=B.$$
If $\mu_B(f)=\mu_B(g)=t$, then
$$B+(f-g)=f-(g-B)=f-B=B,$$
thus,
$$f-g\in\operatorname{Stab}(B)=\{0\}.$$
Hence,
$$|\{f\in G:\mu_B(f)=t\}|\leqslant1$$
and
$$\binom{2t}{2}
\leqslant|2^{\wedge}B|+(t-2)|\mathcal P(B)|+1.$$

If $t=3$, then $|2^{\wedge}B|\leqslant9$, thus,
$$15\leqslant9+|\mathcal P(B)|+1,$$
it follows that
$$|\mathcal P(B)|\geqslant5=b-1.$$

If $t\geqslant4$, then $|2^{\wedge}B|\leqslant4t-2$, thus,
$$
|\mathcal P(B)|
\geqslant\frac{\binom{2t}{2}-(4t-2)-1}{t-2}
=2t-1-\frac1{t-2}>2t-2.
$$
Hence,
$$|\mathcal P(B)|\geqslant2t-1=b-1.$$

It remains to consider $b=2t+1$, where $t\geqslant3$.
Then $|2^{\wedge}B|\leqslant4t$.
Suppose that $|\mathcal P(B)|\leqslant2t-1$.
Since
$$
\begin{aligned}
0
&\leqslant\sum_{s\in\mathcal P(B)}(t-\mu_B(s))\\
&=|2^{\wedge}B|+(t-1)|\mathcal P(B)|
-\binom{2t+1}{2}\\
&\leqslant1-(t-1)(2t-1-|\mathcal P(B)|),
\end{aligned}
$$
we have
$$0\leqslant2t-1-|\mathcal P(B)|\leqslant\frac1{t-1}<1.$$
Hence,
$$
|\mathcal P(B)|=2t-1, \ \
0\leqslant\sum_{s\in\mathcal P(B)}(t-\mu_B(s))\leqslant1.
$$

Put $F=\{f\in G:\mu_B(f)=t\}$. Then
$$
|F|\in\{2t-2,2t-1\}, \ \ |F|\geqslant4, \ \
|\mathcal P(B)|\leqslant|F|+1.
$$
For every $f\in F$, let $u_f$ be the unique point of $B$
outside the $t$ disjoint pairs of sum $f$. Then
$$
\begin{gathered}
u_f=\sum_{x\in B}x-tf, \ \
f-(B\setminus\{u_f\})=B\setminus\{u_f\},\\
\{x\in B:2x=f\}\subseteq\{u_f\}.
\end{gathered}
$$
For every $f,g\in F$, the bijection $x\mapsto f-x$
on $B\setminus\{u_f\}$ gives
$$
\mu_{B\setminus\{u_f\}}(2f-g)
=\mu_{B\setminus\{u_f\}}(g).
$$
Since distinct pairs of sum $g$ are disjoint, we have
$$0\leqslant\mu_B(g)-\mu_{B\setminus\{u_f\}}(g)\leqslant1.$$
Thus,
$$
\begin{aligned}
\mu_B(2f-g)
&\geqslant\mu_{B\setminus\{u_f\}}(2f-g)\\
&=\mu_{B\setminus\{u_f\}}(g)
\geqslant t-1\geqslant2.
\end{aligned}
$$
Put
$$2\cdot F=\{2f:f\in F\}, \ \ S=(2\cdot F)-F.$$
Then
$$S\subseteq\mathcal P(B), \ \ |S|\leqslant2t-1\leqslant|F|+1.$$

Put $H=\operatorname{Stab}(S)$.
By Lemma \ref{thm:Kneser}, we have
$$
|(2\cdot F)+H|
\leqslant|S|+|H|-|F+H|
\leqslant|H|+1.
$$
If $|H|=1$, then $|2\cdot F|\leqslant2$.
Put $K=\{x\in G:2x=0\}$.
Since $G$ is cyclic, we have $|K|\leqslant2$, thus,
$$
4\leqslant|F|\leqslant|F+K|
=|2\cdot F||K|\leqslant4.
$$
Hence,
$$|K|=2, \ \ F+K=F$$
and
$$S+K=(2\cdot F)-(F+K)=S, \ \ K\subseteq\operatorname{Stab}(S)=\{0\},$$
a contradiction. Thus, $|H|\geqslant2$.
Since
$$
|H|\mid|(2\cdot F)+H|, \ \
|H|\leqslant|(2\cdot F)+H|\leqslant|H|+1,
$$
we have
$$|(2\cdot F)+H|=|H|.$$

Fix $f_0\in F$ and put
$$H'=\{x\in G:2x\in H\}.$$
Since $G/H$ is cyclic, we have
$$[H':H]\leqslant2.$$
Moreover,
$$2\cdot F\subseteq2f_0+H, \ \ F\subseteq f_0+H'.$$
Since $f=2f-f\in S$ for every $f\in F$, we have
$$F\subseteq S, \ \ F+H\subseteq S\subseteq f_0+H'.$$
If $|F+H|>|H|$, then
$$
[H':H]=2, \ \
f_0+H'=F+H\subseteq S\subseteq f_0+H',
$$
thus,
$$H=\operatorname{Stab}(S)=H',$$
a contradiction.
Hence, $F\subseteq f_0+H$.
Since $S=(2\cdot F)-F$, $S+H=S$ and $S\neq\emptyset$, we have
$$
S=f_0+H, \ \
2t-2\leqslant|F|\leqslant|H|\leqslant2t-1.
$$

Put $C=u_{f_0}+H$. Since
$$u_f-u_{f_0}=-t(f-f_0)\in H, \ \ f\in F,$$
we have
$$\{u_f:f\in F\}\subseteq C.$$
Since $b>|H|=|C|$, choose $x\in B\setminus C$.
Then $x\neq u_f$ for every $f\in F$, thus,
$$f-x\in B, \ \ f\in F.$$
Put
$$A=F-x, \ \ E=f_0-x+H.$$
Then
$$A\subseteq B\cap E, \ \ |A|=|F|.$$

Suppose that $E+E\neq S$.
Since $E+E$ and $S$ are $H$-cosets, we have
$$
\mathcal P(A)
\subseteq(E+E)\cap\mathcal P(B)
\subseteq\mathcal P(B)\setminus S.
$$
Thus,
$$
|\mathcal P(A)|
\leqslant|\mathcal P(B)|-|S|
\leqslant|F|+1-|H|.
$$
Since
$$
|2^{\wedge}A|\leqslant|H|, \ \
|H|\in\{|F|,|F|+1\}, \ \ |F|\geqslant4,
$$
by Lemma \ref{lem:disjoint-pairs}, we have
$$
\begin{aligned}
\binom{|F|}{2}
&=|2^{\wedge}A|
+\sum_{s\in\mathcal P(A)}(\mu_A(s)-1)\\
&\leqslant|H|
+\left(\left\lfloor\frac{|F|}{2}\right\rfloor-1\right)
(|F|+1-|H|)\\
&\leqslant|F|+\left\lfloor\frac{|F|}{2}\right\rfloor-1\\
&<\binom{|F|}{2},
\end{aligned}
$$
a contradiction. Thus,
$$E+E=S=f_0+H.$$
Since $E=f_0-x+H$, we have
$$f_0-2x\in H, \ \ E=x+H.$$
Together with $x\notin C$, we have
$$E\cap C=\emptyset.$$

For every $y\in A$, we have $2y\in E+E=S$.
If $2y\in F$, then
$$y=u_{2y}\in C, \ \ y\in A\subseteq E,$$
contrary to $C\cap E=\emptyset$.
Thus,
$$2\cdot A\subseteq S\setminus F.$$
Since the doubling map on a cyclic group has fibres of size at most two,
we have
$$
2\leqslant\frac{|F|}{2}
=\frac{|A|}{2}
\leqslant|2\cdot A|
\leqslant|S\setminus F|
=|H|-|F|\leqslant1,
$$
a contradiction. Hence,
$$|\mathcal P(B)|\geqslant2t=b-1.$$

This completes the proof of Lemma \ref{cyc:lem:aperiodic-popular-endpoint}.
\end{proof}

\begin{pro}\label{cyc:pro:equality-popular-period}
Let $B\subseteq\mathbb Z_N$ be a set of $b\geqslant5$ elements such that $\langle B-B\rangle=\mathbb Z_N$.
Suppose that
$$|3^{\wedge}B|\leqslant 2b-2, \ \ |3^{\wedge}B|<N.$$

(1) If $b\geqslant 7$, then $B$ is aperiodic.

(2) If $B$ is aperiodic and $|\mathcal P(B)|\geqslant b-1$, and put $H=\operatorname{Stab}(B+\mathcal P(B))$. Then
$$B+\mathcal P(B)=3^{\wedge}B, \ \ |3^{\wedge}B|=2b-2, \ \ \mathcal P(B)+H=\mathcal P(B), \ \ |\mathcal P(B)|=b-1.$$
Moreover, if $B+\mathcal P(B)$ is periodic, let $\pi:\mathbb Z_N\to\mathbb Z_N/H$ be the quotient map.
Then there exists an integer $a\geqslant 1$ such that
$$b=a|H|+1, \ \ |\mathcal P(B)|=a|H|, \ \ |\pi(B)|=a+1.$$
\end{pro}

\begin{proof}
For every $z\in B$ and $1\leqslant j\leqslant b$, we have
$$
\begin{aligned}
j^{\wedge}(B-z)&=j^{\wedge}B-jz,\\
\mathcal P(B-z)&=\mathcal P(B)-2z,\\
(B-z)+\mathcal P(B-z)&=B+\mathcal P(B)-3z.
\end{aligned}
$$
Since translations preserve stabilizers and quotient cardinalities, we assume that
$$0\in B, \ \ \langle B\rangle=\mathbb Z_N.$$

(1) For $b\geqslant 7$, assume that $B$ is periodic. By Proposition \ref{cyc:pro:periodic-input-strict-pair-triple}, we have
$$|3^{\wedge}B|\geqslant\min\{N, 2b-1\}>|3^{\wedge}B|,$$
a contradiction.

(2) Suppose that $B$ is aperiodic and $|\mathcal P(B)|\geqslant b-1$.
By Lemma \ref{lem:disjoint-pairs}, we have
\begin{equation}\label{e5}B+\mathcal P(B)\subseteq 3^{\wedge}B, \end{equation}
thus,
\begin{equation}\label{e7}|B+\mathcal P(B)|\leqslant|3^{\wedge}B|<N. \end{equation}
Put $H=\operatorname{Stab}(B+\mathcal P(B))$.
If $|B+H|=|H|$, then by $0\in B$ we have
$$B\subseteq H, \ \ \mathbb Z_N=\langle B\rangle\subseteq H,$$
thus,
$$B+\mathcal P(B)=\mathbb Z_N,$$
contrary to \eqref{e7}. Hence, $|B+H|\neq |H|$. Since $|H|\mid|B+H|$, we have
\begin{equation}\label{e6}|B+H|\geqslant 2|H|. \end{equation}

By Lemma \ref{thm:Kneser} and \eqref{e5}, we have
$$
\begin{aligned}
2b-2
&\geqslant|B+\mathcal P(B)|\\
&\geqslant|B+H|+|\mathcal P(B)+H|-|H|\\
&\geqslant|B+H|+b-1-|H|.
\end{aligned}
$$
Thus,
$$0\leqslant|B+H|-b\leqslant|H|-1.$$

Suppose that $|B+H|-b\leqslant |H|-2$.
Then $|H|\geqslant 2$.
Since $B$ is aperiodic, we have $B+H\neq B$, thus,
$$1\leqslant |B+H|-b\leqslant|H|-2.$$
Hence,
$$|\mathcal P(B)+H|\geqslant b-1\geqslant |B+H|-|H|+1. $$
Since $|H|$ divides both $|B+H|$ and $|\mathcal P(B)+H|$, we have
$$|\mathcal P(B)+H|\geqslant |B+H|.$$
By Lemma \ref{thm:Kneser}, we have
\begin{equation}\label{e8}2|B+H|-|H|\leqslant|B+\mathcal P(B)|\leqslant 2b-2, \end{equation}
thus,
$$2(|B+H|-b)\leqslant|H|-2, $$
it follows that $|H|\geqslant 4$. By \eqref{e6} we have
$$b\geqslant|B+H|-\frac{|H|-2}{2}\geqslant 2|H|-\frac{|H|-2}{2}\geqslant 7.$$

By \eqref{e7} and \eqref{e8}, we have
$$2|B+H|-|H|\leqslant|B+\mathcal P(B)|<N, $$
combining with $|H|\mid N$, we have
$$N\geqslant 2|B+H|.$$
By Proposition \ref{cyc:pro:periodic-input-strict-pair-triple}, we have
$$|3^{\wedge}B|\geqslant\min\{N,2b-1\}>|3^{\wedge}B|,$$
a contradiction.
Hence,
$$|B+H|-b=|H|-1.$$

By Lemma \ref{thm:Kneser}, we have
$$
\begin{aligned}
2b-2
&\geqslant|3^{\wedge}B|
\geqslant|B+\mathcal P(B)|\\
&\geqslant|B+H|+|\mathcal P(B)+H|-|H|\\
&=b-1+|\mathcal P(B)+H|\\
&\geqslant b-1+|\mathcal P(B)|
\geqslant2b-2.
\end{aligned}
$$
Thus,
$$|B+\mathcal P(B)|=|3^{\wedge}B|=2b-2, \ \
|\mathcal P(B)|=|\mathcal P(B)+H|=b-1.$$
Together with \eqref{e5} and $\mathcal P(B)\subseteq\mathcal P(B)+H$, we have
$$B+\mathcal P(B)=3^{\wedge}B, \ \
\mathcal P(B)+H=\mathcal P(B).$$

Assume that $|H|\geqslant 2$. Taking $a=\frac{|B+H|}{|H|}-1\geqslant 1$, we have
$$b=a|H|+1, \ \ |\mathcal P(B)|=a|H|, \ \ |\pi(B)|=a+1.$$

This completes the proof of Proposition \ref{cyc:pro:equality-popular-period}.
\end{proof}

\begin{cor}\label{cyc:cor:maximum-pair-triple}
Let $B\subseteq\mathbb Z_N$ be a set of $b\geqslant7$ elements such that $0\in B$ and $\langle B\rangle=\mathbb Z_N$. Then
$$\max\{|2^{\wedge}B|, |3^{\wedge}B|\}\geqslant\min\{N, 2b-2\}.$$
\end{cor}

\begin{proof}
Suppose that
$$\max\{|2^{\wedge}B|, |3^{\wedge}B|\}<\min\{N, 2b-2\}.$$
By Proposition \ref{cyc:pro:equality-popular-period} (1), we have $B$ is aperiodic.
By Lemma \ref{cyc:lem:aperiodic-popular-endpoint}, we have
$$|\mathcal P(B)|\geqslant b-1.$$
By Proposition \ref{cyc:pro:equality-popular-period} (2), we have
$$2b-2=|3^{\wedge}B|<2b-2,$$
a contradiction.

This completes the proof of Corollary \ref{cyc:cor:maximum-pair-triple}.
\end{proof}

\begin{thm}\label{cyc:thm:pair-triple}
Let $B\subseteq\mathbb Z_N$ be a set of $b\geqslant6$ elements such that $0\in B$ and $\langle B\rangle=\mathbb Z_N$. Then
$$|2^{\wedge}B\cup3^{\wedge}B|
\geqslant\min\{N,2b-2\}.$$
\end{thm}

\begin{proof}
If $b\geqslant 7$, then by Corollary \ref{cyc:cor:maximum-pair-triple}, we have
$$|2^{\wedge}B\cup3^{\wedge}B|\geqslant\max\{|2^{\wedge}B|, |3^{\wedge}B|\}\geqslant\min\{N, 2b-2\}.$$

Suppose that $b=6$.
If $B$ is aperiodic and
$|2^{\wedge}B\cup3^{\wedge}B|<\min\{N,10\}$, then
$$|2^{\wedge}B|\leqslant9, \ \ |3^{\wedge}B|\leqslant9, \ \
|3^{\wedge}B|<N.$$
By Lemma \ref{cyc:lem:aperiodic-popular-endpoint}, we have $|\mathcal P(B)|\geqslant 5$.
By Proposition \ref{cyc:pro:equality-popular-period} (2), we have $10=|3^{\wedge}B|\leqslant 9$, a contradiction.

Suppose that $B$ is periodic and put $H=\operatorname{Stab}(B)$.
Then $|H|\mid 6$, thus,
$$|H|\in\{2, 3, 6\}. $$

If $|H|=6$, then by $0\in B$ and $\langle B\rangle=\mathbb Z_N$, we have $B=H=\mathbb Z_N$.
By Lemma \ref{cyc:lem:strict-cyclic-layer}, we have $2^{\wedge}B=\mathbb Z_N$.

If $|H|=2$, let $\pi:\mathbb Z_N\to\mathbb Z_N/H$ be the quotient map and put $T=\pi(B)$.
Then
$$B=\pi^{-1}(T), \ \ |T|=3, \ \ 0\in T, \ \
\langle T\rangle=\mathbb Z_N/H.$$
By Lemmas \ref{cyc:lem:full-coset-lifting} and \ref{cyc:lem:mixed-pair}, we have
$$|3^{\wedge}B|=2|T+2^{\wedge}T|\geqslant 2\min\left\{\frac N2, 5\right\}=\min\{N, 10\}. $$

If $|H|=3$, write $B=H\cup(x+H)$, where $x\notin H$.
Then
$$\langle x+H\rangle=\mathbb Z_N/H.$$
By Lemma \ref{cyc:lem:strict-cyclic-layer}, we have
$$2^{\wedge}H=H, \ \ 2^{\wedge}(x+H)=2x+H, \ \
H+(x+H)=x+H,$$
thus,
$$2^{\wedge}B=H\cup(x+H)\cup(2x+H).$$
If $\frac N3\in\{2, 3\}$, then $2^{\wedge}B=\mathbb Z_N$.
If $\frac N3\geqslant 4$, then $H$, $x+H$, $2x+H$ and $3x+H$ are pairwise disjoint.
Since $\sum_{y\in H}y=0$, we have
$$3x=\sum_{y\in x+H}y\in3^{\wedge}B\setminus2^{\wedge}B, $$
thus,
$$|2^{\wedge}B\cup3^{\wedge}B|\geqslant|2^{\wedge}B|+1=10.$$

This completes the proof of Theorem \ref{cyc:thm:pair-triple}.
\end{proof}

\begin{pro}\label{cyc:pro:constant-pair-padding}
Let $B\subseteq\mathbb Z_N$ be a set of $b\geqslant7$ elements such that $0\in B$, $\langle B\rangle=\mathbb Z_N$
and $3\leqslant m\leqslant\left\lfloor\frac{b-1}{2}\right\rfloor$.
If $\mu_B(c)\geqslant\left\lfloor\frac{m}{2}\right\rfloor+2$ for some $c\in\mathbb Z_N$, then
$$|m^{\wedge}B\cup(m+1)^{\wedge}B|\geqslant\max\{|2^{\wedge}B|, |3^{\wedge}B|\}\geqslant\min\{N, 2b-2\}.$$
\end{pro}

\begin{proof}
Write
$$\left\{A\subseteq B: |A|=2, \ \sum_{x\in A}x=c\right\}=\left\{A_1, \ldots, A_{\mu_B(c)}\right\}, $$
where $A_1, \ldots, A_{\mu_B(c)}$ are pairwise disjoint.
For $r\in\{2, 3\}$, put $k=\left\lfloor\frac{m+1-r}{2}\right\rfloor$. Then
$$0\leqslant k\leqslant\left\lfloor \frac{m}{2}\right\rfloor+2-r\leqslant\mu_B(c)-r, \ \ m\leqslant r+2k\leqslant m+1.$$
For every $X\subseteq B$ with $|X|=r$, we have
$$|\{i\in[1, \mu_B(c)]: A_i\cap X\neq\emptyset\}|\leqslant|X|=r.$$
Choose
$$K\subseteq\{i\in[1, \mu_B(c)]: A_i\cap X=\emptyset\}, \ \ |K|=k. $$
Then
$$\left|X\cup\bigcup_{i\in K}A_i\right|=r+2k, \ \ \sum_{x\in X}x+kc\in(r+2k)^{\wedge}B. $$
Thus,
$$kc+r^{\wedge}B\subseteq(r+2k)^{\wedge}B\subseteq m^{\wedge}B\cup(m+1)^{\wedge}B, \ \ r\in\{2, 3\}.$$
By Corollary \ref{cyc:cor:maximum-pair-triple}, we have
$$|m^{\wedge}B\cup(m+1)^{\wedge}B|\geqslant\max\{|2^{\wedge}B|, |3^{\wedge}B|\}\geqslant\min\{N, 2b-2\}.$$

This completes the proof of Proposition \ref{cyc:pro:constant-pair-padding}.
\end{proof}

\begin{pro}\label{cyc:lem:uniform-low-pair-density}
Let $B\subseteq\mathbb Z_N$ be a set of $b\geqslant 7$ elements and $3\leqslant m\leqslant\left\lfloor\frac{b-1}{2}\right\rfloor$.
Suppose that
$$\mu_B(s)\leqslant\left\lfloor\frac m2\right\rfloor+1, \ \ s\in\mathbb Z_N.$$
If $m=3$, $0\in B$ and $\langle B\rangle=\mathbb Z_N$, then
$$|3^{\wedge}B|\geqslant\min\{N, 2b-1\}.$$
If $m\geqslant 4$ and $N\leqslant 2b+3$, then
$$m^{\wedge}B=\mathbb Z_N.$$
\end{pro}

\begin{proof}
Suppose that $m=3$, $0\in B$ and $\langle B\rangle=\mathbb Z_N$. Then $\mu_B(s)\leqslant2$ for every $s\in\mathbb Z_N$, thus,
$$\binom b2=|2^{\wedge}B|+|\mathcal P(B)|.$$
For every $s\in 2^{\wedge}B$, distinct pairs of sum $s$ are disjoint, thus,
$$
\left|\{x\in B: s\in 2^{\wedge}(B\setminus\{x\})\}\right|=
\begin{cases}
b-2, &s\notin\mathcal P(B),\\
b, &s\in\mathcal P(B).
\end{cases}
$$
Since
$$x+2^{\wedge}(B\setminus\{x\})\subseteq3^{\wedge}B, \ \ x\in B,$$
we have
$$
\begin{aligned}
b|3^{\wedge}B|
&\geqslant\sum_{x\in B}|2^{\wedge}(B\setminus\{x\})|\\
&=(b-2)|2^{\wedge}B|+2|\mathcal P(B)|\\
&=(b-2)\binom b2-(b-4)|\mathcal P(B)|.
\end{aligned}
$$

Suppose that $|3^{\wedge}B|<\min\{N, 2b-1\}$. Then
$$|3^{\wedge}B|\leqslant 2b-2, \ \ |3^{\wedge}B|<N, $$
thus,
$$|\mathcal P(B)|\geqslant\frac{b(b-1)(b-6)}{2(b-4)}=b+\frac{b(b-7)(b-2)}{2(b-4)}\geqslant b. $$
By Proposition \ref{cyc:pro:equality-popular-period}, $B$ is aperiodic and $b\leqslant|\mathcal P(B)|=b-1$, a contradiction. Thus,
$$|3^{\wedge}B|\geqslant\min\{N, 2b-1\}.$$

Suppose that $m\geqslant4$ and $N\leqslant2b+3$.
For $A\subseteq B$, put
$$W_3(A)=\sum_{s\in\mathbb Z_N}\min\{\mu_A(s), 3\}.$$
We prove that
$$4^{\wedge}A\neq\mathbb Z_N\Longrightarrow W_3(A)\leqslant 3\left\lfloor\frac N2\right\rfloor+1.$$

Let $g\in\mathbb Z_N\setminus4^{\wedge}A$.
For $X, Y\subseteq A$ with $|X|=|Y|=2$, we have
$$\sum_{x\in X}x+\sum_{y\in Y}y=g\Longrightarrow X\cap Y\neq\emptyset, $$
thus,
$$2s=g\Longrightarrow\mu_A(s)\leqslant 1.$$

Suppose that $s+t=g$, $s\neq t$ and $\mu_A(s)\mu_A(t)>0$.
Choose $X\subseteq A$ with $|X|=2$ and $\sum_{x\in X}x=s$.
Since the pairs of sum $t$ are disjoint and each meets $X$, we have
$$
\mu_A(t)
\leqslant
\sum_{\substack{Y\subseteq A,\ |Y|=2\\ \sum_{y\in Y}y=t}}
|X\cap Y|
\leqslant|X|=2.
$$
Interchanging $s$ and $t$, we have $\mu_A(s)\leqslant2$.
If $\mu_A(s)\mu_A(t)=0$, then
$$\min\{\mu_A(s),3\}+\min\{\mu_A(t),3\}\leqslant3.$$
Thus, for $s+t=g$ with $s\neq t$, we have
$$\min\{\mu_A(s),3\}+\min\{\mu_A(t),3\}\leqslant3$$
unless $\mu_A(s)=\mu_A(t)=2$.

Suppose that $\mu_A(s)=\mu_A(t)=2$.
Write the two pairs of sum $s$ as $\{x_1,x_2\}$ and $\{x_3,x_4\}$.
Since each pair of sum $t$ meets both pairs of sum $s$, after interchanging $x_3,x_4$ if necessary, we have
$$
\begin{gathered}
|\{x_1,x_2,x_3,x_4\}|=4,\\
x_1+x_2=x_3+x_4=s, \ \ x_1+x_3=x_2+x_4=t.
\end{gathered}
$$
Thus,
$$2(x_2-x_3)=0, \ \ t-s=x_3-x_2\neq0.$$
If $N$ is odd, then $x_2=x_3$, a contradiction.
If $N$ is even, then
$$t=s+\frac N2, \ \ 2s=g-\frac N2.$$
For any two solutions $s,s'$ of the latter equation, we have
$$2(s'-s)=0, \ \ s'\in\left\{s,s+\frac N2\right\}.$$
Thus, they determine the same unordered pair $\{s,t\}$, and there is at most one exceptional unordered pair.

Put $f=|\{s\in\mathbb Z_N:2s=g\}|$.
If $N$ is odd, then $f=1$ and there is no exceptional pair, thus,
$$W_3(A)\leqslant3\frac{N-1}{2}+1.$$
If $N$ is even, then $f\in\{0,2\}$ and there is at most one exceptional pair, thus,
$$W_3(A)\leqslant3\frac{N-f}{2}+f+1
\leqslant\frac{3N}{2}+1.$$
Hence,
$$W_3(A)\leqslant3\left\lfloor\frac N2\right\rfloor+1.$$

Suppose that $g\notin m^{\wedge}B$.
For every $A\subseteq B$ with $|A|\leqslant b-m+4$, we have
$$|B\setminus A|\geqslant m-4.$$
Choose $C\subseteq B\setminus A$ with $|C|=m-4$. Then
$$\sum_{x\in C}x+4^{\wedge}A\subseteq m^{\wedge}B, \ \
g-\sum_{x\in C}x\notin4^{\wedge}A.$$
Thus,
$$W_3(A)\leqslant3\left\lfloor\frac N2\right\rfloor+1
\leqslant3b+4.$$

If $m\in\{4,5\}$, choose $A\subseteq B$ with $|A|=b-m+4$.
Since $\mu_A(s)\leqslant\mu_B(s)\leqslant3$ and $b\geqslant2m+1$, we have
$$
\begin{aligned}
W_3(A)
&=\binom{b-m+4}{2}
\geqslant\frac{(b+9)(b+7)}8\\
&=3b+4+\frac{(b-4)^2+15}{8}>3b+4,
\end{aligned}
$$
a contradiction.

Suppose that $m\geqslant6$. Then $b\geqslant13$.
Choose the least positive integer $a$ such that
$$\binom a2\geqslant4b+3.$$
Since
$$
\binom{b-m+4}{2}
\geqslant\frac{(b+9)(b+7)}8
=4b+3+\frac{(b-3)(b-13)}8
\geqslant4b+3,
$$
we have $a\leqslant b-m+4$.
By the minimality of $a$ and $a\leqslant b$, we have
$$4b+3\leqslant\binom a2
<4b+3+a-1\leqslant5b+2<10b.$$

Choose $A$ uniformly among all $a$-element subsets of $B$.
For every nonnegative integer $x$, we have
$$\min\{x,3\}\geqslant x-\frac{x(x-1)}{10},$$
since
$$
\min\{x,3\}-x+\frac{x(x-1)}{10}
=
\begin{cases}
\frac{x(x-1)}{10},&0\leqslant x\leqslant3,\\
\frac{(x-5)(x-6)}{10},&x\geqslant4
\end{cases}
\geqslant0.
$$
For every $X\subseteq B$ with $|X|=j\in\{2,4\}$, we have
$$\mathbb P(X\subseteq A)
=\frac{\binom{b-j}{a-j}}{\binom ba}
=\frac{\binom aj}{\binom bj}.$$
Since distinct pairs of the same sum are disjoint, the union of two such pairs has four elements.
Moreover,
$$\mu_B(s)-1\leqslant\left\lfloor\frac m2\right\rfloor
\leqslant\frac{b-1}{4}, \ \
\frac{(a-2)(a-3)}{(b-2)(b-3)}
\leqslant\frac{a(a-1)}{b(b-1)}.$$
Thus,
$$
\begin{aligned}
\mathbb E\sum_s\mu_A(s)&=\binom a2,\\
\mathbb E\sum_s\mu_A(s)(\mu_A(s)-1)
&=\frac{\binom a4}{\binom b4}
\sum_s\mu_B(s)(\mu_B(s)-1)\\
&\leqslant\frac{(b-1)\binom a2}{4}
\frac{(a-2)(a-3)}{(b-2)(b-3)}\\
&\leqslant\frac{\binom{a}{2}^2}{2b}.
\end{aligned}
$$
Hence,
$$\mathbb E W_3(A)
\geqslant\binom a2-\frac{\binom{a}{2}^2}{20b}.$$
Since $x\mapsto x-x^2/(20b)$ is increasing on $[0,10b]$,
$4b+3\leqslant\binom a2<10b$ and $b\geqslant13$, we have
$$
\begin{aligned}
\mathbb E W_3(A)
&\geqslant4b+3-\frac{(4b+3)^2}{20b}\\
&=3b+4+\frac{4b(b-11)-9}{20b}>3b+4,
\end{aligned}
$$
contrary to $W_3(A)\leqslant3b+4$ for every $a$-element subset $A\subseteq B$.
Thus,
$$m^{\wedge}B=\mathbb Z_N.$$

This completes the proof of Proposition \ref{cyc:lem:uniform-low-pair-density}.
\end{proof}

\begin{thm}\label{cyc:thm:popular-small-triple-seed}
Let $B\subseteq\mathbb Z_N$ be a set of $b\geqslant7$ elements such that $\langle B-B\rangle=\mathbb Z_N$.
If $N\geqslant2b$ and
$$|3^{\wedge}B|\leqslant2b-2, \ \ |\mathcal P(B)|\geqslant b-1,$$
then there exist a subgroup $H\leqslant\mathbb Z_N$ and $x,y\in\mathbb Z_N$ such that
$$|H|=b-1, \ \ x\notin y+H, \ \ B=(y+H)\cup\{x\}.$$
\end{thm}

\begin{proof}
For every $z\in B$, we have
$$|3^{\wedge}(B-z)|=|3^{\wedge}B|, \ \ \mathcal P(B-z)=\mathcal P(B)-2z, \langle B-z\rangle=\langle B-B\rangle. $$
By translation, assume that
$$0\in B, \ \ \langle B\rangle=\mathbb Z_N.$$
By Proposition \ref{cyc:pro:equality-popular-period}, we have
$$B+\mathcal P(B)=3^{\wedge}B, \ \
|3^{\wedge}B|=2b-2, \ \ |\mathcal P(B)|=b-1.$$

Suppose that $B+\mathcal P(B)$ is aperiodic.
Then
$$|B+\mathcal P(B)|=|B|+|\mathcal P(B)|-1=2b-2\leqslant N-2.$$
Since $|B|,|\mathcal P(B)|\geqslant2$, by the Kemperman Structure Theorem \cite[Section~2]{KempermanStructure},
there exist a subgroup $H\leqslant\mathbb Z_N$ with $|H|\geqslant 2$ and decompositions
$$B=B_1\mathbin{\dot\cup}B_0, \ \ \mathcal P(B)=P_1\mathbin{\dot\cup}P_0$$
such that
$$B_1+H=B_1, \ \ P_1+H=P_1, \ \ |B_0+H|=|P_0+H|=|H|, $$
and $(B_0, P_0)$ is an elementary pair.

Suppose that $H\subsetneq\mathbb Z_N$.
If $B_1=\emptyset$, then by $0\in B=B_0$ we have
$$B\subseteq H, \ \ \mathbb Z_N=\langle B\rangle\subseteq H\subsetneq\mathbb Z_N,$$
a contradiction. Thus, $B_1\neq\emptyset$.
Since
$$\operatorname{Stab}(B), \ \operatorname{Stab}(\mathcal P(B))\subseteq\operatorname{Stab}(B+\mathcal P(B))=\{0\},$$
we have
$$1\leqslant|B_0|, |P_0|\leqslant|H|-1.$$
If $|B_0|=1$, then
$$|\mathcal P(B)|=|B|-1=|B_1|.$$
Since $|H|$ divides both $|B_1|$ and $|P_1|$, we have
$$|H|\mid|\mathcal P(B)|-|P_1|=|P_0|,$$
contrary to $1\leqslant|P_0|\leqslant|H|-1$.
Thus, $|B_0|\geqslant 2$ and
$$|B+H|-b=|H|-|B_0|\leqslant|H|-2.$$
Taking $x\in B_1$, we have
$$x+H\subseteq B_1\subseteq B.$$
By Proposition \ref{cyc:pro:periodic-input-strict-pair-triple}, we have
$$|3^{\wedge}B|\geqslant\min\{N, 2b-1\}=2b-1>2b-2,$$
a contradiction.

If $H=\mathbb Z_N$, then $B_1=P_1=\emptyset$, thus, $(B, \mathcal P(B))$ is an elementary pair.
Since $|B|,|\mathcal P(B)|\geqslant 2$, type I cannot hold.
In types III and IV, there exist a subgroup $K\leqslant\mathbb Z_N$ and $x\in\mathbb Z_N$ such that
$$B\subseteq x+K, \ \ |B+\mathcal P(B)|\geqslant|K|-1.$$
Since $0\in B$ and $\langle B\rangle=\mathbb Z_N$, we have
$$B\subseteq K, \ \ K=\mathbb Z_N,$$
thus,
$$N-1\leqslant|B+\mathcal P(B)|\leqslant N-2,$$
a contradiction.

Hence, type II holds, and
$$B=y+d\cdot[0, b-1]$$
for some $y, d\in\mathbb Z_N$.
Since $0\in B$, we have
$$y\in\langle d\rangle, \ \ B\subseteq\langle d\rangle, \ \ \mathbb Z_N=\langle B\rangle\subseteq\langle d\rangle.$$
Thus, $\langle d\rangle=\mathbb Z_N$.
By \eqref{eq:unit-interval-restricted-ranks}, we have
$$3^{\wedge}B=3y+d\cdot\pi_N([3, 3b-6]),$$
and
$$|3^{\wedge}B|=|\pi_N([3,3b-6])|=\min\{N,3b-8\}>2b-2,$$
a contradiction. Hence, $B+\mathcal P(B)$ is periodic.

Put $H=\operatorname{Stab}(3^{\wedge}B)$. Then $|H|\geqslant 2$.
Let $\pi: \mathbb Z_N\to\mathbb Z_N/H$ be the quotient map and put
$$Q=\pi(B), \ \ R=\pi(3^{\wedge}B), \ \ B_y=B\cap\pi^{-1}(\{y\}), \ \ y\in Q.$$
By Proposition \ref{cyc:pro:equality-popular-period} (2), there exists an integer $a\geqslant1$ such that
$$b=a|H|+1, \ \ |Q|=a+1, \ \ |R|=2a.$$
Moreover,
$$0\in Q, \ \ \langle Q\rangle=\mathbb Z_N/H, \ \ \operatorname{Stab}(R)=\{0\}, \ \ \frac{N}{|H|}\geqslant 2a+1.$$

Suppose that $|B_y|\geqslant 2$ for every $y\in Q$.
Then
$$Q+2^{\wedge}Q=3^{\wedge}Q\cup\{2u+v: u, v\in Q, \ u\neq v\}\subseteq R.$$
If $a\geqslant 2$, then by Lemma \ref{cyc:lem:mixed-pair}, we have
$$|R|\geqslant\min\left\{\frac{N}{|H|}, 2a+1\right\}=2a+1,$$
a contradiction.

If $a=1$, write $Q=\{0, y\}$.
Since $\langle y\rangle=\mathbb Z_N/H$ and $N/|H|\geqslant 3$, we have
$$\{y, 2y\}\subseteq R, \ \ |\{y, 2y\}|=2.$$
Since $|B_0|+|B_y|=b\geqslant 7$, there exists $u\in\{0, y\}$ such that $|B_u|\geqslant 4$.
Thus,
$$\emptyset\neq 3^{\wedge}B_u, \ \ 3u\in R\setminus\{y, 2y\}, \ \ |R|\geqslant 3,$$
contrary to $|R|=2$.

Hence, $|B_u|=1$ for some $u\in Q$.
Since
$$\sum_{y\in Q}(|H|-|B_y|)
=(a+1)|H|-b
=|H|-1
=|H|-|B_u|,$$
we have
$$|B_y|=|H|, \ \ y\in Q\setminus\{u\}.$$
Write
$$B=\pi^{-1}(T)\mathbin{\dot\cup}\{z\}, \ \
Q=T\mathbin{\dot\cup}\{\pi(z)\}, \ \ |T|=a.$$
Since $|H|\geqslant2$, we have
$$\pi(2^{\wedge}\pi^{-1}(T))=T+T.$$
Moreover,
$$3^{\wedge}B
=3^{\wedge}\pi^{-1}(T)
\cup\bigl(z+2^{\wedge}\pi^{-1}(T)\bigr),$$
thus,
$$R=\pi(3^{\wedge}\pi^{-1}(T))
\cup\bigl(\pi(z)+T+T\bigr).$$

Suppose that $|H|\geqslant3$.
Since $|\pi^{-1}(\{t\})|=|H|\geqslant3$ for every $t\in T$, we have
$$\pi(3^{\wedge}\pi^{-1}(T))=T+T+T.$$
Thus,
$$
\begin{aligned}
R
&=(T+T+T)\cup\bigl(\pi(z)+T+T\bigr)\\
&=T+T+Q.
\end{aligned}
$$
Since
$$\operatorname{Stab}(T+T)\subseteq\operatorname{Stab}(R)=\{0\},$$
by Lemma \ref{thm:Kneser}, we have
$$2a=|R|
\geqslant|T+T|+|Q|-1
\geqslant(2a-1)+(a+1)-1
=3a-1.$$
Together with $a\geqslant1$, we have $a=1$.
Writing $T=\{\pi(y)\}$, we obtain
$$B=(y+H)\cup\{z\}, \ \ z\notin y+H, \ \ |H|=b-1.$$

It remains to consider $|H|=2$.
Since $b=2a+1\geqslant7$, we have $a\geqslant3$.
By Lemma \ref{cyc:lem:full-coset-lifting}, we have
$$3^{\wedge}\pi^{-1}(T)=\pi^{-1}(T+2^{\wedge}T).$$
Thus,
$$
\begin{aligned}
R
&=(T+2^{\wedge}T)\cup\bigl(\pi(z)+T+T\bigr)\\
&=T+\bigl(2^{\wedge}T\cup(\pi(z)+T)\bigr)\\
&=T+2^{\wedge}Q.
\end{aligned}
$$
Since $\operatorname{Stab}(R)=\{0\}$, by Lemma \ref{thm:Kneser}, we have
$$2a=|R|\geqslant a+|2^{\wedge}Q|-1.$$
By Lemma \ref{cyc:lem:strict-cyclic-layer}, we have
$$|Q|\leqslant|2^{\wedge}Q|\leqslant a+1=|Q|,$$
thus,
$$|2^{\wedge}Q|=|Q|.$$

Since
$$|Q|=a+1<2a+1\leqslant\frac{N}{|H|}, \ \ 0\in Q, \ \ \langle Q\rangle=\mathbb Z_N/H,$$
the set $Q$ is not a subgroup coset.
By the equality cases in the proof of Lemma \ref{cyc:lem:strict-cyclic-layer}, there exists a subgroup $K\leqslant\mathbb Z_N/H$ such that
$$|Q|=4, \ \ |K|=2, \ \ Q+K=Q.$$
Put $\widetilde K=\pi^{-1}(K)$. Then
$$|\widetilde K|=4, \ \ \pi^{-1}(Q)+\widetilde K=\pi^{-1}(Q), \ \ |\pi^{-1}(Q)\setminus B|=1.$$
Since $\pi^{-1}(Q)$ is the union of two $\widetilde K$-cosets, there exists $x\in B$ such that
$$x+\widetilde K\subseteq B.$$
Moreover,
$$B+\widetilde K=\pi^{-1}(Q), \ \ |B+\widetilde K|-b=1\leqslant|\widetilde K|-2.$$
By Proposition \ref{cyc:pro:periodic-input-strict-pair-triple}, we have
$$|3^{\wedge}B|\geqslant\min\{N, 2b-1\}=2b-1,$$
contrary to $|3^{\wedge}B|=2b-2$.

This completes the proof of Theorem \ref{cyc:thm:popular-small-triple-seed}.
\end{proof}

\begin{cor}\label{cyc:cor:maximum-pair-triple-inverse}
Let $B\subseteq\mathbb Z_N$ be a set of $b\geqslant 7$ elements such that $\langle B-B\rangle=\mathbb Z_N$ and $N\geqslant 2b$.
If
$$\max\{|2^{\wedge}B|, |3^{\wedge}B|\}\leqslant 2b-2,$$
then there exist a subgroup $H\leqslant\mathbb Z_N$ and $x, y\in\mathbb Z_N$ such that
$$B=(y+H)\cup\{x\}, \ \ |H|=b-1, \ \ x\notin y+H.$$
\end{cor}

\begin{proof}
Since $|3^{\wedge}B|\leqslant 2b-2<N$, by Proposition \ref{cyc:pro:equality-popular-period}, $B$ is aperiodic.
By Lemma \ref{cyc:lem:aperiodic-popular-endpoint}, we have
$$|\mathcal P(B)|\geqslant b-1.$$
By Theorem \ref{cyc:thm:popular-small-triple-seed}, there exist a subgroup $H\leqslant\mathbb Z_N$ and $x,y\in\mathbb Z_N$ such that
$$B=(y+H)\cup\{x\}, \ \ |H|=b-1, \ \ x\notin y+H.$$

This completes the proof of Corollary \ref{cyc:cor:maximum-pair-triple-inverse}.
\end{proof}

\begin{lem}\label{cyc:lem:uniform-zero-deletion}
Let $0\in B\subseteq\mathbb Z_N$ be a set of $b\geqslant6$ elements such that
$\langle B\rangle=\mathbb Z_N$, $N\geqslant 2b$ and $2\leqslant m\leqslant\left\lfloor\frac{b-1}{2}\right\rfloor$.
If $|U_m(B)|\leqslant 2b-2$, then every set $B\setminus\{z\}$, $z\in B$, either
$$\langle B\setminus\{z\}-B\setminus\{z\}\rangle=\mathbb Z_N$$ or is a subgroup of order $b-1$.
In particular,
$$\langle B\setminus\{0\}-B\setminus\{0\}\rangle=\mathbb Z_N. $$
If $B=H\cup\{x\}$ for a subgroup $H$ of order $b-1$, then
$$U_m(B)=H\cup(x+H), \ \ \operatorname{Stab}(U_m(B))=H.$$
\end{lem}

\begin{proof}
Fix $z\in B$ and put $C=B\setminus\{z\}$.
Suppose that $H=\langle C-C\rangle\subsetneq\mathbb Z_N$.
Choose $y\in C$. Then
$$|C|=b-1\geqslant5, \ \ C\subseteq y+H.$$
If $z\in y+H$, then $0\in B\subseteq y+H$ gives
$$y+H=H, \ \
\mathbb Z_N=\langle B\rangle\subseteq H\subsetneq\mathbb Z_N,$$
a contradiction. Thus, $z\notin y+H$.

Suppose that $C$ is not a subgroup coset.
Since
$$1\leqslant m-1<m
\leqslant\left\lfloor\frac{b-1}{2}\right\rfloor
\leqslant b-3,$$
by Lemma \ref{cyc:lem:strict-cyclic-layer}, we have
$$|(m-1)^{\wedge}C|\geqslant b-1, \ \
|m^{\wedge}C|\geqslant b.$$
Since $z\notin y+H$, we have
$$
\begin{gathered}
m^{\wedge}C\subseteq my+H, \ \
z+(m-1)^{\wedge}C\subseteq z+(m-1)y+H,\\
my+H\neq z+(m-1)y+H.
\end{gathered}
$$
Thus,
$$m^{\wedge}B
=m^{\wedge}C\mathbin{\dot\cup}\bigl(z+(m-1)^{\wedge}C\bigr),$$
and
$$
\begin{aligned}
|U_m(B)|
&\geqslant|m^{\wedge}B|\\
&=|m^{\wedge}C|+|(m-1)^{\wedge}C|\\
&\geqslant b+(b-1)=2b-1,
\end{aligned}
$$
contrary to $|U_m(B)|\leqslant2b-2$.

Hence, $C=y+H$ and $|H|=b-1$.
For $j\in\{m-1,m,m+1\}$, we have
$$1\leqslant j\leqslant b-2=|H|-1.$$
By Lemma \ref{cyc:lem:strict-cyclic-layer}, we have
$$j^{\wedge}C=jy+H, \ \ j\in\{m-1,m,m+1\}.$$
Since $N\geqslant2b>2|H|$ and $|H|\mid N$, we have
$$\frac{N}{|H|}\geqslant3.$$

If $0\notin C$, then $z=0$ and
$$\langle y+H\rangle=\mathbb Z_N/H.$$
Thus,
$$U_m(B)
=((m-1)y+H)\mathbin{\dot\cup}(my+H)
\mathbin{\dot\cup}((m+1)y+H),$$
and
$$|U_m(B)|=3|H|=3(b-1)>2b-2,$$
a contradiction.
Hence, $0\in C$ and $C=H$.

Since $B=H\mathbin{\dot\cup}\{z\}$, we have
$$
\begin{aligned}
U_m(B)
&=m^{\wedge}H\cup\bigl(z+(m-1)^{\wedge}H\bigr)
\cup(m+1)^{\wedge}H\cup\bigl(z+m^{\wedge}H\bigr)\\
&=H\cup(z+H).
\end{aligned}
$$
Thus,
$$|U_m(B)|=2|H|=2b-2.$$
Since
$$\langle z+H\rangle=\mathbb Z_N/H, \ \ \frac{N}{|H|}\geqslant3,$$
we have $2z\notin H$.
If $t\in\operatorname{Stab}(U_m(B))\setminus H$, then
$$\{H,z+H\}+(t+H)=\{H,z+H\},$$
thus,
$$t+H=z+H, \ \ 2z\in H,$$
a contradiction.
Together with $H\subseteq\operatorname{Stab}(U_m(B))$, we have
$$\operatorname{Stab}(U_m(B))=H.$$

Since $0\notin B\setminus\{0\}$, the set $B\setminus\{0\}$ is not a subgroup.
Hence,
$$\langle(B\setminus\{0\})-(B\setminus\{0\})\rangle=\mathbb Z_N.$$
If $B=H\cup\{x\}$ as in the last assertion, taking $z=x$ and $C=H$ in the preceding calculation gives
$$U_m(B)=H\cup(x+H), \ \ \operatorname{Stab}(U_m(B))=H.$$

This completes the proof of Lemma \ref{cyc:lem:uniform-zero-deletion}.
\end{proof}

\begin{proof}[Proof of Theorems \ref{thm:r4-coarse-cyclic-all-orders} and \ref{thm:r4-coarse-cyclic-inverse}]
If $B=H\cup\{x\}$ with $|H|=b-1$, then by Lemma
\ref{cyc:lem:uniform-zero-deletion}, we have
$$U_m(B)=H\mathbin{\dot\cup}(x+H), \ \ |U_m(B)|=2|H|=2b-2.$$
This proves the converse assertion of Theorem
\ref{thm:r4-coarse-cyclic-inverse}.

We prove the direct assertion and the necessary equality condition
simultaneously by strong induction on $b$, in the fixed group $\mathbb Z_N$.
Let $0\in B\subseteq\mathbb Z_N$ satisfy
$\langle B\rangle=\mathbb Z_N$ and $|B|=b\geqslant6$.

For $m=2$, by Theorem \ref{cyc:thm:pair-triple}, we have
$$|U_2(B)|\geqslant\min\{N,2b-2\}.$$
If $b\geqslant7$, $N\geqslant2b$ and $|U_2(B)|=2b-2$, then
$$\max\{|2^{\wedge}B|,|3^{\wedge}B|\}\leqslant2b-2.$$
By Corollary \ref{cyc:cor:maximum-pair-triple-inverse}, we have
$$B=(y+H)\cup\{x\}, \ \ |H|=b-1, \ \ x\notin y+H.$$
Since
$$\langle(B\setminus\{x\})-(B\setminus\{x\})\rangle
=H<\mathbb Z_N,$$
by Lemma \ref{cyc:lem:uniform-zero-deletion}, $y+H$ is a subgroup.
Thus,
$$y+H=H, \ \ B=H\cup\{x\}.$$

Let $3\leqslant m\leqslant\left\lfloor\frac{b-1}{2}\right\rfloor$.
If $\mu_B(c)\geqslant\left\lfloor\frac m2\right\rfloor+2$
for some $c\in\mathbb Z_N$, then by Proposition
\ref{cyc:pro:constant-pair-padding}, we have
$$|U_m(B)|\geqslant\max\{|2^{\wedge}B|,|3^{\wedge}B|\}
\geqslant\min\{N,2b-2\}.$$
If also $N\geqslant2b$ and $|U_m(B)|=2b-2$, then
$$\max\{|2^{\wedge}B|,|3^{\wedge}B|\}\leqslant2b-2.$$
By Corollary \ref{cyc:cor:maximum-pair-triple-inverse}
and Lemma \ref{cyc:lem:uniform-zero-deletion}, we have
$$B=H\cup\{x\}, \ \ |H|=b-1.$$

Suppose that
$$\mu_B(c)\leqslant\left\lfloor\frac m2\right\rfloor+1, \ \ c\in\mathbb Z_N. $$
If $m=3$, then by Proposition \ref{cyc:lem:uniform-low-pair-density}, we have
$$|U_3(B)|\geqslant|3^{\wedge}B|
\geqslant\min\{N,2b-1\}.$$
This proves the direct assertion and excludes equality when $N\geqslant 2b$.

The preceding arguments prove the required assertions when $6\leqslant b\leqslant8$.
Suppose that $b\geqslant9$ and both assertions hold for all smaller
cardinalities in $\mathbb Z_N$.
It remains to consider
$$4\leqslant m\leqslant\left\lfloor\frac{b-1}{2}\right\rfloor.$$

If $N\leqslant 2b+3$, then by Proposition \ref{cyc:lem:uniform-low-pair-density}, we have
$$\mathbb Z_N=m^{\wedge}B\subseteq U_m(B)\subseteq\mathbb Z_N,$$
thus,
$$|U_m(B)|=N\geqslant\min\{N,2b-2\}.$$
If also $N\geqslant2b$, then
$$|U_m(B)|=N\geqslant2b>2b-2.$$

Suppose that $N\geqslant2b+4$.
If $|U_m(B)|\geqslant2b-1$, then the direct assertion holds
and equality at $2b-2$ does not occur.
Thus, we may assume that
$$|U_m(B)|\leqslant2b-2.$$
The case $B=H\cup\{x\}$ with $|H|=b-1$ was proved above.
Suppose that $B$ is not of this form.
By Lemma \ref{cyc:lem:uniform-zero-deletion}, we have
$$\langle(B\setminus\{z\})-(B\setminus\{z\})\rangle
=\mathbb Z_N, \ \ z\in B.$$

Suppose that a subgroup $K\leqslant\mathbb Z_N$ satisfies
$$[\mathbb Z_N:K]\geqslant3, \ \
|B\cap K|\geqslant b-3, \ \ |B\setminus K|\geqslant2.$$
Put $C=B\cap K$.
For $j\in\{m-2,m-1,m\}$, we have
$$2\leqslant j\leqslant m
\leqslant\left\lfloor\frac{b-1}{2}\right\rfloor
\leqslant b-5\leqslant|C|-2.$$
By Lemma \ref{cyc:lem:strict-cyclic-layer}, we have
$$U_j(C)\subseteq K, \ \ |U_j(C)|\geqslant|C|.$$
Since $\langle B\rangle=\mathbb Z_N$, we have
$$\left\langle\{x+K:x\in B\setminus K\}\right\rangle
=\mathbb Z_N/K.$$

If $y,z\in B\setminus K$ satisfy $y+K\neq z+K$, then
$$|\{K,y+K,z+K\}|=3.$$
By \eqref{cyc:eq:disjoint-layer}, we have
$$
U_m(C)\mathbin{\dot\cup}
\bigl(y+U_{m-1}(C)\bigr)\mathbin{\dot\cup}
\bigl(z+U_{m-1}(C)\bigr)
\subseteq U_m(B).
$$
Otherwise, choose distinct $y,z\in B\setminus K$. Then
$$y+K=z+K, \ \ \langle y+K\rangle=\mathbb Z_N/K.$$
Since $[\mathbb Z_N:K]\geqslant3$, we have
$$|\{K,y+K,y+z+K\}|=3.$$
By \eqref{cyc:eq:disjoint-layer}, we have
$$
U_m(C)\mathbin{\dot\cup}
\bigl(y+U_{m-1}(C)\bigr)\mathbin{\dot\cup}
\bigl(y+z+U_{m-2}(C)\bigr)
\subseteq U_m(B).
$$
In both cases, we have
$$2b-2\geqslant|U_m(B)|\geqslant3|C|
\geqslant3b-9>2b-2,$$
a contradiction.

Put $A=B\setminus\{0\}$. Then
$$|A|=b-1\geqslant8, \ \ \langle A-A\rangle=\mathbb Z_N.$$
For distinct $e,f\in A$ such that
$\langle B\setminus\{e,f\}\rangle\neq\mathbb Z_N$, put
$$X=B\setminus\{e,f\}, \ \ H=\langle X\rangle<\mathbb Z_N.$$
Since $0\in X$ and
$$\langle(X\cup\{e\})-(X\cup\{e\})\rangle
=\langle(X\cup\{f\})-(X\cup\{f\})\rangle=\mathbb Z_N,$$
we have
$$e,f\notin H, \ \ B\cap H=X, \ \
|B\cap H|=b-2, \ \ B\setminus H=\{e,f\}.$$
By the preceding subgroup exclusion, we have
$$[\mathbb Z_N:H]=2.$$
Since a cyclic group has at most one subgroup of index two, we have
$$
\left|
\left\{\{e,f\}\subseteq A:
e\neq f,\ \langle B\setminus\{e,f\}\rangle\neq\mathbb Z_N
\right\}
\right|\leqslant1.
$$

Put
$$
\mathcal D=
\{e-f:e,f\in A,\ e\neq f,\
\langle B\setminus\{e,f\}\rangle=\mathbb Z_N\}.
$$
Then
$$\mathcal D=-\mathcal D, \ \ 0\notin\mathcal D.$$
If distinct $e,f\in A$ satisfy
$\langle B\setminus\{e,f\}\rangle\neq\mathbb Z_N$,
choose $a\in A\setminus\{e,f\}$.
Since at most one pair has this property, we have
$$\langle B\setminus\{e,a\}\rangle
=\langle B\setminus\{a,f\}\rangle=\mathbb Z_N,$$
thus,
$$e-a,a-f\in\mathcal D, \ \
e-f=(e-a)+(a-f)\in\langle\mathcal D\rangle.$$
Hence,
$$A-A\subseteq\langle\mathcal D\rangle, \ \
\mathbb Z_N=\langle A-A\rangle
\subseteq\langle\mathcal D\rangle\subseteq\mathbb Z_N,$$
and therefore $\langle\mathcal D\rangle=\mathbb Z_N$.

We prove that $|\mathcal D|\geqslant11$.
Suppose that distinct $e,f\in A$ satisfy
$\langle B\setminus\{e,f\}\rangle\neq\mathbb Z_N$,
and put
$$H=\langle B\setminus\{e,f\}\rangle, \ \ A_0=A\cap H.$$
Then
$$|A_0|=b-3, \ \ e\notin H.$$
Since every pair other than $\{e,f\}$ has a generating complement, we have
$$
\bigl((A_0-A_0)\setminus\{0\}\bigr)
\mathbin{\dot\cup}(e-A_0)\subseteq\mathcal D.
$$
Thus,
$$
\begin{aligned}
|\mathcal D|
&\geqslant|A_0-A_0|-1+|A_0|\\
&\geqslant2|A_0|-1=2b-7\geqslant11.
\end{aligned}
$$

Suppose that
$$\langle B\setminus\{e,f\}\rangle=\mathbb Z_N,
\ \ e,f\in A,\ \ e\neq f.$$
Then
$$\mathcal D=(A-A)\setminus\{0\}.$$
If $A-A=\mathbb Z_N$, then
$$|\mathcal D|=N-1\geqslant2b-1\geqslant11.$$
Otherwise, put $K=\operatorname{Stab}(A-A)$. Then $K<\mathbb Z_N$.
If $|A+K|=|K|$, then
$$A-A\subseteq K, \ \
\mathbb Z_N=\langle A-A\rangle\subseteq K<\mathbb Z_N,$$
a contradiction.
Since $|K|\mid|A+K|$, we have
$$|A+K|\geqslant2|K|.$$
Together with $|-A+K|=|A+K|$, by Lemma \ref{thm:Kneser}, we have
$$
\begin{aligned}
|\mathcal D|
&=|A-A|-1\\
&\geqslant2|A+K|-|K|-1\\
&\geqslant\frac32|A+K|-1\\
&\geqslant\frac32(b-1)-1\geqslant11.
\end{aligned}
$$

For distinct $e,f\in A$ such that
$\langle B\setminus\{e,f\}\rangle=\mathbb Z_N$,
put $X=B\setminus\{e,f\}$. Then
$$
0\in X, \ \ \langle X\rangle=\mathbb Z_N, \ \
|X|=b-2\geqslant7, \ \
2\leqslant m-1\leqslant
\left\lfloor\frac{|X|-1}{2}\right\rfloor.
$$
Since $N\geqslant2b>2|X|$, by the direct induction hypothesis, we have
$$|U_{m-1}(X)|\geqslant2|X|-2=2b-6.$$
If equality holds, then by the inverse induction hypothesis, we have
$$X=K\cup\{x\}, \ \ |K|=b-3\geqslant6, \ \
B=K\mathbin{\dot\cup}\{x,e,f\}.$$
Since $N\geqslant2b>2|K|$, we have
$$[\mathbb Z_N:K]\geqslant3, \ \
|B\cap K|=b-3, \ \ |B\setminus K|=3,$$
contrary to the preceding subgroup exclusion.
Thus,
$$|U_{m-1}(X)|\geqslant2b-5.$$
Since $e,f\notin X$, we have
$$e+U_{m-1}(X)\subseteq U_m(B), \ \
f+U_{m-1}(X)\subseteq U_m(B).$$
Together with
$$e+U_{m-1}(X)=(e-f)+(f+U_{m-1}(X))
\subseteq U_m(B)+e-f,$$
we have
$$e+U_{m-1}(X)\subseteq
U_m(B)\cap(U_m(B)+e-f).$$
Hence, for every $d\in\mathcal D$, we have
$$
\begin{aligned}
|(U_m(B)+d)\setminus U_m(B)|
&=|U_m(B)|-|U_m(B)\cap(U_m(B)+d)|\\
&\leqslant(2b-2)-(2b-5)=3.
\end{aligned}
$$

Since
$$13\leqslant2b-5\leqslant|U_m(B)|
\leqslant2b-2\leqslant N-6$$
and
$$\mathcal D=-\mathcal D, \ \ 0\notin\mathcal D, \ \
\langle\mathcal D\rangle=\mathbb Z_N, \ \ |\mathcal D|\geqslant11,$$
by Lemma \ref{lem:olson-boundary}, there exists $d\in\mathcal D$ such that
$$
\begin{aligned}
3
&\geqslant|(U_m(B)+d)\setminus U_m(B)|\\
&\geqslant
\min\left\{
\frac{|U_m(B)|+1}{2},
\frac{N-|U_m(B)|+1}{2},
\frac{|\mathcal D|+2}{4}
\right\}\\
&\geqslant\min\left\{7,\frac72,\frac{13}{4}\right\}
=\frac{13}{4}>3,
\end{aligned}
$$
a contradiction.

This completes the proofs of Theorems
\ref{thm:r4-coarse-cyclic-all-orders} and
\ref{thm:r4-coarse-cyclic-inverse}.
\end{proof}

\begin{thm}\label{thm:strict-critical-cyclic-surplus}
Let $p\geqslant 2$, $N\geqslant\max\{6, 2p\}$ and $L\geqslant 2N+2$.
Let $J\subseteq[1, L-1]$ be a set of $N$ integers such that $B=\{0\}\cup\pi_L(J)$ generates $\mathbb Z_L$.
If $|p^{\wedge}J|<2p(N-p)-N+2$, then
$$|p^{\wedge}B\cup(p+1)^{\wedge}B|\geqslant2N+1.$$
\end{thm}

\begin{proof}
Since $J\subseteq[1,L-1]$, the map $\pi_L$ is injective on $J$
and $0\notin\pi_L(J)$. Thus,
$$0\in B, \ \ |B|=N+1\geqslant7, \ \
\langle B\rangle=\mathbb Z_L.$$
Since
$$2\leqslant p\leqslant
\left\lfloor\frac N2\right\rfloor
=\left\lfloor\frac{|B|-1}{2}\right\rfloor,
\ \ L\geqslant2N+2=2|B|,$$
by Theorem \ref{thm:r4-coarse-cyclic-all-orders}, we have
$$|U_p(B)|\geqslant2|B|-2=2N.$$

Suppose that $|U_p(B)|=2N$.
By Theorem \ref{thm:r4-coarse-cyclic-inverse}, we have
$$B=H\mathbin{\dot\cup}\{x\}, \ \ |H|=N, \ \ x\notin H.$$
Put $d=L/N$. Then
$$d\in\mathbb Z, \ \ H=\pi_L(d\cdot[0,N-1]), \ \
\pi_L(J)=(H\setminus\{0\})\mathbin{\dot\cup}\{x\}.$$
Since $\pi_L$ is injective on $[1,L-1]$, we have
$$J=C\mathbin{\dot\cup}\{x'\}, \ \
C=d\cdot[1,N-1], \ \ \pi_L(x')=x, \ \ d\nmid x'.$$
Thus,
$$p^{\wedge}C\subseteq d\mathbb Z, \ \
x'+(p-1)^{\wedge}C\subseteq x'+d\mathbb Z$$
and
$$p^{\wedge}J
=p^{\wedge}C\mathbin{\dot\cup}
\bigl(x'+(p-1)^{\wedge}C\bigr).$$
Since $1\leqslant p-1<p\leqslant N-1$, by
\eqref{eq:unit-interval-restricted-ranks}, we have
$$
\begin{aligned}
|p^{\wedge}J|
&=|p^{\wedge}C|+|(p-1)^{\wedge}C|\\
&=p(N-p-1)+1+(p-1)(N-p)+1\\
&=2p(N-p)-N+2,
\end{aligned}
$$
contrary to $|p^{\wedge}J|<2p(N-p)-N+2$.
Hence,
$$|p^{\wedge}B\cup(p+1)^{\wedge}B|
=|U_p(B)|\geqslant2N+1.$$

This completes the proof of Theorem \ref{thm:strict-critical-cyclic-surplus}.
\end{proof}

\section{Dense integer sets and their projections}\label{sec:dense-integer-sets}
\begin{lem}\label{lem:dense-strong-endpoint-deletion}
Let $J\subseteq[0, L]$ be a set of $n$ integers such that $0, L\in J$ and $B=J\setminus\{0, L\}$.
Put $g=L-n+1$.

(1) If $J$ is primitive, $n\geqslant7$ and $3\leqslant m\leqslant n-3$, then
$$e_m(J)+D_m(J)\geqslant\min\{g,n-3\}.$$

(2) If $n\geqslant7$ and $L\leqslant2n-4$, then for every $3\leqslant m\leqslant n-3$, we have
$$\pi_L(m^{\wedge}J)=\mathbb Z_L,$$
$$e_m(J)+D_m(J)\geqslant g+e_{m-1}(B)+D_{m-1}(B)\geqslant g,
\ \ e_m(J)\leqslant e_{m-1}(B).$$

(3) If $n\geqslant4$ and $L\leqslant2n-5$, the conclusions of (2) hold for every $2\leqslant m\leqslant n-2$. Moreover,
$$D_2(J)=D_{n-2}(J)\geqslant0.$$
\end{lem}

\begin{proof}
Without loss of generality, we assume that
$$2\leqslant m\leqslant\left\lfloor\frac n2\right\rfloor.$$
Moreover, $L\geqslant n-1\geqslant3$.
Put $C=\{0\}\cup\pi_L(B)$. Then
$$|C|=n-1, \ \ \pi_L(m^{\wedge}J)=U_{m-1}(C).$$
If $J$ is primitive, then $\langle C\rangle=\mathbb Z_L$.
Under either density hypothesis, if $\langle C\rangle\neq\mathbb Z_L$, then
$$n-1=|C|\leqslant|\langle C\rangle|
\leqslant\frac L2\leqslant n-2,$$
a contradiction. Thus, $\langle C\rangle=\mathbb Z_L$.

By Corollary \ref{lem:exact-endpoint-transfer}, we have
$$|m^{\wedge}J|-|(m-1)^{\wedge}B|\geqslant|U_{m-1}(C)|$$
and
$$
\begin{aligned}
e_m(J)
&=e_{m-1}(B)+L-|m^{\wedge}J|+|(m-1)^{\wedge}B|\\
&\leqslant e_{m-1}(B)+L-|U_{m-1}(C)|,\\
e_m(J)+D_m(J)
&=e_{m-1}(B)+D_{m-1}(B)\\
&\qquad+|m^{\wedge}J|-|(m-1)^{\wedge}B|-n+1\\
&\geqslant e_{m-1}(B)+D_{m-1}(B)+|U_{m-1}(C)|-n+1.
\end{aligned}
$$
By Theorem \ref{thm:nathanson-restricted-minimum}, we have
$$e_{m-1}(B)+D_{m-1}(B)\geqslant0.$$

If $n\geqslant7$ and $m\geqslant3$, then $0\in C$, $|C|\geqslant6$ and
$$2\leqslant m-1\leqslant\left\lfloor\frac{n-2}{2}\right\rfloor
=\left\lfloor\frac{|C|-1}{2}\right\rfloor.$$
By Theorem \ref{thm:r4-coarse-cyclic-all-orders}, we have
\begin{equation}\label{e1}
|U_{m-1}(C)|\geqslant\min\{L,2n-4\}.
\end{equation}
Thus,
$$
\begin{aligned}
e_m(J)+D_m(J)
&\geqslant|U_{m-1}(C)|-n+1\\
&\geqslant\min\{L,2n-4\}-n+1\\
&=\min\{L-n+1,n-3\}\\
&=\min\{g,n-3\}.
\end{aligned}
$$
This proves (1).
Under the hypotheses of (2), \eqref{e1} gives
$$U_{m-1}(C)=\mathbb Z_L.$$

Under the hypotheses of (3), the remaining cases are
$m=2$ or $(n,m)=(6,3)$.
For every $z\in\mathbb Z_L$, we have
$$|C\cap(z-C)|\geqslant2|C|-L\geqslant3, \ \
|\{c\in\mathbb Z_L:2c=z\}|\leqslant2.$$
Thus,
$$
\begin{aligned}
\left|\bigl(C\cap(z-C)\bigr)
\setminus\{c\in\mathbb Z_L:2c=z\}\right|
&\geqslant|C\cap(z-C)|-2\\
&\geqslant1.
\end{aligned}
$$
Hence, $z\in2^{\wedge}C$ and $2^{\wedge}C=\mathbb Z_L$.
Since $m-1\in\{1,2\}$, we have
$$\mathbb Z_L=2^{\wedge}C
\subseteq U_{m-1}(C)\subseteq\mathbb Z_L.$$

Under either density hypothesis, we have
$$\pi_L(m^{\wedge}J)=U_{m-1}(C)=\mathbb Z_L.$$
Substituting $|U_{m-1}(C)|=L$ into the preceding inequalities, we have
$$
\begin{aligned}
e_m(J)&\leqslant e_{m-1}(B),\\
e_m(J)+D_m(J)
&\geqslant g+e_{m-1}(B)+D_{m-1}(B)\geqslant g.
\end{aligned}
$$
For $m=2$, since $l(B)\leqslant L-2$, we have
$$D_2(J)\geqslant g-e_2(J)\geqslant g-e_1(B)=L-l(B)-2\geqslant 0. $$
Together with $D_{n-2}(J)=D_2(J)$, this proves (3).

This completes the proof of Lemma \ref{lem:dense-strong-endpoint-deletion}.
\end{proof}

\begin{thm}\label{thm:dense-all-orders}
Let $A\subseteq[0, L]$ be a set of $k\geqslant8$ integers such that $0, L\in A$ and $3\leqslant h\leqslant k-3$.
If $L\leqslant2k-4$, then
$$|h^{\wedge}A|\geqslant L+h(k-h)-k+2.$$
\end{thm}

\begin{proof}
By Lemma \ref{lem:dense-strong-endpoint-deletion} (2), we have
$$e_h(A)+D_h(A)\geqslant L-k+1,$$
thus,
$$|h^{\wedge}A|
=h(k-h)+1+e_h(A)+D_h(A)
\geqslant L+h(k-h)-k+2.$$

This completes the proof of Theorem \ref{thm:dense-all-orders}.
\end{proof}

\begin{lem}\label{lem:dense-cores}
Let $E\subseteq[0, l]$ be a set of $n$ integers such that $0, l\in E$ and $g=l-n+1$.
If $0\leqslant g\leqslant n-4$, then
$$[2g+1, 2l-2g-1]\subseteq2^{\wedge}E, \ \
[2g+3, 3l-2g-3]\subseteq3^{\wedge}E.$$
\end{lem}

\begin{proof}
Since $g\leqslant n-4$, we have
$$l=n+g-1\geqslant2g+3.$$
Fix $r\in\{2,3\}$ and $2g+2r-3\leqslant u\leqslant l$.
The pairs
$$\{a,u-a\}, \ \
r-2\leqslant a\leqslant\left\lfloor\frac{u-1}{2}\right\rfloor,$$
are pairwise disjoint subsets of $[0, l]$.

Suppose that $u\notin r^{\wedge}E$. Then
$$|\{a,u-a\}\setminus E|\geqslant1.$$
Thus,
$$
\begin{aligned}
g
&=|[0,l]\setminus E|
\geqslant\left|
\bigcup_{a=r-2}^{\lfloor(u-1)/2\rfloor}
\bigl(\{a,u-a\}\setminus E\bigr)\right|\\
&=\sum_{a=r-2}^{\lfloor(u-1)/2\rfloor}
|\{a,u-a\}\setminus E|
\geqslant\left\lfloor\frac{u-1}{2}\right\rfloor-r+3
\geqslant g+1,
\end{aligned}
$$
a contradiction. Hence,
$$[2g+2r-3, l]\subseteq r^{\wedge}E.$$

Since $l-E$ satisfies the same hypotheses and
$r^{\wedge}(l-E)=rl-r^{\wedge}E$, we have
$$[(r-1)l, rl-2g-2r+3]\subseteq r^{\wedge}E.$$
Taking $r=2$ and $r=3$, respectively, we have
$$[2g+1, 2l-2g-1]\subseteq2^{\wedge}E,$$
$$[2g+3, l]\cup[2l, 3l-2g-3]\subseteq3^{\wedge}E.$$

Since $n\geqslant4$ and $l\leqslant2n-5$, by
Lemma \ref{lem:dense-strong-endpoint-deletion} (3), we have
$$\pi_l(2^{\wedge}E)=\mathbb Z_l.$$
Together with $2^{\wedge}E\subseteq[1, 2l-1]$, we have
$$[l+1, 2l-1]=(l+(2^{\wedge}E\cap[1, l-1]))\cup(2^{\wedge}E\cap[l+1, 2l-1]). $$
Since $E\subseteq[0, l]$ and $0, l\in E$, we have
$$
\begin{aligned}
l+(2^{\wedge}E\cap[1, l-1])
&\subseteq l+2^{\wedge}(E\setminus\{l\})
\subseteq3^{\wedge}E,\\
2^{\wedge}E\cap[l+1, 2l-1]
&\subseteq0+2^{\wedge}(E\setminus\{0\})
\subseteq3^{\wedge}E,
\end{aligned}
$$
thus,
$$[l+1, 2l-1]\subseteq3^{\wedge}E.$$
Hence,
$$[2g+3, 3l-2g-3]\subseteq3^{\wedge}E.$$

This completes the proof of Lemma \ref{lem:dense-cores}.
\end{proof}

\begin{lem}
\label{lem:dense-fourth-core}
Let $n\geqslant5$, $0\leqslant g\leqslant n-4$, $l=n+g-1$ and
$$E=\{a_0,a_1,\ldots,a_{n-1}\}\subseteq\mathbb Z, \ \
0=a_0<a_1<\cdots<a_{n-1}=l.$$
For $3\leqslant r\leqslant n-2$, put
$$Q_r=\left[
\sum_{i=0}^{r-3}a_i+2g+2r-3,
\sum_{i=n-r+2}^{n-1}a_i+2l-2g-2r+3
\right].$$
If $g\leqslant n-r-1$, then
$$Q_r\subseteq r^{\wedge}E, \ \
|Q_r|\geqslant r(n-r)+1-g.$$
\end{lem}

\begin{proof}
For $1\leqslant k\leqslant n-1$, by
$a_{n-k+i-1}-a_i\geqslant n-k-1$, we have
$$l(k^{\wedge}E)
=l+\sum_{i=1}^{k-1}(a_{n-k+i-1}-a_i)
\geqslant k(n-k)+g.$$
Write $Q_r=[u,v]$. Taking $k=r-2$, we have
$$v-u+1=l((r-2)^{\wedge}E)+2l-4g-4r+7
\geqslant r(n-r)+1-g>0,$$
thus,
$$|Q_r|\geqslant r(n-r)+1-g.$$

Put $N=n-r+3$.
For $0\leqslant i\leqslant r-3$, put
$$E_i=\{a_i,\ldots,a_{i+N-1}\}, \ \
s_i=\sum_{a\in E\setminus E_i}a.$$
Then
$$|E_i|=N\geqslant5, \ \ |E\setminus E_i|=r-3$$
and
$$0\leqslant l(E_i)-N+1\leqslant g
\leqslant n-r-1=N-4.$$
Put
$$[u_i,v_i]
=s_i+[a_i+2a_{i+N-1}-2N+5,\,
2a_i+a_{i+N-1}+2N-5].$$
By Lemma \ref{lem:dense-cores} applied to $E_i-a_i$, we have
$$[u_i,v_i]\subseteq s_i+3^{\wedge}E_i
\subseteq r^{\wedge}E.$$

For $0\leqslant i<r-3$, since
$s_{i+1}-s_i=a_i-a_{i+N}$ and
$$
\begin{aligned}
a_{i+N-1}-a_{i+1}
&\leqslant N-2+g,\\
(a_{i+1}-a_i)+(a_{i+N}-a_{i+N-1})
&\leqslant g+2,
\end{aligned}
$$
we have
$$
\begin{aligned}
v_{i+1}-u_i+1
&=4N-9-2(a_{i+N-1}-a_{i+1})\\
&\geqslant2N-5-2g\geqslant3,\\
v_i-u_{i+1}+1
&=4N-9-(a_{i+1}-a_i)-(a_{i+N}-a_{i+N-1})\\
&\geqslant4N-11-g>0.
\end{aligned}
$$
Thus,
$$[u_i,v_i]\cap[u_{i+1},v_{i+1}]\neq\emptyset.$$
Since
$$
\begin{aligned}
u_{r-3}
&=\sum_{j=0}^{r-3}a_j+2g+2r-3=u,\\
v_0
&=\sum_{j=n-r+2}^{n-1}a_j+2l-2g-2r+3=v,
\end{aligned}
$$
we have
$$Q_r=[u,v]\subseteq\bigcup_{i=0}^{r-3}[u_i,v_i]
\subseteq r^{\wedge}E.$$

This completes the proof of Lemma \ref{lem:dense-fourth-core}.
\end{proof}

\begin{cor}\label{cor:two-fringe-hole-bound}
Let $E$ be a set of $n\geqslant7$ integers and put $g=l(E)-n+1$.
If $0\leqslant g<\frac n2$ and $3\leqslant r\leqslant\frac n2$, then
$$e_r(E)\leqslant2g.$$
\end{cor}

\begin{proof}
By translation, we may assume that
$$E=\{a_0,\ldots,a_{n-1}\}, \ \
0=a_0<\cdots<a_{n-1}=l.$$
Put $H=[0,l]\setminus E$. Then $|H|=g$.
Since $g<\frac n2\leqslant n-r$ and $n\geqslant7$, we have
$$g\leqslant n-r-1, \ \
g\leqslant\left\lfloor\frac{n-1}{2}\right\rfloor
\leqslant n-4.$$
By Lemma \ref{lem:dense-fourth-core}, we have
$$\emptyset\neq Q_r\subseteq r^{\wedge}E.$$

Put
$$s=\sum_{i=0}^{r-2}a_i, \ \
t=\sum_{i=n-r+1}^{n-1}a_i.$$
Then
$$\min r^{\wedge}E=s+a_{r-1}, \ \
\max r^{\wedge}E=t+a_{n-r}.$$
Since $a_{r-2}\geqslant r-2$ and $a_{n-r+1}\leqslant l-r+2$, we have
$$
\begin{aligned}
\min Q_r-s
&=2g+2r-3-a_{r-2}
\leqslant2g+r-1\leqslant l,\\
\max Q_r-t
&=2l-2g-2r+3-a_{n-r+1}
\geqslant l-2g-r+1\geqslant0.
\end{aligned}
$$
Thus,
$$[\min r^{\wedge}E,\max r^{\wedge}E]
\subseteq\bigl(s+[a_{r-1},l]\bigr)
\cup Q_r\cup\bigl(t+[0,a_{n-r}]\bigr).$$
Together with $Q_r\subseteq r^{\wedge}E$ and
$$\bigl(s+(E\cap[a_{r-1},l])\bigr)
\cup\bigl(t+(E\cap[0,a_{n-r}])\bigr)
\subseteq r^{\wedge}E,$$
we have
$$
\begin{aligned}
&[\min r^{\wedge}E,\max r^{\wedge}E]\setminus r^{\wedge}E\\
&\subseteq
\bigl(s+([a_{r-1},l]\setminus E)\bigr)
\cup\bigl(t+([0,a_{n-r}]\setminus E)\bigr)\\
&\subseteq(s+H)\cup(t+H),
\end{aligned}
$$
and hence
$$e_r(E)\leqslant|s+H|+|t+H|=2g.$$

This completes the proof of Corollary \ref{cor:two-fringe-hole-bound}.
\end{proof}

\begin{lem}\label{lem:dense-low-adjacent-projection}
Let $m\geqslant 2$, $n\geqslant\max\{5, 2m\}$ and $L\geqslant 2$ be integers.
Let
$$I=\{x_0, \ldots, x_{n-1}\}\subseteq [1, L-1], \ \ x_0<\cdots<x_{n-1}.$$
Put $g=l(I)-n+1$.
If $g\leqslant n-m-2$, then
$$\left|\pi_L((m-1)^{\wedge}I\cup m^{\wedge}I\cup(m+1)^{\wedge}I)\right|\geqslant \min\{L, (m+1)(n-m+1)-g\}. $$
\end{lem}

\begin{proof}
For $r\in\{m-1, m, m+1\}$ with $r\geqslant2$, put
$$J_r=\left[\sum_{i=0}^{r-3}x_i+2x_0+2g+2r-3, \sum_{i=n-r+2}^{n-1}x_i+2x_{n-1}-2g-2r+3\right],$$
where an empty sum equals zero.
Since $m+1\leqslant n-2$ and
$$g\leqslant n-m-2\leqslant n-r-1, \ \ g\leqslant n-4,$$
by Lemmas \ref{lem:dense-cores} and \ref{lem:dense-fourth-core},
applied to $I-x_0$, we have
$$J_r\subseteq r^{\wedge}I$$
and
$$|J_r|\geqslant r(n-r)+1-g, \ \ r\geqslant 3.$$
If $r=2$, then
$$|J_2|=2n-2g-3\geqslant 2m+1.$$
If $r\geqslant3$, then $r\leqslant m+1\leqslant n-2$, thus,
$$r(n-r)=2(n-2)+(r-2)(n-r-2)\geqslant 2(n-2).$$
Since $0\leqslant g\leqslant n-m-2$, we have
$$|J_r|\geqslant r(n-r)+1-g\geqslant 2n-g-3\geqslant 2n-2g-3\geqslant 2m+1.$$
Hence,
$$|J_r|\geqslant2m+1.$$

Put
$$S=(m-1)^{\wedge}I\cup m^{\wedge}I\cup(m+1)^{\wedge}I,$$
$$a=\min J_{m+1}-\min J_m=x_{m-2}+2,\ \ b=\max J_{m+1}-\max J_m=x_{n-m+1}-2.$$
Put $I'=L-I$. Then
$$I'\subseteq[1,L-1],\ \ |I'|=n,\ \ l(I')=l(I).$$
By the definition of $J_r$, we have
$$J_m(I')=mL-J_m,\ \  J_{m+1}(I')=(m+1)L-J_{m+1},$$
thus,
$$
\begin{aligned}
\min J_{m+1}(I')-\min J_m(I')&=L-b,\\
L-\max J_{m+1}(I')+\max J_m(I')&=a.
\end{aligned}
$$
Moreover,
$$\pi_L\bigl((m-1)^{\wedge}I'\cup m^{\wedge}I'
\cup(m+1)^{\wedge}I'\bigr)=-\pi_L(S).$$
Thus, by reflection, we may assume that $a\leqslant L-b$.

If $a\geqslant2m+1$, put $K=J_m$.
If $a<2m+1$ and $m\geqslant3$, then
$$
\begin{aligned}
m&\leqslant\min J_m-\min J_{m-1}
=x_{m-3}+2\leqslant a-1<|J_{m-1}|,\\
\max J_m-\max J_{m-1}
&=x_{n-m+2}-2\geqslant b+1>0.
\end{aligned}
$$
Thus,
$$K=J_{m-1}\cup J_m=[\min J_{m-1},\max J_m]\subseteq S.$$
Moreover,
$$\min J_{m+1}-\min K=a+x_{m-3}+2\geqslant2m+1.$$
If also $L-b<2m+1$, then
$$
\begin{aligned}
|K|&=|J_{m-1}|+x_{n-m+2}-2\\
&\geqslant2m+1+b+1>L,
\end{aligned}
$$
and $\pi_L(S)=\mathbb Z_L$.

If $a<2m+1$ and $m=2$, then $x_0\in\{1,2\}$.
Put $F=\{0\}\cup I$. Since $|F|=n+1$ and
$$0\leqslant l(F)-|F|+1=g+x_0-1\leqslant n-3,$$
by Lemma \ref{lem:dense-cores}, we have
$$[2x_0+2g-1,2n-1]\subseteq2^{\wedge}F=I\cup2^{\wedge}I.$$
Since
$$2x_0+2g+1\leqslant2n-1\leqslant2x_0+2n-3,$$
we have
$$K=[\min J_2-2,\max J_2]\subseteq S,\ \
\min J_3-\min K=a+2\geqslant5.$$
If also $L-b<5$, then $L-x_{n-1}\leqslant2$, and
$$L=l(I)+x_0+L-x_{n-1}\leqslant n+g+3\leqslant2n-1.$$
By Lemmas \ref{lem:dense-strong-endpoint-deletion} (3) and
\ref{lem:endpoint-lifting}, applied to $\{0,L\}\cup I$, we have
$$\pi_L(S)=\pi_L(3^{\wedge}(\{0,L\}\cup I))=\mathbb Z_L.$$

In the remaining cases,
$$K\subseteq S,\ \  \max K=\max J_m,\ \  |K|\geqslant 2m+1,$$
$$\min J_{m+1}-\min K\geqslant 2m+1,\ \  L-b\geqslant 2m+1.$$
If $|J_{m+1}|\geqslant L$, then $\pi_L(S)=\mathbb Z_L$.
Otherwise, put $R=K-\min J_{m+1}+L$. Then
$$
\begin{aligned}
|R\cap[|J_{m+1}|,L-1]|
&=\min\{\max R,L-1\}-\max\{\min R,|J_{m+1}|\}+1\\
&=\min\{L-|J_{m+1}|,\min J_{m+1}-\min K,L-b,|K|\}\\
&\geqslant\min\{L-|J_{m+1}|,2m+1\}>0.
\end{aligned}
$$
Since $J_{m+1}-\min J_{m+1}=[0,|J_{m+1}|-1]$ and
$\pi_L(R)\subseteq\pi_L(S-\min J_{m+1})$, we have
$$
\begin{aligned}
|\pi_L(S)|
&=|\pi_L(S-\min J_{m+1})|\\
&\geqslant\left|\pi_L\bigl([0,|J_{m+1}|-1]
\cup(R\cap[|J_{m+1}|,L-1])\bigr)\right|\\
&=|J_{m+1}|+|R\cap[|J_{m+1}|,L-1]|\\
&\geqslant|J_{m+1}|+\min\{L-|J_{m+1}|,2m+1\}\\
&=\min\{L,|J_{m+1}|+2m+1\}\\
&\geqslant\min\{L,(m+1)(n-m+1)-g\}.
\end{aligned}
$$

This completes the proof of Lemma \ref{lem:dense-low-adjacent-projection}.
\end{proof}

\begin{thm}\label{thm:r3-uniform-dense-small-excess}
Let $m\geqslant 3$, $n\geqslant\max\{7, 2m\}$ and $L\geqslant 2$ be integers.
Let $I\subseteq[1, L-1]$ be a primitive set of $n$ integers and put $g=l(I)-n+1$.
If $e_m(I)+D_m(I)<\frac{n}{2}$, then
$$0\leqslant g\leqslant e_m(I)+D_m(I)$$
and
$$\left|\pi_L\bigl((m-1)^{\wedge}I\cup m^{\wedge}I\cup(m+1)^{\wedge}I\bigr)\right|\geqslant\min\{L, (m+1)(n-m+1)-g\}, $$
$$
\begin{aligned}
&\left|\pi_L\bigl((m-1)^{\wedge}I\cup m^{\wedge}I\cup(m+1)^{\wedge}I\bigr)\right|
+e_m(I)+D_m(I)\\
&\geqslant\min\{L, (m+1)(n-m+1)\}.
\end{aligned}
$$
\end{thm}

\begin{proof}
Since $3\leqslant m\leqslant n-3$, by Lemma
\ref{lem:dense-strong-endpoint-deletion} (1) applied to $I-\min I$, we have
$$e_m(I)+D_m(I)\geqslant\min\{g, n-3\}.$$
Since $e_m(I)+D_m(I)<\frac{n}{2}\leqslant n-3$ and $g$ is an integer, we have
$$0\leqslant g\leqslant e_m(I)+D_m(I)<\frac{n}{2}, \ \ g\leqslant n-4,$$
thus,
$$l(I)=n+g-1\leqslant2n-5.$$

Put $S=(m-1)^{\wedge}I\cup m^{\wedge}I\cup(m+1)^{\wedge}I$.
If $g\leqslant n-m-2$, then by Lemma \ref{lem:dense-low-adjacent-projection}, we have
$$|\pi_L(S)|\geqslant\min\{L, (m+1)(n-m+1)-g\}.$$

Suppose that $g\geqslant n-m-1$.
Since $n\geqslant2m$, we have
$$2g\geqslant2n-2m-2\geqslant n-2,$$
thus,
$$0\leqslant e_m(I)+D_m(I)-g<\frac{n}{2}-g\leqslant1.$$
Since $e_m(I)+D_m(I)-g$ is an integer, we have
$$e_m(I)+D_m(I)=g.$$

Write
$$I=\{x_0, \ldots, x_{n-1}\}, \ \ x_0<\cdots<x_{n-1},$$
and put $B=I\setminus\{x_0, x_{n-1}\}$.
By Lemma \ref{lem:dense-strong-endpoint-deletion} (3) applied to $I-x_0$, we have
$$0\leqslant e_{m-1}(B)+D_{m-1}(B)
\leqslant e_m(I)+D_m(I)-g=0.$$
Thus,
$$|(m-1)^{\wedge}B|=(m-1)(n-m-1)+1.$$
Since $|B|=n-2\geqslant5$ and $2\leqslant m-1\leqslant|B|-2$,
by Theorem \ref{thm:nathanson-restricted-minimum}, we have
$$B=\{x_1, x_1+d(B), \ldots, x_1+(n-3)d(B)\}.$$
Since
$$(n-3)d(B)=l(B)\leqslant l(I)-2=n+g-3\leqslant2n-7<2(n-3),$$
we have
$$d(B)=1, \ \ B=[x_1, x_{n-2}].$$

For $r\in\{m-1, m, m+1\}$, we have $2\leqslant r\leqslant n-2$.
Since $l(I)\leqslant2n-5$, by Lemma
\ref{lem:dense-strong-endpoint-deletion} (3), we have
$$0\leqslant e_r(I)\leqslant e_{r-1}(B)=0,$$
thus,
$$r^{\wedge}I=[\min r^{\wedge}I, \max r^{\wedge}I].$$

Put $I'=[x_1-1, x_{n-2}+1]$. Then
$$I'\subseteq[1, L-1], \ \ |I'|=n, \ \ l(I')-n+1=0.$$
For every $r\in\{m-1, m, m+1\}$, we have
$$
\begin{aligned}
r^{\wedge}I'
&\subseteq[\min r^{\wedge}I+x_1-x_0-1,
\max r^{\wedge}I+x_{n-2}+1-x_{n-1}]\\
&\subseteq[\min r^{\wedge}I, \max r^{\wedge}I]=r^{\wedge}I.
\end{aligned}
$$
By Lemma \ref{lem:dense-low-adjacent-projection} applied to $I'$, we have
$$
\begin{aligned}
|\pi_L(S)|
&\geqslant
\left|\pi_L\bigl((m-1)^{\wedge}I'\cup m^{\wedge}I'
\cup(m+1)^{\wedge}I'\bigr)\right|\\
&\geqslant\min\{L, (m+1)(n-m+1)\}.
\end{aligned}
$$
Thus, the first inequality holds in every case.
Since $e_m(I)+D_m(I)\geqslant g\geqslant0$, we have
$$
\begin{aligned}
|\pi_L(S)|+e_m(I)+D_m(I)
&\geqslant\min\{L, (m+1)(n-m+1)-g\}+e_m(I)+D_m(I)\\
&\geqslant\min\{L, (m+1)(n-m+1)\}.
\end{aligned}
$$

This completes the proof of Theorem \ref{thm:r3-uniform-dense-small-excess}.
\end{proof}

\section{Adding the smallest and largest elements}\label{sec:cardinality-transition}\label{sec:reserve-transition}
\begin{lem}\label{lem:stratified-incomplete-projection}
Let $p\geqslant 3$ and $N\geqslant\max\{7, 2p\}$. Let
$$Y\subseteq[1, L-1], \ \ |Y|=N, \ \ X=\{0, L\}\cup Y.$$
Suppose that $Y$ and $X$ are primitive and critical at orders $p$ and $p+1$, respectively, and
$$D_p(Y)\geqslant0, \ \ D_{p+1}(X)<\frac{N+1}{2}, \ \ a_{p+1}(X)\geqslant1.$$
Put
$$
\begin{aligned}
v&=\left|\pi_L\bigl((p-1)^{\wedge}Y\cup p^{\wedge}Y\cup(p+1)^{\wedge}Y\bigr)\right|,\\
\rho&=|(p+1)^{\wedge}X|-|p^{\wedge}Y|-v\geqslant0.
\end{aligned}
$$
If $v<L$, then $L\geqslant2N+2$ and
$$
\begin{aligned}
a_p(Y)&\geqslant a_{p+1}(X)+\rho+1,\\
v&=2N+a_p(Y)-a_{p+1}(X)-\rho,\\
L&=3N+D_p(Y)+2a_p(Y)-2a_{p+1}(X)-D_{p+1}(X)-1.
\end{aligned}
$$
Moreover,
$$e_p(Y)+D_p(Y)\geqslant\frac{N}{2}, \ \ D_p(Y)<\frac{N-1}{2},$$
$$e_p(Y)\geqslant N+D_p(Y)+a_p(Y)-a_{p+1}(X)-D_{p+1}(X)-1>2D_p(Y).$$
If also $2L\geqslant3v$, then
$$4D_p(Y)-2D_{p+1}(X)+a_p(Y)+4\rho-1\leqslant N.$$
\end{lem}

\begin{proof}
Put $C=\{0\}\cup\pi_L(Y)$. Since $X$ is primitive, we have
$$|C|=N+1, \ \ \langle C\rangle=\mathbb Z_L.$$
Since $1\leqslant p\leqslant N-1$, by Lemma \ref{lem:endpoint-lifting} and Corollary \ref{lem:exact-endpoint-transfer}, we have
$$
\begin{aligned}
v&=|p^{\wedge}C\cup(p+1)^{\wedge}C|,\\
D_{p+1}(X)&=D_p(Y)+2(v+\rho)-L-N-1,\\
e_{p+1}(X)&=e_p(Y)+L-v-\rho.
\end{aligned}
$$
Since $N+1\geqslant8$ and $2\leqslant p\leqslant\left\lfloor\frac{N}{2}\right\rfloor$, by Theorem \ref{thm:r4-coarse-cyclic-all-orders}, we have
$$v\geqslant\min\{L, 2N\}.$$
Together with $v<L$, we have $v\geqslant2N$ and $L\geqslant2N+1$.
If $L=2N+1$, then
$$D_{p+1}(X)\geqslant D_p(Y)+4N-L-N-1\geqslant N-2\geqslant\frac{N+1}{2},$$
a contradiction. Thus,
$$L\geqslant 2N+2. $$
Since $Y$ is critical, by Theorem \ref{thm:strict-critical-cyclic-surplus}, we have $v\geqslant 2N+1$.
Since
$$
\begin{aligned}
|p^{\wedge}Y|&=2p(N-p)-N+1-a_p(Y),\\
|(p+1)^{\wedge}X|&=2p(N-p)+N+1-a_{p+1}(X),
\end{aligned}
$$
we have
$$
\begin{aligned}
v&=|(p+1)^{\wedge}X|-|p^{\wedge}Y|-\rho
=2N+a_p(Y)-a_{p+1}(X)-\rho,\\
a_p(Y)&=v-2N+a_{p+1}(X)+\rho\geqslant a_{p+1}(X)+\rho+1,\\
L&=D_p(Y)+2(v+\rho)-N-1-D_{p+1}(X)\\
&=3N+D_p(Y)+2a_p(Y)-2a_{p+1}(X)-D_{p+1}(X)-1.
\end{aligned}
$$

Suppose that $e_p(Y)+D_p(Y)<\frac{N}{2}$.
By Theorem \ref{thm:r3-uniform-dense-small-excess}, we have
$$l(Y)-N+1\leqslant e_p(Y)+D_p(Y), \ \ v\geqslant\min\{L, p(N-p)+2N-l(Y)\}.$$
Since $v<L$, we have $v\geqslant p(N-p)+2N-l(Y)$. Thus,
$$
\begin{aligned}
1&\leqslant a_{p+1}(X)+\rho\\
&=p(N-p)+N-e_p(Y)-D_p(Y)-v\\
&\leqslant l(Y)-N-e_p(Y)-D_p(Y)\leqslant-1,
\end{aligned}
$$
a contradiction. Hence, $e_p(Y)+D_p(Y)\geqslant\frac{N}{2}$.

Since $2(e_p(Y)+D_p(Y))\geqslant N$ and $a_{p+1}(X)\geqslant1$, we have
$$
\begin{aligned}
|(p+1)^{\wedge}X|-2|p^{\wedge}Y|
&=N-1-a_{p+1}(X)-2(e_p(Y)+D_p(Y))\\
&\leqslant-1-a_{p+1}(X)<0.
\end{aligned}
$$
Put $u=\min p^{\wedge}Y$ and $T=p^{\wedge}Y-u$. Then
$$u+\bigl(T\cup(L+T)\bigr)\subseteq(p+1)^{\wedge}X, \ \ |T\cup(L+T)|\leqslant|(p+1)^{\wedge}X|<2|T|.$$
Thus, $T\cap(L+T)\neq\emptyset$. Since $\min T=0$, we have
$$\emptyset\neq T\cap(L+T)\subseteq[L, l(T)], \ \ l(T)\geqslant L.$$
Since $[0, L+l(T)]=[0, l(T)]\cup(L+[0, l(T)])$, by Lemma \ref{lem:endpoint-lifting}, we have
$$
\begin{aligned}
e_{p+1}(X)
&=\left|[0, L+l(T)]\setminus\bigl((p+1)^{\wedge}X-u\bigr)\right|\\
&\leqslant\left|\bigl([0, l(T)]\setminus T\bigr)
\cup\bigl(L+([0, l(T)]\setminus T)\bigr)\right|\\
&\leqslant2|[0, l(T)]\setminus T|=2e_p(Y).
\end{aligned}
$$
Together with $e_{p+1}(X)=e_p(Y)+L-v-\rho$, we have
$$e_p(Y)\geqslant L-v-\rho=N+D_p(Y)+a_p(Y)-a_{p+1}(X)-D_{p+1}(X)-1.$$

Put
$$U=\pi_L\bigl((p-1)^{\wedge}Y\cup p^{\wedge}Y\cup(p+1)^{\wedge}Y\bigr)-\pi_L(u).$$
Then $\pi_L(T)\subseteq U$ and $|U|=v$.
Identify $U$ with its representatives in $[0, L-1]$.
Since $l(T)\geqslant L$, write
$$l(T)+1=kL+z, \ \ k\geqslant1, \ \ 0\leqslant z<L, \ \ c=|U\cap[0, z-1]|.$$
Then
$$T\subseteq\left(\bigcup_{j=0}^{k-1}(jL+U)\right)
\cup\bigl(kL+(U\cap[0, z-1])\bigr),$$
thus,
$$|T|\leqslant\sum_{j=0}^{k-1}|jL+U|
+|kL+(U\cap[0, z-1])|=kv+c, \ \ 0\leqslant c\leqslant\min\{v, z\}.$$
Since $e_p(Y)=l(T)+1-|T|$, we have
$$|T|-2e_p(Y)=3|T|-2(l(T)+1)\leqslant k(3v-2L)+3c-2z.$$
If $2L\geqslant3v$, then $k\geqslant1$ and $3c-2z\leqslant c\leqslant v$ give
$$|T|-2e_p(Y)\leqslant(3v-2L)+v=4v-2L.$$
Since
$$
\begin{aligned}
p(N-p)&=|T|-e_p(Y)-D_p(Y)-1,\\
|T|+v+\rho&=2p(N-p)+N+1-a_{p+1}(X),
\end{aligned}
$$
we have
$$|T|-2e_p(Y)=v+\rho+2D_p(Y)+1+a_{p+1}(X)-N.$$
Together with
$$2L-3v=2D_p(Y)-2D_{p+1}(X)+a_p(Y)-a_{p+1}(X)+3\rho-2,$$
we have
\begin{equation}\label{eq:incomplete-projection-capacity}
2L\geqslant3v \Longrightarrow
4D_p(Y)-2D_{p+1}(X)+a_p(Y)+4\rho-1\leqslant N.
\end{equation}

Suppose that $D_p(Y)\geqslant\frac{N-1}{2}$.
Since $D_p(Y), D_{p+1}(X)\in\mathbb Z$ and
$D_{p+1}(X)<\frac{N+1}{2}$, we have
$D_p(Y)\geqslant D_{p+1}(X)$.
Since $a_p(Y)\geqslant a_{p+1}(X)+\rho+1$, we have
$$
\begin{aligned}
L-v
&=N+D_p(Y)+a_p(Y)-a_{p+1}(X)-D_{p+1}(X)+\rho-1
\geqslant N,\\
3v-2L
&=2(D_{p+1}(X)-D_p(Y))-a_p(Y)+a_{p+1}(X)-3\rho+2
\leqslant1.
\end{aligned}
$$
Thus,
$$v+\rho+2D_p(Y)+1+a_{p+1}(X)-N
=|T|-2e_p(Y)\leqslant k+3c-2z,$$
and
$$
\begin{aligned}
(v-c)+2(z-c)
&\leqslant k+N-\rho-2D_p(Y)-1-a_{p+1}(X)
\leqslant k-a_{p+1}(X),\\
e_p(Y)
&=kL+z-|T|\geqslant k(L-v)+z-c\geqslant k(L-v).
\end{aligned}
$$

Write $Y=\{y_0, \ldots, y_{N-1}\}$, where $y_0<\cdots<y_{N-1}$.
Since $N\geqslant2p$ and $u=\sum_{i=0}^{p-1}y_i$, we have
$$u-y_i\in(p-1)^{\wedge}Y \ (0\leqslant i<p), \ \
u+y_i\in(p+1)^{\wedge}Y \ (N-p\leqslant i<N).$$
Thus,
$$\{L-y_i:0\leqslant i<p\}, \ \
\{y_i:N-p\leqslant i<N\}\subseteq U.$$
For every $A\subseteq U$ with $|A|=p$, since
$$|U\cap[0,z-1]|=c, \ \ |U\cap[z,L-1]|=v-c, \ \ |U\setminus A|=v-p,$$
we have
$$
\begin{aligned}
\sum A
&=\sum U-\sum(U\setminus A)\\
&\leqslant\sum_{i=z-c}^{z-1}i+\sum_{i=L-v+c}^{L-1}i-\sum_{i=0}^{v-p-1}i\\
&=\frac{c(2z-c-1)+(v-c)(2L-v+c-1)-(v-p)(v-p-1)}2\\
&=\frac{p(2v-p-1)}2+c(z-c)+(L-v)(v-c).
\end{aligned}
$$
Since
$$2v-L-N-1=D_{p+1}(X)-D_p(Y)-2\rho\leqslant0, \ \ c\leqslant v\leqslant2(L-v)+1,$$
we have
$$
\begin{aligned}
l(T)
&=\sum_{i=0}^{p-1}(L-y_i)+\sum_{i=N-p}^{N-1}y_i-pL\\
&\leqslant p(2v-L-p-1)+2c(z-c)+2(L-v)(v-c)\\
&\leqslant p(N-p)+(2(L-v)+1)\bigl((v-c)+2(z-c)\bigr)\\
&\leqslant p(N-p)+(2(L-v)+1)(k-a_{p+1}(X)).
\end{aligned}
$$
Since $Y\subseteq[1,L-1]$, we have
$$kL\leqslant l(T)+1\leqslant p(L-2)+1<pL, \ \ k\leqslant p-1.$$
Together with
$$l(T)=p(N-p)+2e_p(Y)+D_p(Y)
\geqslant p(N-p)+2k(L-v)+D_p(Y),$$
we have
$$2N\leqslant2a_{p+1}(X)(L-v)+D_p(Y)
\leqslant k-a_{p+1}(X)\leqslant p-2<N,$$
a contradiction.
Hence, $D_p(Y)<\frac{N-1}{2}$.
Since $D_{p+1}(X)<\frac{N+1}{2}$ and
$a_p(Y)\geqslant a_{p+1}(X)+1$, we have
$$
\begin{aligned}
e_p(Y)
&\geqslant N+D_p(Y)+a_p(Y)-a_{p+1}(X)-D_{p+1}(X)-1\\
&\geqslant N+D_p(Y)-D_{p+1}(X)>2D_p(Y).
\end{aligned}
$$

This completes the proof of Lemma \ref{lem:stratified-incomplete-projection}.
\end{proof}

\begin{lem}\label{lem:two-block-completion}
Let $r\geqslant4$ and $0\leqslant\lambda\leqslant r-2$ be integers and put
$$P=[0, r-1], \ \ C=[0, r-2]\cup\{r-1+\lambda\}.$$
Let $t\in\mathbb Z$ and
$$Q=t+C \ \hbox{or}\ Q=t+r-1-C, \ \ \min Q\geqslant r.$$
Put $X=P\cup Q$. Then
$$|r^{\wedge}X|\geqslant(r-2)\min\{t, 2r\}+2r+1+\lambda.$$
\end{lem}

\begin{proof}
For $1\leqslant j\leqslant r-1$, since
$C=[0, r-2]\cup\{r-1+\lambda\}$, by
\eqref{eq:unit-interval-restricted-ranks}, we have
$$
\begin{aligned}
(r-j)^{\wedge}P+j^{\wedge}C
&=\bigl((r-j)^{\wedge}P+j^{\wedge}[0, r-2]\bigr)\\
&\qquad\cup\bigl((r-j)^{\wedge}P+r-1+\lambda
+(j-1)^{\wedge}[0, r-2]\bigr)\\
&=\frac{r(r-1)}2-j(r-j)\\
&\qquad+\bigl([0, 2j(r-j)-j]
\cup[r-j+\lambda, 2j(r-j)+\lambda]\bigr).
\end{aligned}
$$
Since
$$r-j+\lambda-(2j(r-j)-j)=r+\lambda-2j(r-j)\leqslant\lambda-r+2\leqslant 0,$$
we have
$$(r-j)^{\wedge}P+j^{\wedge}C=\left[\frac{r(r-1)}2-j(r-j), \frac{r(r-1)}2+j(r-j)+\lambda\right].$$

For $0\leqslant j\leqslant r$, put
$$T_j=(r-j)^{\wedge}P+j^{\wedge}Q.$$
Since $r-1-P=P$, for $1\leqslant j\leqslant r-1$, we have
$$T_j=jt+((r-j)^{\wedge}P+j^{\wedge}C) \text{ or } T_j=jt+r(r-1)-((r-j)^{\wedge}P+j^{\wedge}C). $$
Thus,
$$T_j=[\min T_j, \max T_j], \ \ |T_j|=2j(r-j)+1+\lambda.$$

Since $\min Q\geqslant r$, we have
$$T_0=\left\{\sum P\right\}, \ \
T_r=\left\{\sum Q\right\},$$
$$\min T_1-\max T_0
=\min T_r-\max T_{r-1}
=\min Q-r+1\geqslant1.$$
Since $t\geqslant\min Q\geqslant r$, for
$2\leqslant j\leqslant r-1$, in both cases we have
$$
\begin{aligned}
\min T_j-\min T_{j-1}&=t-r+2j-1>0,\\
\max T_j-\max T_{j-1}&=t+r-2j+1>0.
\end{aligned}
$$
Hence,
$$
\begin{aligned}
\left|T_j\setminus\bigcup_{i=1}^{j-1}T_i\right|
&=|T_j\setminus T_{j-1}|\\
&=\min\{|T_j|, \max T_j-\max T_{j-1}\}\\
&=\min\{2j(r-j)+1+\lambda, t+r-2j+1\}.
\end{aligned}
$$
Since
$$2j(r-j)+1+\lambda-(3r-2j+1)
=2(j-2)(r-j-1)+r-4+\lambda\geqslant0,$$
we have
$$
\begin{aligned}
\left|T_j\setminus\bigcup_{i=1}^{j-1}T_i\right|
&\geqslant\min\{3r-2j+1, t+r-2j+1\}\\
&=\min\{t, 2r\}+r-2j+1.
\end{aligned}
$$
Since $P\cap Q=\emptyset$ and
$r^{\wedge}X=\bigcup_{j=0}^{r}T_j$, we have
$$
\begin{aligned}
|r^{\wedge}X|
&=2+|T_1|
+\sum_{j=2}^{r-1}\left|T_j\setminus\bigcup_{i=1}^{j-1}T_i\right|\\
&\geqslant2+(2r-1+\lambda)
+\sum_{j=2}^{r-1}\bigl(\min\{t, 2r\}+r-2j+1\bigr)\\
&=(r-2)\min\{t, 2r\}+2r+1+\lambda.
\end{aligned}
$$

This completes the proof of Lemma \ref{lem:two-block-completion}.
\end{proof}

\begin{lem}\label{lem:inner-cardinality-barrier}
Let $p\geqslant 2$ and $N\geqslant\max\{5, 2p\}$.
Let $X\subseteq[0, L]$ be a primitive set of $N+2$ integers such that $0, L\in X$.
Put $Y=X\setminus\{0, L\}$.
Suppose that
$$|p^{\wedge}Y|\geqslant2p(N-p)-N+1, \ \ a_{p+1}(X)\geqslant1.$$
Then
$$D_{p+1}(X)\geqslant0.$$
If $N\geqslant7$ and $4D_{p+1}(X)<(a_{p+1}(X)-1)(N+1)$, then
$$D_{p+1}(X)\geqslant L-N-1.$$
\end{lem}

\begin{proof}
Put $B=\{0\}\cup\pi_L(Y)$. Since $X$ is primitive, we have
$$|B|=N+1, \ \ \langle B\rangle=\mathbb Z_L.$$
Since $N+1\geqslant6$ and
$2\leqslant p\leqslant\left\lfloor\frac{N}{2}\right\rfloor\leqslant N-1$,
by Lemma \ref{lem:endpoint-lifting} and Theorem \ref{thm:r4-coarse-cyclic-all-orders}, we have
$$
\begin{aligned}
\min\{L, 2N\}
&\leqslant|(p+1)^{\wedge}X|-|p^{\wedge}Y|\\
&\leqslant\bigl(2p(N-p)+N+1-a_{p+1}(X)\bigr)
-\bigl(2p(N-p)-N+1\bigr)\\
&=2N-a_{p+1}(X)<2N.
\end{aligned}
$$
Thus, $L\leqslant2N-a_{p+1}(X)$. Put
$$u=2N-a_{p+1}(X)-L\geqslant0.$$

Write $X=\{x_0, \ldots, x_{N+1}\}$, where
$0=x_0<\cdots<x_{N+1}=L$.
Put
$$P=\{x_i:0\leqslant i\leqslant p\}, \ \
C=\{L-x_{N+1-i}:0\leqslant i\leqslant p\},$$
$$\lambda=\sum P+\sum C-p(p+1).$$
Since $x_i\geqslant i$ and $L-x_{N+1-i}\geqslant i$,
we have $\lambda\geqslant0$ and
$$l((p+1)^{\wedge}X)
=\left((p+1)L-\sum C\right)-\sum P
=(p+1)L-p(p+1)-\lambda.$$
Since
$$L=2N-a_{p+1}(X)-u, \ \
|(p+1)^{\wedge}X|
=2(p+1)(N-p+1)-N-1-a_{p+1}(X)$$
and
$$D_{p+1}(X)
=2|(p+1)^{\wedge}X|-l((p+1)^{\wedge}X)
-(p+1)(N-p+1)-2,$$
we have
$$
\begin{aligned}
D_{p+1}(X)-L+N+1
&=p(a_{p+1}(X)-3)+(p-2)(N-2p)+(p+2)u+\lambda,\\
D_{p+1}(X)
&=(p-1)(a_{p+1}(X)+N-2p)-p-1+(p+1)u+\lambda.
\end{aligned}
$$

Suppose that $p=2$ and $D_3(X)<0$.
Since $a_3(X)\geqslant1$, $N\geqslant5$ and $u,\lambda\geqslant0$, we have
$$0>D_3(X)=a_3(X)+N-7+3u+\lambda\geqslant-1,$$
thus,
$$a_3(X)=1, \ \ N=5, \ \ u=\lambda=0, \ \ L=9.$$
Since $\lambda=0$, we have $P=C=[0, 2]$, thus,
$$X=E\cup\{x\}, \ \ E=[0, 2]\cup[7, 9], \ \ x\in[3, 6].$$
Since $9-E=E$, by reflection we may assume that $x\in\{3, 4\}$.
By \eqref{eq:unit-interval-restricted-ranks}, we have
$$
\begin{aligned}
3^{\wedge}E&=\{3\}\cup[8, 12]\cup[15, 19]\cup\{24\},\\
2^{\wedge}E&=[1, 3]\cup[7, 11]\cup[15, 17].
\end{aligned}
$$
Since $x\in\{3, 4\}$, we have
$$\bigl(x+[1, 3]\bigr)\cup\{13, 14, 20\}
\subseteq(x+2^{\wedge}E)\setminus3^{\wedge}E,$$
thus,
$$
\begin{aligned}
|3^{\wedge}X|
&=|3^{\wedge}E|+|(x+2^{\wedge}E)\setminus3^{\wedge}E|\\
&\geqslant12+\left|\bigl(x+[1, 3]\bigr)\cup\{13, 14, 20\}\right|
=18>17=|3^{\wedge}X|,
\end{aligned}
$$
a contradiction. Hence, $D_3(X)\geqslant0$ when $p=2$.
If $p\geqslant3$ and $a_{p+1}(X)\geqslant2$, then
$$
\begin{aligned}
D_{p+1}(X)
&=(p-1)(a_{p+1}(X)+N-2p)-p-1+(p+1)u+\lambda\\
&\geqslant2(p-1)-p-1=p-3\geqslant0.
\end{aligned}
$$
If $a_{p+1}(X)=1$ and $D_{p+1}(X)<0$, then $p\geqslant3$ and
$$0>D_{p+1}(X)=-2+(p-1)(N-2p)+(p+1)u+\lambda.$$
Since $N-2p,u,\lambda$ are nonnegative integers, we have
$$N=2p, \ \ u=0, \ \ \lambda\leqslant1.$$

For the second assertion, suppose that $N\geqslant7$ and
$4D_{p+1}(X)<(a_{p+1}(X)-1)(N+1)$.
If $a_{p+1}(X)=1$, then $D_{p+1}(X)<0$, which is the preceding case.
If $a_{p+1}(X)\geqslant3$, then
$$D_{p+1}(X)-L+N+1
=p(a_{p+1}(X)-3)+(p-2)(N-2p)+(p+2)u+\lambda\geqslant0.$$
If $a_{p+1}(X)=2$ and $u\geqslant1$, then
$$
\begin{aligned}
D_{p+1}(X)-L+N+1
&=-p+(p-2)(N-2p)+(p+2)u+\lambda\\
&\geqslant-p+(p+2)=2>0.
\end{aligned}
$$
It remains to consider $a_{p+1}(X)=2$ and $u=0$.
Since $D_{p+1}(X)=p-3+(p-1)(N-2p)+\lambda$, we have
$$4D_{p+1}(X)-(N+1)
=3N-21+4(p-2)(N-2p-1)+4\lambda.$$
If $N\geqslant2p+1$, then $N\geqslant7$ and $p\geqslant2$ give
$$4D_{p+1}(X)-(N+1)\geqslant0,$$
contrary to $4D_{p+1}(X)<N+1$. Thus,
$$N=2p\geqslant8, \ \ p\geqslant4, \ \
\lambda<\frac{13-2p}{4}\leqslant\frac54.$$
Since $\lambda\in\mathbb Z$, we have $\lambda\leqslant1$.

In both remaining cases, we have
$$
\begin{gathered}
a_{p+1}(X)\in\{1, 2\}, \ \
p\geqslant a_{p+1}(X)+2, \ \
N=2p, \ \ \lambda\leqslant1,\\
L=4p-a_{p+1}(X), \ \
|(p+1)^{\wedge}X|=2p(p+1)+1-a_{p+1}(X).
\end{gathered}
$$
Since
$$
\begin{gathered}
0=x_0\leqslant x_1-1\leqslant\cdots\leqslant x_p-p,\\
0=L-x_{N+1}\leqslant L-x_N-1\leqslant\cdots
\leqslant L-x_{N+1-p}-p,\\
\sum_{i=0}^{p}(x_i-i)
+\sum_{i=0}^{p}(L-x_{N+1-i}-i)=\lambda\leqslant1,
\end{gathered}
$$
we have
$$x_i=L-x_{N+1-i}=i, \ \ 0\leqslant i\leqslant p-1,$$
$$x_p-p+(L-x_{N+1-p}-p)=\lambda.$$
Since reflection $X\mapsto L-X$ interchanges $P$ and $C$, we may assume that
$$P=[0, p], \ \ C=[0, p-1]\cup\{p+\lambda\}.$$
Since $N=2p$, we have
$$X=P\cup(L-C), \ \ P\cap(L-C)=\emptyset.$$
Since
$$
\begin{gathered}
0\leqslant\lambda\leqslant1\leqslant p-1,\\
L-p=3p-a_{p+1}(X)\geqslant2(p+1),\\
\min(L-C)=L-p-\lambda\geqslant2p+1\geqslant p+1,
\end{gathered}
$$
by Lemma \ref{lem:two-block-completion}, with $r=p+1$ and $t=L-p$, we have
$$|(p+1)^{\wedge}X|\geqslant2p(p+1)+1+\lambda
>2p(p+1)+1-a_{p+1}(X),$$
a contradiction.

This completes the proof of Lemma \ref{lem:inner-cardinality-barrier}.
\end{proof}

\section{Sums of three distinct integers}\label{sec:threefold-sums}

\begin{lem}
\label{lem:six-small-pair-inverse}
Let $0\in B\subseteq\mathbb Z_L$ be a set of five elements such that $\langle B\rangle=\mathbb Z_L$. Then
$$|2^{\wedge}B|\geqslant\min\{L, 7\}. $$
Let $F$ be a primitive set of six integers. Then
$$|2^{\wedge}F|\geqslant 4+\min\{l(F), 7\}, \ \ |3^{\wedge}F|\geqslant 5+\min\{l(F), 7\}. $$
Moreover, if $|2^{\wedge}F|\leqslant 10$, then
$$l(F)=|2^{\wedge}F|-4\in\{5, 6\}. $$
\end{lem}

\begin{proof}
If $L\leqslant 6$, then for every $z\in\mathbb Z_L$, we have
$$|B\cap(z-B)|\geqslant10-L\geqslant4, \ \ |\{x\in\mathbb Z_L:2x=z\}|\leqslant2.$$
Thus, there exists $x\in B\cap(z-B)$ such that $x\neq z-x$, and
$$z=x+(z-x)\in2^{\wedge}B.$$
Hence, $2^{\wedge}B=\mathbb Z_L$.

Suppose that $L\geqslant7$ and $|2^{\wedge}B|\leqslant6$.
If $B$ is periodic, then by $|\operatorname{Stab}(B)|\mid|B|=5$,
$0\in B$ and $\langle B\rangle=\mathbb Z_L$, we have
$$|\operatorname{Stab}(B)|=5, \ \ B=\operatorname{Stab}(B), \ \ \mathbb Z_L=\langle B\rangle=B,$$
contrary to $L\geqslant7$. Thus, $B$ is aperiodic.
By Lemma \ref{lem:disjoint-pairs}, we have
$$\mu_B(s)\leqslant\left\lfloor\frac{|B|}{2}\right\rfloor=2, \ \ s\in\mathbb Z_L.$$
Thus, $\mu_B(s)=1$ for $s\in2^{\wedge}B\setminus\mathcal P(B)$ and $\mu_B(s)=2$ for $s\in\mathcal P(B)$.
Hence,
$$
\begin{aligned}
10=\binom{|B|}{2}
&=\sum_{s\in2^{\wedge}B}\mu_B(s)\\
&=|2^{\wedge}B\setminus\mathcal P(B)|+2|\mathcal P(B)|\\
&=|2^{\wedge}B|+|\mathcal P(B)|.
\end{aligned}
$$
Since $3^{\wedge}B=\sum_{x\in B}x-2^{\wedge}B$, we have
$$|\mathcal P(B)|\geqslant4, \ \ |3^{\wedge}B|=|2^{\wedge}B|\leqslant6<L.$$
By Proposition \ref{cyc:pro:equality-popular-period} (2), we have
$$8=|3^{\wedge}B|\leqslant6,$$
a contradiction.

By translation, we may assume that $F\subseteq[0,L]$, where $0,L\in F$ and $L=l(F)\geqslant5$.
Put
$$I=F\setminus\{0,L\}, \ \ B=\pi_L(F\setminus\{L\}).$$
Since $F$ is primitive and $|I|=4$, by Lemma \ref{lem:endpoint-lifting}, we have
$$|B|=5, \ \ \langle B\rangle=\mathbb Z_L.$$
By Lemma \ref{lem:endpoint-lifting}, with $p\in\{1,2\}$, and Theorem \ref{thm:nathanson-restricted-minimum}, we have
$$
\begin{aligned}
|2^{\wedge}F|
&\geqslant4+|B\cup2^{\wedge}B|\geqslant4+\min\{L,7\},\\
|3^{\wedge}F|
&\geqslant|2^{\wedge}I|+|2^{\wedge}B\cup3^{\wedge}B|\geqslant5+\min\{L,7\}.
\end{aligned}
$$

If $|2^{\wedge}F|\leqslant10$, then $5=|F|-1\leqslant L\leqslant6$.
For $L=5$, we have
$$F=[0,5], \ \ 2^{\wedge}F=[1,9], \ \ |2^{\wedge}F|=9.$$
For $L=6$, we have
$$10\geqslant|2^{\wedge}F|\geqslant4+L=10.$$
Thus,
$$L=|2^{\wedge}F|-4\in\{5,6\}.$$

This completes the proof of Lemma \ref{lem:six-small-pair-inverse}.
\end{proof}

\begin{lem}
\label{lem:six-point-mixed-triple-mechanism}
Let $J$ be a primitive set of six integers such that $\min J=0$, $\max J=\ell$ and $W=l(3^{\wedge}J)$. If
$$8\leqslant\ell\leqslant 10, \ \ |3^{\wedge}J|=12, \ \ W>\ell+8,$$
then
$$J=[0,2]\cup[\ell-2,\ell].$$
\end{lem}
\begin{proof}
Put
$$I=J\setminus\{0,\ell\}, \ \ B=\pi_\ell(J\setminus\{\ell\}).$$
Since $J$ is primitive and $|I|=4$, by Lemma \ref{lem:endpoint-lifting}, with $p=2$, we have
$$|B|=5, \ \ \langle B\rangle=\mathbb Z_\ell.$$
By Lemmas \ref{lem:endpoint-lifting} and \ref{lem:six-small-pair-inverse}, and Theorem \ref{thm:nathanson-restricted-minimum}, we have
$$12\geqslant|2^{\wedge}I|+|2^{\wedge}B\cup3^{\wedge}B|\geqslant5+7=12.$$
Thus, $|2^{\wedge}I|=5$.
Write $I=\{x_0,x_1,x_2,x_3\}$, where $x_0<x_1<x_2<x_3$.
Since
$$
\begin{aligned}
x_0+x_1&<x_0+x_2<\min\{x_0+x_3,x_1+x_2\},\\
\max\{x_0+x_3,x_1+x_2\}&<x_1+x_3<x_2+x_3,
\end{aligned}
$$
we have $x_0+x_3=x_1+x_2$. Hence,
$$W-\ell=l(2^{\wedge}I)=x_2+x_3-x_0-x_1=2(x_2-x_0).$$
Since $W-\ell>8$ is even and $\ell\leqslant10$, we have
$$W-\ell\geqslant10\geqslant\ell, \ \ W\geqslant2\ell.$$

Write $J=P\cup Q$, where $|P|=|Q|=3$ and $\max P<\min Q$, and put $R=Q-P$.
Then $W=\sum Q-\sum P$ and
$$
\begin{aligned}
3^{\wedge}J-\sum P
&=\bigcup_{j=0}^3\left((3-j)^{\wedge}P+j^{\wedge}Q-\sum P\right)\\
&=\{0\}\cup R\cup(W-R)\cup\{W\}.
\end{aligned}
$$
Since $0<\min R\leqslant\max R=\ell<W$, we have
$$\{0,W\}\cap(R\cup(W-R))=\emptyset.$$
If $W=2\ell$, then
$$R\cap(W-R)=\{\ell\}, \ \ 12=2+2|R|-1=2|R|+1,$$
a contradiction. Thus, $W>2\ell$ and
$$\max R=\ell<W-\ell=\min(W-R), \ \ 12=2+2|R|, \ \ |R|=5.$$

Since $|P|=|\ell-Q|=3$, $\min P=\min(\ell-Q)=0$ and
$$|P+(\ell-Q)|=|\ell-R|=5=|P|+|\ell-Q|-1,$$
by \cite[Theorem 1.5]{Nathanson1996}, there exists a positive integer $a$ such that
$$P=\ell-Q=a\cdot[0,2].$$
Thus,
$$J=a\cdot[0,2]\cup(\ell-2a+a\cdot[0,2]), \ \ W=\sum Q-\sum P=3\ell-6a.$$
Since $W>\ell+8$ and $\ell\leqslant10$, we have
$$3a+4<\ell\leqslant10, \ \ 1\leqslant a<2,$$
thus, $a=1$ and $J=[0,2]\cup[\ell-2,\ell]$.

This completes the proof of Lemma \ref{lem:six-point-mixed-triple-mechanism}.
\end{proof}

\begin{lem}\label{lem:seminar-eight-weak-defect}
Let $X\subseteq[0, L]$ be a primitive set of eight integers such that $0, L\in X$.
If $|4^{\wedge}X|\leqslant24$, then
$$D_4(X)\geqslant0.$$
\end{lem}

\begin{proof}
Suppose that $D_4(X)<0$.
Write
$$X=\{0,L\}\cup(u+J), \ \ \min J=0, \ \ \max J=\ell, \ \
u,v\geqslant1, \ \ L=u+\ell+v,$$
and put $Y=u+J$.
If $d(J)>1$, then by Proposition \ref{pro:one-residue-class},
with $(k,h)=(8,4)$, we have
$$24\geqslant|4^{\wedge}X|\geqslant2\cdot4(8-4)-8+1=25,$$
a contradiction. Thus, $J$ is primitive.

If $|3^{\wedge}J|\geqslant13$, then
$$|3^{\wedge}Y|=|3^{\wedge}J|
\geqslant13=2\cdot3(6-3)-6+1, \ \
a_4(X)=25-|4^{\wedge}X|\geqslant1.$$
By Lemma \ref{lem:inner-cardinality-barrier},
with $(p,N)=(3,6)$, we have $D_4(X)\geqslant0$,
a contradiction. Hence, $|3^{\wedge}J|\leqslant12$.

Put
$$B=\{0\}\cup\pi_L(Y), \ \
U=\pi_L(2^{\wedge}Y\cup3^{\wedge}Y\cup4^{\wedge}Y).$$
Since $X$ is primitive and $|Y|=6$,
by Lemma \ref{lem:endpoint-lifting}, with $p=3$, we have
$$|B|=7, \ \ \langle B\rangle=\mathbb Z_L, \ \
U=3^{\wedge}B\cup4^{\wedge}B.$$
By Lemma \ref{lem:endpoint-lifting} and
Theorem \ref{thm:r4-coarse-cyclic-all-orders}, we have
$$|4^{\wedge}X|\geqslant|3^{\wedge}J|+|U|
\geqslant|3^{\wedge}J|+\min\{L,12\},$$
$$D_4(X)=2|4^{\wedge}X|-L-l(3^{\wedge}J)-18.$$
Since $3^{\wedge}J\subseteq[3,3\ell-3]$, we have
$$l(3^{\wedge}J)\leqslant3\ell-6.$$

If $\ell\leqslant6$, then $|J|=6$ gives $5\leqslant\ell\leqslant6$.
By Lemma \ref{lem:six-small-pair-inverse}, we have
$$|3^{\wedge}J|\geqslant\ell+5.$$
Since $0\leqslant l(Y)-6+1=\ell-5\leqslant1=6-3-2$,
by Lemma \ref{lem:dense-low-adjacent-projection},
with $m=3$, $n=6$ and $I=Y$, we have
$$|U|\geqslant\min\{L,21-\ell\}.$$
Thus,
$$
\begin{aligned}
24\geqslant|4^{\wedge}X|
&\geqslant\ell+5+\min\{L,21-\ell\}\\
&=\min\{L+\ell+5,26\}.
\end{aligned}
$$
Hence,
$$L+\ell+5\leqslant24<26, \ \
|4^{\wedge}X|\geqslant L+\ell+5.$$
Since $L\geqslant\ell+2$, we have
$$
\begin{aligned}
D_4(X)
&\geqslant2(L+\ell+5)-L-(3\ell-6)-18\\
&=L-\ell-2\geqslant0,
\end{aligned}
$$
a contradiction. Thus, $\ell\geqslant7$.
By Lemma \ref{lem:six-small-pair-inverse}, we have
$$12\geqslant|3^{\wedge}J|\geqslant5+\min\{\ell,7\}=12,$$
thus, $|3^{\wedge}J|=12$.

If $L=13$, then by Lemma \ref{lem:group-restricted}, we have
$$|U|\geqslant|4^{\wedge}B|
\geqslant\min\{13,4(7-4)+1\}=13.$$
If $L\geqslant14=2\cdot6+2$, then by
$$|3^{\wedge}Y|=12<2\cdot3(6-3)-6+2=14$$
and Theorem \ref{thm:strict-critical-cyclic-surplus},
with $(p,N)=(3,6)$, we have $|U|\geqslant13$.
In both cases, we have
$$24\geqslant|4^{\wedge}X|
\geqslant|3^{\wedge}J|+|U|\geqslant12+13=25,$$
a contradiction. Thus, $L\leqslant12$ and
$$|4^{\wedge}X|\geqslant L+12, \ \
0>D_4(X)\geqslant L+6-l(3^{\wedge}J).$$
Hence,
$$3\ell-6\geqslant l(3^{\wedge}J)>L+6\geqslant\ell+8, \ \
\ell\leqslant L-2\leqslant10,$$
thus, $8\leqslant\ell\leqslant10$.
By Lemma \ref{lem:six-point-mixed-triple-mechanism}, we have
$$J=[0,2]\cup[\ell-2,\ell], \ \
l(3^{\wedge}J)=3\ell-6.$$

Since $3\ell-6>L+6$ and $L\leqslant12$, we have
$$3(u+v)=3(L-\ell)<2L-12\leqslant12,$$
thus,
$$2\leqslant u+v\leqslant3.$$
Since $\ell-J=J$, we have
$$L-X=\{0,L\}\cup(v+J), \ \ D_4(L-X)=D_4(X).$$
By reflection, we may assume that $u\leqslant v$. Then
$$u=1, \ \ v\in\{1,2\}, \ \ L=\ell+v+1$$
and
$$X=[0,3]\cup\bigl(\ell-1+\{0,1,2,v+2\}\bigr).$$
By Lemma \ref{lem:two-block-completion}, with $r=4$,
$\lambda=v-1\in[0,1]$ and $t=\ell-1\geqslant7$, we have
$$24\geqslant|4^{\wedge}X|
\geqslant2\min\{\ell-1,8\}+v+8.$$
Thus,
$$\min\{\ell-1,8\}\leqslant\frac{16-v}{2}<8, \ \
\ell\leqslant8.$$
Together with $\ell\geqslant8$, we have
$$\ell=8, \ \ |4^{\wedge}X|\geqslant2\ell+v+6.$$
Hence,
$$
\begin{aligned}
D_4(X)
&=2|4^{\wedge}X|-L-l(3^{\wedge}J)-18\\
&\geqslant2(2\ell+v+6)-(\ell+v+1)-(3\ell-6)-18\\
&=v-1\geqslant0,
\end{aligned}
$$
a contradiction.

This completes the proof of Lemma \ref{lem:seminar-eight-weak-defect}.
\end{proof}

\begin{cor}\label{cor:seven-below-eighteen}
If $E\subseteq\mathbb Z$ is primitive, $|E|=7$ and $|3^{\wedge}E|\leqslant 17$, then
$$D_3(E)\geqslant0.$$
\end{cor}

\begin{proof}
By translation, we may assume that $\min E=0$ and $\max E=\ell$.
Put $F=E\setminus\{0, \ell\}$. Then
$$|F|=5, \ \ F\subseteq[1, \ell-1].$$
If $|2^{\wedge}F|\geqslant8=2\cdot2(5-2)-5+1$, then by
$a_3(E)=18-|3^{\wedge}E|\geqslant1$ and
Lemma \ref{lem:inner-cardinality-barrier}, with $(p,N)=(2,5)$,
we have $D_3(E)\geqslant0$.

Suppose that $|2^{\wedge}F|\leqslant7$.
By Theorem \ref{thm:nathanson-restricted-minimum}, we have
$$|2^{\wedge}F|=7, \ \
F=\min F+d(F)\cdot[0, 4], \ \ d(F)\geqslant1.$$
If $d(F)\geqslant2$, then by Proposition \ref{pro:one-residue-class}, we have
$$17\geqslant|3^{\wedge}E|\geqslant18,$$
a contradiction. Thus, $d(F)=1$ and
$$F=\min F+[0, 4], \ \ l(F)-|F|+1=4-5+1=0.$$
By Lemma \ref{lem:dense-low-adjacent-projection},
with $m=2$, $n=5$, $g=0$ and $I=F$,
and Lemma \ref{lem:endpoint-lifting}, we have
$$
\begin{aligned}
17&\geqslant|3^{\wedge}E|\\
&\geqslant7+\left|\pi_\ell(F\cup2^{\wedge}F\cup3^{\wedge}F)\right|\\
&\geqslant7+\min\{\ell, 12\}.
\end{aligned}
$$
Thus,
$$\ell\leqslant10, \ \
\pi_\ell(F\cup2^{\wedge}F\cup3^{\wedge}F)=\mathbb Z_\ell.$$
Since $2^{\wedge}F=2\min F+[1, 7]$, we have
$$D_2(F)=2\cdot7-6-2(5-2)-2=0.$$
By $\ell\geqslant|E|-1=6$ and
Corollary \ref{lem:exact-endpoint-transfer}, we have
$$D_3(E)\geqslant D_2(F)+\ell-6=\ell-6\geqslant0.$$

This completes the proof of Corollary \ref{cor:seven-below-eighteen}.
\end{proof}

\begin{cor}\label{cor:seminar-seven-pair-elementary-floor}
Let $E\subseteq[0, w]$ be a primitive set of seven integers such that $0, w\in E$.
Then
$$|2^{\wedge}E|\geqslant5+\min\{w, 8\}.$$
\end{cor}

\begin{lem}\label{lem:seminar-seven-hexagon-from-cyclic-equality}
Let $E\subseteq[0, w]$ be a primitive set of seven integers such that $0, w\in E$.
If $w\geqslant9$ and $|2^{\wedge}E|\leqslant13$, then there exist positive integers $p,c$ with $\gcd(p,c)=1$ such that
$$E=\{0, p, p+c, 2p+c, 3p+c, 3p+2c, 4p+2c\}.$$
\end{lem}

\begin{proof}[Proof of Corollary \ref{cor:seminar-seven-pair-elementary-floor} and Lemma \ref{lem:seminar-seven-hexagon-from-cyclic-equality}]
Write
$$E=\{0,a,b,c,d,e,w\}, \ \ 0<a<b<c<d<e<w.$$
If $w\leqslant9=2|E|-5$, then by
Lemma \ref{lem:dense-strong-endpoint-deletion} (3), we have
$$|2^{\wedge}E|\geqslant w+5\geqslant5+\min\{w,8\}.$$
For $w=9$, this gives $|2^{\wedge}E|\geqslant14>13$.
If $w\geqslant10$ and $|2^{\wedge}E|\geqslant14$, then
$$|2^{\wedge}E|\geqslant5+\min\{w,8\}.$$
Thus, we may assume that $w\geqslant10$ and $|2^{\wedge}E|\leqslant13$.

Put $F=E\setminus\{w\}$.
If $d(F)>1$, then by $d(E)=1$, we have
$$F\subseteq d(F)\mathbb Z, \ \ w\notin d(F)\mathbb Z.$$
Thus,
$$2^{\wedge}F\subseteq d(F)\mathbb Z, \ \
w+F\subseteq w+d(F)\mathbb Z, \ \
2^{\wedge}F\cap(w+F)=\emptyset.$$
By Theorem \ref{thm:nathanson-restricted-minimum}, we have
$$13\geqslant|2^{\wedge}E|
=|2^{\wedge}F|+|w+F|\geqslant9+6=15,$$
a contradiction. Hence, $d(F)=1$.

If $e\leqslant6$, then
$$5\leqslant e\leqslant6, \ \ d\leqslant5, \ \
2^{\wedge}F\cap(w+F)\subseteq[w,d+e]\subseteq[10,e+5].$$
By Lemma \ref{lem:six-small-pair-inverse}, we have
$$
\begin{aligned}
13\geqslant|2^{\wedge}E|
&=|2^{\wedge}F|+|w+F|-|2^{\wedge}F\cap(w+F)|\\
&\geqslant(e+4)+6-(e-4)=14,
\end{aligned}
$$
a contradiction. Thus, $e\geqslant7$.
Since
$$2^{\wedge}(w-E)=2w-2^{\wedge}E, \ \
w-(E\setminus\{0\})=(w-E)\setminus\{w\},$$
applying the preceding argument to $w-E$, we have
$$d(E\setminus\{0\})=1, \ \ w-a\geqslant7.$$
By Lemma \ref{lem:six-small-pair-inverse}, we have
$$|2^{\wedge}F|\geqslant11, \ \
|2^{\wedge}(E\setminus\{0\})|\geqslant11.$$

Since
$$w+d>d+e=\max2^{\wedge}F, \ \
b<a+b=\min2^{\wedge}(E\setminus\{0\}),$$
we have
$$
\begin{aligned}
\{w+d,w+e\}
&\subseteq2^{\wedge}E\setminus2^{\wedge}F,\\
\{a,b\}
&\subseteq2^{\wedge}E\setminus2^{\wedge}(E\setminus\{0\}).
\end{aligned}
$$
Thus,
$$
\begin{aligned}
13\geqslant|2^{\wedge}E|
&\geqslant|2^{\wedge}F|+2\geqslant13,\\
13\geqslant|2^{\wedge}E|
&\geqslant|2^{\wedge}(E\setminus\{0\})|+2\geqslant13.
\end{aligned}
$$
Hence,
$$|2^{\wedge}E|=13, \ \
|2^{\wedge}F|=|2^{\wedge}(E\setminus\{0\})|=11,$$
$$
\begin{aligned}
2^{\wedge}E\setminus2^{\wedge}F&=\{w+d,w+e\},\\
2^{\wedge}E\setminus2^{\wedge}(E\setminus\{0\})&=\{a,b\}.
\end{aligned}
$$
This proves Corollary \ref{cor:seminar-seven-pair-elementary-floor}.

Put $J=\{a,b,c,d,e\}$.
Since
$$2^{\wedge}F=J\cup2^{\wedge}J, \ \
2^{\wedge}(E\setminus\{0\})=2^{\wedge}J\cup(w+J)$$
and $e<w<w+a$, we have
$$
\begin{aligned}
\{c,d,e,w\}
&\subseteq2^{\wedge}(E\setminus\{0\})\setminus(w+J)
\subseteq2^{\wedge}J,\\
\{w+a,w+b,w+c\}
&\subseteq2^{\wedge}F\setminus J\subseteq2^{\wedge}J.
\end{aligned}
$$
Since $c\in2^{\wedge}J$ and $J\cap[1,c-1]=\{a,b\}$, we have
$$c\in2^{\wedge}\{a,b\}=\{a+b\}, \ \ c=a+b.$$
Applying this equality to $w-E$, we have
$$w-c=(w-e)+(w-d), \ \ w=d+e-c.$$

Since $d\in2^{\wedge}J$ and $d>c=a+b$, we have
$$d\in\{a+c,b+c\}, \ \ d\geqslant a+c.$$
Together with $w+b=d+e-a<d+e$ and $w+b\in2^{\wedge}J$, we have
$$
\begin{aligned}
c
&\leqslant d-a=w+b-e\\
&\leqslant\max\bigl(2^{\wedge}J\setminus\{d+e\}\bigr)-e
=(c+e)-e=c.
\end{aligned}
$$
Thus,
$$d=a+c=2a+b, \ \ w=a+e.$$

Suppose that $e>a+2b$.
Since
$$e+2a=w+a\in2^{\wedge}J, \ \ e+2a>3a+2b=c+d=\max2^{\wedge}\{a, b, c, d\},$$
we have
$$e+2a\in e+\{a, b, c, d\}, \ \ 2a\in\{a, b, c, d\}.$$
Since $a<2a<c<d$, we have $b=2a$.
Together with $e\in2^{\wedge}J$, we have
$$e\in2^{\wedge}\{a, 2a, 3a, 4a\}\subseteq a\cdot[3, 7], \ \ E\subseteq a\mathbb Z, \ \ w=a+e\leqslant 8a.$$
By $d(E)=1$, we have $a=1$ and $10\leqslant w\leqslant8$,
a contradiction. Thus, $e\leqslant a+2b$.
Applying this inequality to $w-E$, we have
$$e=w-a\leqslant(w-e)+2(w-d)=2e-a-2b,$$
thus,
$$e\geqslant a+2b, \ \ e=a+2b.$$
Hence,
$$E=\{0,a,b,a+b,2a+b,a+2b,2a+2b\}.$$
Since $d(E)=1$, we have
$$\gcd(a, b)=\gcd(a, b-a)=1, \ \ b-a>0.$$
This is the required form with parameters $a$ and $b-a$.

This completes the proofs of Corollary \ref{cor:seminar-seven-pair-elementary-floor} and Lemma \ref{lem:seminar-seven-hexagon-from-cyclic-equality}.
\end{proof}

\begin{cor}\label{cor:seminar-hexagon-arbitrary-span-completion}
Let $p,c,a,b$ be positive integers with $\gcd(p,c)=1$.
Put $h=2p+c$ and $L=2h+a+b$.
Suppose that $2h\geqslant10$ and let
$$E=\{0,p,p+c,2p+c,3p+c,3p+2c,4p+2c\}, \ \
X=\{0,L\}\cup(a+E).$$
Then
$$|3^{\wedge}X|\geqslant27.$$
If $|3^{\wedge}X|\leqslant28$, then $2h=10$ and $L\leqslant16$.
\end{cor}

\begin{proof}
If $c=p$ or $c=2p$, then by $\gcd(p,c)=1$, we have
$$(p,c)\in\{(1,1),(1,2)\}, \ \ 2h\in\{6,8\},$$
contrary to $2h\geqslant10$. Thus, $c\notin\{p,2p\}$.

Put
$$F=E-h=\{0,\pm p,\pm(p+c),\pm h\}.$$
Since $h=p+(p+c)$, we have
$$
\begin{aligned}
2^{\wedge}F
&=(F\setminus\{0\})\cup2^{\wedge}\{\pm p,\pm(p+c),\pm h\}\\
&=\{0,\pm p,\pm c,\pm(p+c),\pm h,\pm(h+p),\pm(h+p+c)\},
\end{aligned}
$$
$$
\begin{aligned}
3^{\wedge}F
&=2^{\wedge}(F\setminus\{0\})\cup(F\setminus\{0\})\\
&\qquad\cup\{\varepsilon_1p+\varepsilon_2(p+c)+\varepsilon_3h:
\varepsilon_1,\varepsilon_2,\varepsilon_3\in\{-1,1\}\}\\
&=2^{\wedge}F\cup\{\pm2p,\pm(2p+2c),\pm2h\}.
\end{aligned}
$$
Since
$$
\begin{gathered}
0<p<p+c<h<h+p<h+p+c<2h,\\
0<c<p+c, \ \ p<2p<h<2p+2c<h+p+c,
\end{gathered}
$$
the only possible coincidences among
$p,c,p+c,h,h+p,h+p+c,2p,2p+2c,2h$ are
$$c=p, \ \ c=2p, \ \ p+c=2p, \ \ h+p=2p+2c.$$
Each gives $c=p$ or $c=2p$, a contradiction. Hence,
$$|3^{\wedge}F|=19, \ \ |2^{\wedge}F\cap\mathbb Z_{>0}|=6, \ \
\max3^{\wedge}F=2h,$$
and
$$3^{\wedge}F\cap[h+1,2h]=h+\{p,c,p+c,h\}.$$

For every positive integer $t$, put
$$\nu(t)=\left|\bigl((2^{\wedge}F+h+t)\setminus3^{\wedge}F\bigr)
\cap\mathbb Z_{>0}\right|.$$
Since
$$2h+t+\{0,p,p+c\}
\subseteq\bigl((2^{\wedge}F+h+t)\setminus3^{\wedge}F\bigr)
\cap\mathbb Z_{>0}$$
and
$$0<h+t\leqslant h+t+p+c<2h+t,$$
we have
$$\nu(t)\geqslant3+
\left|(t+\{0,p,c,p+c\})\setminus\{p,c,p+c,h\}\right|.$$
If $t+p+c\in\{p,c,p+c,h\}$, then
$$t+p+c=h, \ \ t=p, \ \ t+p=2p\notin\{p,c,p+c,h\}.$$
Thus,
$$\left|(t+\{0,p,c,p+c\})\setminus\{p,c,p+c,h\}\right|\geqslant1,
\ \ \nu(t)\geqslant4.$$

Suppose that $\nu(t)=4$. Then
$$\left|(t+\{0,p,c,p+c\})\cap\{p,c,p+c,h\}\right|=3.$$
If $t\neq p$, then $t+p+c\notin\{p,c,p+c,h\}$, thus,
$$\{t,t+p,t+c\}\subseteq\{p,c,p+c,h\}.$$
Since $t+p$ and $t+c$ are distinct and greater than $t$, we have
$$t\in\{p,c\}, \ \ t=c, \ \ 2c\in\{p,c,p+c,h\}.$$
Together with $c\notin\{p,2p\}$ and $\gcd(p,c)=1$, we have
$$p=2c, \ \ (p,c,t)=(2,1,1), \ \ h=5.$$

If $t=p$, then $h+2p\notin3^{\wedge}F$.
Since
$$3p=h+p-c\in2^{\wedge}F+h+p, \ \ 0<3p<h+2p<2h+p$$
and $\nu(p)=4$, we have $3p\in3^{\wedge}F$.
Since
$$\{3p\}\cap\{p,2p,h+p,h+p+c,2h\}=\emptyset,$$
we have
$$3p\in\{c,p+c,h,2p+2c\}.$$
The equalities $3p=p+c$ and $3p=h$ give $c=2p$ and $c=p$,
respectively, a contradiction. Thus, $c=3p$ or $p=2c$.
By $\gcd(p,c)=1$, we have
$$(p,c)\in\{(1,3),(2,1)\}, \ \ h=5, \ \ t=p\leqslant2.$$
Hence,
$$\nu(t)=4\Longrightarrow h=5\text{ and }t\leqslant2.$$

Since $X-(a+h)=F\cup\{-h-a,h+b\}$, we have
$$
\begin{aligned}
&3^{\wedge}F
\mathbin{\dot\cup}
\left(\bigl((2^{\wedge}F-h-a)\setminus3^{\wedge}F\bigr)
\cap\mathbb Z_{<0}\right)\\
&\mathbin{\dot\cup}
\left(\bigl((2^{\wedge}F+h+b)\setminus3^{\wedge}F\bigr)
\cap\mathbb Z_{>0}\right)\\
&\subseteq3^{\wedge}X-3(a+h).
\end{aligned}
$$
Since $F=-F$, we have
$$2^{\wedge}F=-2^{\wedge}F, \ \ 3^{\wedge}F=-3^{\wedge}F,$$
thus,
$$
\begin{aligned}
&\left|\bigl((2^{\wedge}F-h-a)\setminus3^{\wedge}F\bigr)
\cap\mathbb Z_{<0}\right|\\
&=\left|\bigl((2^{\wedge}F+h+a)\setminus3^{\wedge}F\bigr)
\cap\mathbb Z_{>0}\right|=\nu(a).
\end{aligned}
$$
Hence,
$$|3^{\wedge}X|\geqslant|3^{\wedge}F|+\nu(a)+\nu(b)
=19+\nu(a)+\nu(b)\geqslant27.$$

If $|3^{\wedge}X|\leqslant28$, then
$$\nu(a)+\nu(b)\leqslant9, \ \ \nu(a),\nu(b)\geqslant4, \ \
\min\{\nu(a),\nu(b)\}=4.$$
Thus,
$$h=5, \ \ \min\{a,b\}\leqslant2, \ \ \nu(a),\nu(b)\leqslant5.$$
If $t\in\{a,b\}$ and $t\geqslant5$, then $h+t\geqslant2h$ and
$$h+t+(2^{\wedge}F\cap\mathbb Z_{>0})
\subseteq\bigl((2^{\wedge}F+h+t)\setminus3^{\wedge}F\bigr)
\cap\mathbb Z_{>0}.$$
Thus,
$$5\geqslant\nu(t)\geqslant|2^{\wedge}F\cap\mathbb Z_{>0}|=6,$$
a contradiction. Hence,
$$2h=10, \ \ \max\{a,b\}\leqslant4, \ \
L=10+\min\{a,b\}+\max\{a,b\}\leqslant10+2+4=16.$$

This completes the proof of Corollary \ref{cor:seminar-hexagon-arbitrary-span-completion}.
\end{proof}

\begin{thm}
\label{thm:threefold-all-cardinalities}
Let $X$ be a primitive set of $n\geqslant8$ integers. Then
$$|3^{\wedge}X|\geqslant2n-7+\min\{l(X),3n-9\}.$$
If $X$ is critical at order three and $a_3(X)\geqslant1$, then
$$D_3(X)\geqslant0.$$
If $n\geqslant9$, $X$ is critical at order three,
$a_3(X)\geqslant1$, $D_3(X)<\frac{n}{2}$ and
$$4D_3(X)<(a_3(X)-1)(n-1),$$
then
$$e_3(X)\leqslant 2D_3(X).$$
\end{thm}

\begin{proof}
If $|3^{\wedge}X|\geqslant 5n-16$, then
$$|3^{\wedge}X|\geqslant 5n-16\geqslant 2n-7+\min\{l(X), 3n-9\}.$$
Since the other two assertions assume that $X$ is critical at order three,
we may assume that
$$|3^{\wedge}X|\leqslant 5n-17.$$
By translation, we may assume that $\min X=0$ and $\max X=L$.
Put $Y=X\setminus\{0, L\}$. Then
$$Y\subseteq[1, L-1], \ \ |Y|=n-2\geqslant 6, \ \ l(Y)\geqslant n-3.$$

For the last assertion, if $D_3(X)\geqslant L-n+1$, then
$$0\leqslant L-n+1\leqslant D_3(X)<\frac{n}{2}.$$
By Corollary \ref{cor:two-fringe-hole-bound}, we have
$$e_3(X)\leqslant2(L-n+1)\leqslant 2D_3(X).$$

If $d(Y)>1$, then by Proposition \ref{pro:one-residue-class}, we have
$$|3^{\wedge}X|\geqslant2\cdot3(n-3)-n+1=5n-17.$$
Equality requires $n=2\cdot3=6$, contrary to $n\geqslant 8$.
Thus,
$$5n-17\geqslant |3^{\wedge}X|\geqslant 5n-16,$$
a contradiction. Hence, $Y$ is primitive.

Put
$$B=\{0\}\cup\pi_L(Y), \ \ U=\pi_L(Y\cup2^{\wedge}Y\cup3^{\wedge}Y).$$
Since $X$ is primitive, we have
$$|B|=n-1\geqslant 7, \ \ \langle B\rangle=\mathbb Z_L.$$
By Lemma \ref{lem:endpoint-lifting}, with $p=2$, we have
$$
\begin{gathered}
U=2^{\wedge}B\cup3^{\wedge}B,\\
|3^{\wedge}X|-|2^{\wedge}Y|\geqslant|U|, \ \
l(3^{\wedge}X)=L+l(2^{\wedge}Y).
\end{gathered}
$$
By Theorem \ref{thm:r4-coarse-cyclic-all-orders} and Corollary \ref{lem:exact-endpoint-transfer}, we have
$$|U|\geqslant\min\{L, 2n-4\},$$
$$D_3(X)=D_2(Y)+2\bigl(|3^{\wedge}X|-|2^{\wedge}Y|\bigr)-L-n+1.$$

Suppose that $l(Y)\leqslant2n-9$. Then
$$0\leqslant l(Y)-n+3\leqslant n-6.$$
Since $|Y|=n-2\geqslant 6$, by
Lemma \ref{lem:dense-strong-endpoint-deletion} (3), applied to $Y-\min Y$, we have
$$
\begin{gathered}
|2^{\wedge}Y|
\geqslant2(n-4)+1+(l(Y)-n+3)=n+l(Y)-4,\\
D_2(Y)\geqslant0.
\end{gathered}
$$
By Lemma \ref{lem:dense-low-adjacent-projection},
with $m=2$ and $I=Y$, we have
$$
\begin{aligned}
|U|
&\geqslant\min\{L,3(n-3)-(l(Y)-n+3)\}\\
&=\min\{L,4n-l(Y)-12\}.
\end{aligned}
$$
If $|U|<L$, then $|U|\geqslant4n-l(Y)-12$, thus,
$$
\begin{aligned}
5n-17
&\geqslant|3^{\wedge}X|\geqslant|2^{\wedge}Y|+|U|\\
&\geqslant(n+l(Y)-4)+(4n-l(Y)-12)=5n-16,
\end{aligned}
$$
a contradiction. Hence, $U=\mathbb Z_L$.
Since $l(Y)\geqslant n-3$, we have
$$
\begin{aligned}
|3^{\wedge}X|
&\geqslant|2^{\wedge}Y|+L\geqslant n+l(Y)-4+L\\
&\geqslant 2n-7+L\geqslant 2n-7+\min\{L, 3n-9\}.
\end{aligned}
$$
Moreover,
$$
\begin{aligned}
D_3(X)
&=D_2(Y)+2\bigl(|3^{\wedge}X|-|2^{\wedge}Y|\bigr)-L-n+1\\
&\geqslant D_2(Y)+L-n+1\\
&\geqslant L-n+1\geqslant0.
\end{aligned}
$$
This proves all three assertions in this case.

Suppose that $l(Y)\geqslant2n-8$.
By Lemma \ref{lem:six-small-pair-inverse},
Corollary \ref{cor:seminar-seven-pair-elementary-floor} and
Theorem \ref{thm:freiman-lev-twofold}, applied to $Y-\min Y$, we have
$$
\begin{aligned}
|2^{\wedge}Y|&\geqslant3n-13,
&&n=8\text{ or }n\geqslant10,\\
|2^{\wedge}Y|&\geqslant13,
&&n=9.
\end{aligned}
$$
If $n=9$ and $|2^{\wedge}Y|=13$, then by
Lemma \ref{lem:seminar-seven-hexagon-from-cyclic-equality}, we have
$$Y=\min Y+\{0,p,p+c,2p+c,3p+c,3p+2c,4p+2c\},$$
where
$$
\begin{gathered}
p,c,\min Y,L-l(Y)-\min Y>0, \ \ \gcd(p,c)=1,\\
l(Y)=4p+2c\geqslant10.
\end{gathered}
$$
Since $|3^{\wedge}X|\leqslant5n-17=28$, by
Corollary \ref{cor:seminar-hexagon-arbitrary-span-completion},
with endpoint gaps $\min Y$ and $L-l(Y)-\min Y$, we have
$$|3^{\wedge}X|\geqslant27, \ \ l(Y)=10, \ \ L\leqslant16.$$
Since $l(2^{\wedge}Y)\leqslant2l(Y)-2=18$, we have
$$|3^{\wedge}X|\geqslant27\geqslant11+L
=2n-7+\min\{L,3n-9\},$$
$$
\begin{aligned}
D_3(X)
&=2|3^{\wedge}X|-L-l(2^{\wedge}Y)-20\\
&\geqslant54-16-18-20=0.
\end{aligned}
$$
Moreover,
$$
\begin{gathered}
a_3(X)=28-|3^{\wedge}X|\leqslant1,\\
(a_3(X)-1)(n-1)\leqslant0\leqslant4D_3(X).
\end{gathered}
$$
Thus, the hypotheses of the last assertion cannot hold.

In all remaining cases, we have
$$|2^{\wedge}Y|\geqslant3n-13
=2\cdot2(n-4)-(n-2)+1.$$
If $a_3(X)\geqslant1$, then by
Lemma \ref{lem:inner-cardinality-barrier},
with $(p,N)=(2,n-2)$, we have
$$D_3(X)\geqslant0.$$
Under the hypotheses of the last assertion, we have
$$n-2\geqslant7, \ \ 4D_3(X)<(a_3(X)-1)(n-1).$$
By the same lemma, we have
$$D_3(X)\geqslant L-n+1.$$

It remains to prove the cardinality assertion.
Suppose that $L\geqslant2n-2$.
Since $|U|\geqslant2n-4$, we have
$$
\begin{aligned}
5n-17
&\geqslant|3^{\wedge}X|\geqslant|2^{\wedge}Y|+|U|\\
&\geqslant(3n-13)+(2n-4)=5n-17.
\end{aligned}
$$
Thus,
$$
\begin{gathered}
|3^{\wedge}X|=|2^{\wedge}Y|+|U|=5n-17,\\
|2^{\wedge}Y|=3n-13, \ \ |U|=2n-4.
\end{gathered}
$$
Since
$$|2^{\wedge}Y|=3n-13<3n-12
=2\cdot2(n-4)-(n-2)+2,$$
by Theorem \ref{thm:strict-critical-cyclic-surplus},
with $(p,N)=(2,n-2)$ and $L\geqslant2n-2$, we have
$$|U|\geqslant2n-3,$$
a contradiction. Hence, $L\leqslant2n-3$.
Since $n\geqslant8$ and $L\leqslant2n-3\leqslant3n-9$, we have
$$
\begin{aligned}
|3^{\wedge}X|
&\geqslant3n-13+\min\{L,2n-4\}\\
&\geqslant3n-14+L\geqslant2n-7+L\\
&=2n-7+\min\{L,3n-9\}.
\end{aligned}
$$

This completes the proof of Theorem \ref{thm:threefold-all-cardinalities}.
\end{proof}

\section{Proof of main theorem}
\begin{pro}\label{pro:positive-slack-span-transition}
Let $X$ be a primitive set of $n\geqslant 7$ integers and $3\leqslant r\leqslant \frac{n}{2}$.
If $|r^{\wedge}X|\leqslant 2r(n-r)-n$, then
$$l(r^{\wedge}X)\leqslant 2|r^{\wedge}X|-r(n-r)-2.$$
\end{pro}

\begin{proof}
By the definitions of $a_r(X)$ and $D_r(X)$, the proposition is equivalent to
$$a_r(X)\geqslant1\Longrightarrow D_r(X)\geqslant0.$$
We prove this assertion for all $n\geqslant\max\{7,2r\}$ by induction on $r$,
simultaneously with the following assertion:
if $X$ is primitive, $n=|X|\geqslant\max\{9,2r\}$ and
$$a_r(X)\geqslant1, \ \ D_r(X)<\frac{n-1}{2}, \ \
4D_r(X)<(a_r(X)-1)(n-1),$$
then
$$e_r(X)\leqslant2D_r(X).$$

For $r=3$ and $a_3(X)\geqslant1$, by Corollary
\ref{cor:seven-below-eighteen} when $n=7$ and Theorem
\ref{thm:threefold-all-cardinalities} when $n\geqslant8$, we have
$$D_3(X)\geqslant0.$$
Since $\frac{n-1}{2}<\frac{n}{2}$, the additional assertion follows from Theorem \ref{thm:threefold-all-cardinalities}.

Suppose that $r\geqslant4$ and both assertions hold at order $r-1$.
Assume that $a_r(X)\geqslant1$.
If $n<9$, then
$$(r,n)=(4,8), \ \
|4^{\wedge}X|\leqslant2\cdot4(8-4)-8=24.$$
By Lemma \ref{lem:seminar-eight-weak-defect} applied to $X-\min X$,
we have
$$D_4(X)\geqslant0.$$

Suppose that $n\geqslant9$.
If $D_r(X)<0$, then
$$D_r(X)<\frac{n-1}{2}, \ \
4D_r(X)<0\leqslant(a_r(X)-1)(n-1).$$
The additional assertion would give
$$0\leqslant e_r(X)\leqslant2D_r(X)<0,$$
a contradiction. Thus, it suffices to prove the additional assertion.

Suppose that the additional assertion fails.
By translation, we may assume that $\min X=0$ and $\max X=L$.
Put $Y=X\setminus\{0,L\}$. Then
$$Y\subseteq[1,L-1], \ \
|Y|=n-2\geqslant\max\{7,2(r-1)\}$$
and
$$
\begin{gathered}
a_r(X)\geqslant1, \ \ D_r(X)<\frac{n-1}{2},\\
4D_r(X)<(a_r(X)-1)(n-1), \ \ e_r(X)>2D_r(X).
\end{gathered}
$$

If $D_r(X)\geqslant L-n+1$, then
$$0\leqslant L-n+1\leqslant D_r(X)<\frac{n-1}{2}<\frac n2.$$
By Corollary \ref{cor:two-fringe-hole-bound}, we have
$$e_r(X)\leqslant2(L-n+1)\leqslant2D_r(X),$$
a contradiction. Hence,
$$D_r(X)<L-n+1.$$

If $d(Y)>1$, then by Proposition \ref{pro:one-residue-class}, we have
$$a_r(X)=2r(n-r)-n+1-|r^{\wedge}X|\leqslant0,$$
a contradiction. Thus, $Y$ is primitive.

If $a_{r-1}(Y)\leqslant0$, then
$$|(r-1)^{\wedge}Y|
\geqslant2(r-1)(n-r-1)-n+3.$$
Since $n-2\geqslant7$ and
$$4D_r(X)<(a_r(X)-1)(n-1),$$
by Lemma \ref{lem:inner-cardinality-barrier},
with $(p,N)=(r-1,n-2)$, we have
$$D_r(X)\geqslant L-n+1,$$
a contradiction. Hence,
$$a_{r-1}(Y)\geqslant1.$$
By the induction hypothesis at order $r-1$, we have
$$D_{r-1}(Y)\geqslant0.$$

Put
$$
\begin{aligned}
v&=\left|\pi_L\bigl((r-2)^{\wedge}Y
\cup(r-1)^{\wedge}Y\cup r^{\wedge}Y\bigr)\right|,\\
\rho&=|r^{\wedge}X|-|(r-1)^{\wedge}Y|-v.
\end{aligned}
$$
By Lemma \ref{lem:endpoint-lifting}, we have $\rho\geqslant0$.
If $v=L$, then by Corollary \ref{lem:exact-endpoint-transfer}, we have
$$D_r(X)=D_{r-1}(Y)+L-n+1+2\rho\geqslant L-n+1,$$
a contradiction. Thus, $v<L$.

By Lemma \ref{lem:stratified-incomplete-projection},
with $(p,N)=(r-1,n-2)$, we have
$$a_{r-1}(Y)\geqslant a_r(X)+\rho+1\geqslant2, \ \
D_{r-1}(Y)<\frac{n-3}{2}$$
and
$$
\begin{aligned}
e_{r-1}(Y)
&\geqslant n-3+D_{r-1}(Y)+a_{r-1}(Y)-a_r(X)-D_r(X)\\
&>2D_{r-1}(Y).
\end{aligned}
$$

If $n\leqslant10$, then
$$(r,n)\in\{(4,9),(4,10),(5,10)\}, \ \
(r-1)(n-r-1)\leqslant2n-4\leqslant16.$$
Since
$$e_{r-1}(Y)+D_{r-1}(Y)
=(r-1)(n-r-1)-n+2-a_{r-1}(Y),$$
we have
$$
\begin{aligned}
D_r(X)
&\geqslant2D_{r-1}(Y)+2n-4-(r-1)(n-r-1)\\
&\qquad+2a_{r-1}(Y)-a_r(X)-1\\
&\geqslant a_r(X)+1>0.
\end{aligned}
$$
Thus,
$$
\begin{aligned}
0&<e_r(X)-2D_r(X)\\
&=(r-1)(n-r-1)-1-a_r(X)-3D_r(X)\\
&\leqslant(r-1)(n-r-1)-4a_r(X)-4.
\end{aligned}
$$
It follows that
$$4a_r(X)+4<(r-1)(n-r-1)\leqslant2n-4\leqslant16,$$
thus,
$$1\leqslant a_r(X)\leqslant2.$$
Since $n\leqslant10$, we have
$$4D_r(X)\geqslant4(a_r(X)+1)>(a_r(X)-1)(n-1),$$
a contradiction. Hence, $n\geqslant11$.

Suppose that
$$4D_{r-1}(Y)\geqslant(a_{r-1}(Y)-1)(n-3).$$
Since $D_{r-1}(Y)<(n-3)/2$ and $a_{r-1}(Y)\geqslant2$, we have
$$(a_{r-1}(Y)-1)(n-3)
\leqslant4D_{r-1}(Y)<2(n-3),$$
thus,
$$a_{r-1}(Y)=2.$$
Together with
$$a_{r-1}(Y)\geqslant a_r(X)+\rho+1, \ \
a_r(X)\geqslant1, \ \ \rho\geqslant0,$$
we have
$$a_r(X)=1, \ \ \rho=0, \ \
4D_r(X)<(a_r(X)-1)(n-1)=0.$$
Since $D_{r-1}(Y)\geqslant0$ and $D_r(X)\in\mathbb Z$,
by Lemma \ref{lem:stratified-incomplete-projection}, we have
$$2L-3v=2(D_{r-1}(Y)-D_r(X))-1\geqslant1.$$
By \eqref{eq:incomplete-projection-capacity}, we have
$$4D_{r-1}(Y)\leqslant n+2D_r(X)-3
\leqslant n-5<n-3,$$
contrary to
$$4D_{r-1}(Y)\geqslant(a_{r-1}(Y)-1)(n-3)=n-3.$$
Thus,
$$4D_{r-1}(Y)<(a_{r-1}(Y)-1)(n-3).$$

Since $Y$ is primitive and
$$
\begin{gathered}
|Y|=n-2\geqslant\max\{9,2(r-1)\}, \ \ a_{r-1}(Y)\geqslant2,\\
D_{r-1}(Y)<\frac{n-3}{2}, \ \
4D_{r-1}(Y)<(a_{r-1}(Y)-1)(n-3),
\end{gathered}
$$
by the additional induction assertion at order $r-1$, we have
$$e_{r-1}(Y)\leqslant2D_{r-1}(Y),$$
contrary to $e_{r-1}(Y)>2D_{r-1}(Y)$.

This completes the proof of Proposition \ref{pro:positive-slack-span-transition}.
\end{proof}

\begin{pro}\label{pro:span-to-cardinality}
Let $p\geqslant 3$ and $N\geqslant\max\{7, 2p\}$.
Let $A\subseteq[0, L]$ be a primitive set of $N+2$ integers such that $0, L\in A$.
Then
$$
\begin{aligned}
|(p+1)^{\wedge}A|
&=|(N-p+1)^{\wedge}A|\\
&\geqslant p(N-p)+1+\min\{L,(p+1)(N-p+1)\}.
\end{aligned}
$$
\end{pro}

\begin{proof}
Put $J=A\setminus\{0,L\}$. Then
$$J\subseteq[1,L-1], \ \ |J|=N.$$
If $d(J)>1$, then by Proposition \ref{pro:one-residue-class},
with $k=N+2$ and $h=p+1$, we have
$$
\begin{aligned}
|(p+1)^{\wedge}A|
&\geqslant(p+1)(N-p+1)-N
+\min\{L,(p+1)(N-p+1)\}\\
&=p(N-p)+1+\min\{L,(p+1)(N-p+1)\}.
\end{aligned}
$$

Suppose that $d(J)=1$ and
$$|(p+1)^{\wedge}A|
<p(N-p)+1+\min\{L,(p+1)(N-p+1)\}.$$
By Theorem \ref{thm:nathanson-restricted-minimum}, we have
$$e_p(J)+D_p(J)=|p^{\wedge}J|-p(N-p)-1\geqslant0.$$
Put
$$
\begin{aligned}
B&=\{0\}\cup\pi_L(J),\\
U&=\pi_L\bigl((p-1)^{\wedge}J\cup p^{\wedge}J
\cup(p+1)^{\wedge}J\bigr).
\end{aligned}
$$
Since $A$ is primitive and $1\leqslant p\leqslant N-1$,
by Lemma \ref{lem:endpoint-lifting}, we have
$$|B|=N+1, \ \ \langle B\rangle=\mathbb Z_L, \ \
U=p^{\wedge}B\cup(p+1)^{\wedge}B$$
and
$$
\begin{aligned}
|(p+1)^{\wedge}A|
&\geqslant|p^{\wedge}J|+|U|\\
&=p(N-p)+1+e_p(J)+D_p(J)+|U|.
\end{aligned}
$$
Thus,
$$|U|+e_p(J)+D_p(J)<\min\{L,(p+1)(N-p+1)\}.$$

If $e_p(J)+D_p(J)<\frac N2$, then by
Theorem \ref{thm:r3-uniform-dense-small-excess}, we have
$$|U|+e_p(J)+D_p(J)\geqslant\min\{L,(p+1)(N-p+1)\},$$
a contradiction. Hence,
$$e_p(J)+D_p(J)\geqslant\frac N2.$$

Since $3\leqslant p\leqslant\left\lfloor\frac N2\right\rfloor$,
by Theorem \ref{thm:r4-coarse-cyclic-all-orders}, we have
$$|U|\geqslant\min\{L,2N\}.$$
Together with
$$|U|+e_p(J)+D_p(J)<L, \ \ e_p(J)+D_p(J)\geqslant0,$$
we have
$$|U|<L, \ \ |U|\geqslant2N.$$
Thus,
$$L>|U|+e_p(J)+D_p(J)\geqslant2N+\frac N2\geqslant2N+2$$
and
$$
\begin{aligned}
e_p(J)+D_p(J)
&<(p+1)(N-p+1)-2N\\
&=p(N-p)-N+1.
\end{aligned}
$$
Since $e_p(J)+D_p(J)\in\mathbb Z$, we have
$$
\begin{aligned}
|p^{\wedge}J|
&=p(N-p)+1+e_p(J)+D_p(J)\\
&\leqslant2p(N-p)-N+1.
\end{aligned}
$$
By Theorem \ref{thm:strict-critical-cyclic-surplus}, we have
$$|U|\geqslant2N+1.$$
Since $|U|+e_p(J)+D_p(J)<(p+1)(N-p+1)$, we have
$$
\begin{aligned}
e_p(J)+D_p(J)
&<(p+1)(N-p+1)-|U|\\
&\leqslant p(N-p)-N,
\end{aligned}
$$
thus,
$$a_p(J)=p(N-p)-N-e_p(J)-D_p(J)\geqslant1.$$
Hence, $J$ is critical.
By Proposition \ref{pro:positive-slack-span-transition}, we have
$$D_p(J)\geqslant0.$$

Since $0,L\in A$ and $J\subseteq[1,L-1]$, we have
$$
\begin{aligned}
p^{\wedge}J\cup(L+p^{\wedge}J)
&\subseteq(p+1)^{\wedge}A,\\
p^{\wedge}J\cap(L+p^{\wedge}J)
&\subseteq[\min p^{\wedge}J+L,\max p^{\wedge}J].
\end{aligned}
$$
Since $D_p(J)\geqslant0$ and $2(e_p(J)+D_p(J))\geqslant N$, we have
$$
\begin{aligned}
|(p+1)^{\wedge}A|
&\geqslant\left|p^{\wedge}J\cup(L+p^{\wedge}J)\right|\\
&=2|p^{\wedge}J|
-\left|p^{\wedge}J\cap(L+p^{\wedge}J)\right|\\
&\geqslant2|p^{\wedge}J|
-\max\{0,l(p^{\wedge}J)-L+1\}\\
&=p(N-p)+1\\
&\qquad+\min\{L+D_p(J),p(N-p)+1+2(e_p(J)+D_p(J))\}\\
&\geqslant p(N-p)+1+\min\{L,(p+1)(N-p+1)\},
\end{aligned}
$$
a contradiction.

Since $|A|=N+2$, we have
$$(p+1)^{\wedge}A=\sum_{a\in A}a-(N-p+1)^{\wedge}A,$$
thus,
$$|(p+1)^{\wedge}A|=|(N-p+1)^{\wedge}A|.$$

This completes the proof of Proposition \ref{pro:span-to-cardinality}.
\end{proof}

\begin{proof}[Proof of Theorem \ref{thm:main}]
Since $0\in A$ and $\gcd(A)=1$, $A$ is primitive and $l(A)=L$. For $1\leqslant h\leqslant k-1$, we have
$$h^{\wedge}A=\sum_{a\in A}a-(k-h)^{\wedge}A, $$
thus,
$$|h^{\wedge}A|=|(k-h)^{\wedge}A|, $$
it suffices to consider
$$1\leqslant h\leqslant \frac{k}{2}. $$

For $h=1$, by $L\geqslant k-1$ we have
$$|1^{\wedge}A|=k=1+\min\{L, k-1\}.$$
For $h=2$, by Theorem \ref{thm:freiman-lev-twofold} we have
$$|2^{\wedge}A|\geqslant k-2+\min\{L, 2k-5\}.$$
For $h=3$, by Theorem \ref{thm:threefold-all-cardinalities} we have
$$|3^{\wedge}A|\geqslant 2k-7+\min\{L, 3k-9\}.$$

Let $4\leqslant h\leqslant \frac{k}{2}$ and put $p=h-1$, $N=k-2$. Then
$$p\geqslant 3, \ \ N\geqslant\max\{7, 2p\}, \ \ |A|=N+2. $$
By Proposition \ref{pro:span-to-cardinality} we have
$$
\begin{aligned}
|h^{\wedge}A|
&\geqslant p(N-p)+1+\min\{L, (p+1)(N-p+1)\}\\
&=h(k-h)-k+2+\min\{L, h(k-h)\}.
\end{aligned}
$$

This completes the proof of Theorem \ref{thm:main}.
\end{proof}


\section*{References}

\begin{thebibliography}{99}


\bibitem{DSH1994} J.A. Dias da Silva and Y.O. Hamidoune, {\it Cyclic spaces for Grassmann derivatives and additive theory}, Bull. London Math. Soc. 26 (1994), 140--146.

\bibitem{Freiman} G.A. Freiman, {\it The addition of finite sets. I}, Izv. Vys\v{s}. U\v{c}ebn. Zaved. Matematika (1959), no. 6 (13), 202--213.

\bibitem{GaoHamidoune} W. Gao and Y.O. Hamidoune, {\it On additive bases}, Acta Arith. 88 (1999), 233--237.

\bibitem{GGH} B. Girard, S. Griffiths and Y.O. Hamidoune, {\it $k$-sums in abelian groups}, Combin. Probab. Comput. 21 (2012), 582--596.

\bibitem{KempermanStructure} D.J. Grynkiewicz, {\it A step beyond Kemperman's structure theorem}, Mathematika 55 (2009), 67--114.

\bibitem{Kneser} M. Kneser, {\it Absch\"atzung der asymptotischen Dichte von Summenmengen}, Math. Z. 58 (1953), 459--484.

\bibitem{Lev2} V.F. Lev, {\it Restricted set addition in groups, I. The classical setting}, J. London Math. Soc. 62 (2000), 27--40.

\bibitem{Lev} V.F. Lev, {\it Structure theorem for multiple addition and the Frobenius problem}, J. Number Theory 58 (1996), 79--88.

\bibitem{LevThree} V.F. Lev, {\it Three-fold restricted set addition in groups}, European J. Combin. 23 (2002), 613--617.

\bibitem{TW} M. Tang and W.H. Wang, {\it Some remarks on sumsets and restricted sumsets}, Bull. Korean Math. Soc. 56 (2019), 667--673.

\bibitem{Nathanson1996} M.B. Nathanson, {\it Additive number theory: Inverse problems and the geometry of sumsets}, Graduate Texts in Math. 165, Springer-Verlag, New York, 1996.

\bibitem{Nathanson1} M.B. Nathanson, {\it Inverse theorems for subset sums}, Trans. Amer. Math. Soc. 347 (1995), 1409--1418.

\bibitem{Olson} J.E. Olson, {\it Sums of sets of group elements}, Acta Arith. 28 (1975), 147--156.

\bibitem{Wang2026} Y.J. Wang and M. Tang, {\it A proof of the Freiman--Lev conjecture}, preprint, \mbox{arXiv:2608.04760v2} (25 August 2026).

\bibitem{Wang} Y.J. Wang and M. Tang, {\it On the Freiman--Lev conjecture}, Contrib. Discrete Math. 20(1) (2025), 42--59.

\bibitem{WangRestricted} Y.J. Wang and M. Tang, {\it Freiman--Lev conjecture (in Chinese)}, Sci. Sin. Math. 56 (2026) (English version: {\it Restricted sumsets in $\mathbb Z$}, arXiv: 2402.01471v3).

\end{thebibliography}
\end{document}